\documentclass[11pt,twoside]{preprint}
\usepackage[full]{textcomp}
\usepackage[osf]{newtxtext}
\usepackage{amsmath,amsthm,amssymb,enumerate}
\usepackage{bbm}
\usepackage{hyperref}
\usepackage{mathrsfs}
\usepackage{breakurl}
\usepackage{xcolor}
\usepackage{mhequ}
\usepackage{comment}
\usepackage{authblk}
\usepackage{microtype}
\usepackage[a4paper,innermargin=1.2in,
outermargin=1.2in,
bottom=1.5in,marginparwidth=1in,marginparsep=3mm]{geometry}
\usepackage{lipsum}

\newcommand\blfootnotetext[1]{
  \begingroup
  \renewcommand\thefootnote{}\footnotetext{#1}
  \endgroup
}

\newcommand\D{\partial}
\newcommand{\bT}{\mathbb{T}}
\newcommand{\diff}{\mathrm{d}}

\newtheorem{theorem}{Theorem}[section]
\newtheorem{lemma}[theorem]{Lemma}
\newtheorem{proposition}[theorem]{Proposition}
\newtheorem{corollary}[theorem]{Corollary}

\theoremstyle{definition}

\newtheorem{definition}[theorem]{Definition}

\theoremstyle{remark}
\newtheorem{remark}[theorem]{Remark}

\def\inner#1{\left( #1 \right)}
\def\norm#1{\left\lVert #1 \right\rVert}
\def\abs#1{\left\lvert #1 \right\rvert}

\allowdisplaybreaks

\newcommand\Z{\mathbb{Z}}
\newcommand{\N}{\mathbb{N}}

\def\eps{\varepsilon}
\def\les{\lesssim}
\newcommand{\R}{\mathbb{R}}
\newcommand{\E}{\mathbb{E}}
\newcommand{\F}{\mathbb{F}}
\def\cA{\mathcal{A}}
\def\cF{\mathcal{F}}
\def\cG{\mathcal{G}}
\def\cY{\mathcal{Y}}
\def\cZ{\mathcal{Z}}

\renewcommand{\P}{\mathbb{P}}
\newcommand{\ind}{\mathbbm{1}}
\newcommand{\cB}{\mathcal{B}}
\newcommand{\A}{\mathcal{A}}

\newcommand{\fp}{{p}}
\newcommand{\fpp}{{p'}}
\newcommand{\intort}{{\int_{\bT}}}
\newcommand{\Ltwo}{{L^2(\bT)}}
\newcommand{\Linfty}{{{L^\infty(\bT)}}}

\newcommand{\messsup}[1]{{\mathop{#1\mbox{\rm{-ess-sup}}}}}

\newcommand{\Aunull}{{\A(u^{(0)})}}
\newcommand{\Atunull}{{\A(\tilde u^{(0)})}}

\newcommand{\Feps}{{F_\eps}}
\newcommand{\Fepsp}{{F'_\eps}}
\newcommand{\Fepsq}{{F^2_\eps}}
\newcommand{\sk}{{\sigma_k}}
\newcommand{\inttime}{{\int_{t_1}^{t_2}}}
\newcommand{\sumkz}{{\sum_{k\in\Z}}}

\def\supp{\mathop{\mathrm{supp}}}

\begin{document}
\title{Non-negative Martingale Solutions to the Stochastic Thin-Film Equation with Nonlinear Gradient Noise}

\blfootnotetext{\emph{Key words:} Stochastic partial differential equation, Degenerate-parabolic equation, Fourth-order equation, Thin fluid film, Thermal fluctuations, Martingale solution, Weak solution}

\blfootnotetext{\emph{2020 Mathematics Subject Classification:} 60H15, 35R60, 76A20, 35K65, 35R37, 35K35, 35K55, 35D30, 76D08}

\author{Konstantinos Dareiotis\thanks{School of Mathematics, University of Leeds, Leeds, LS2 9JT (K.Dareiotis@leeds.ac.uk)}, Benjamin Gess\thanks{Max Planck Institute for Mathematics in the Sciences, Inselstr.~22, 04103 Leipzig, Germany and Faculty of Mathematics, Bielefeld University, Universit\"atsstr.~25, 33615 Bielefeld, Germany (b.gess@mis.mpg.de)}, Manuel V.~Gnann\thanks{Delft Institute of Applied Mathematics, Faculty of Electrical Engineering, Mathematics and Computer Science, Delft University of Technology, Van Mourik Broekmanweg 6, 2628 XE Delft, Netherlands (M.V.Gnann@tudelft.nl)}, and G\"unther Gr\"un\thanks{Department of Mathematics, University of Erlangen-Nuremberg, Cauerstrasse 11,  91058 Erlangen, Germany (gruen@math.fau.de)}}

\maketitle

\begin{abstract}
We prove the existence of nonnegative martingale solutions to a class of stochastic degenerate-parabolic fourth-order PDEs arising in surface-tension driven thin-film flow influenced by thermal noise. The construction applies to a range of mobilites including the cubic one which occurs under the assumption of a no-slip condition at the liquid-solid interface. Since their introduction more than 15 years ago, by Davidovitch, Moro, and Stone and by Gr\"un, Mecke, and Rauscher, the existence of solutions to stochastic thin-film equations for cubic mobilities has been an open problem, even in the case of sufficiently regular noise. Our proof of global-in-time solutions relies on a careful combination of entropy and energy estimates in conjunction with a tailor-made approximation procedure to control the formation of shocks caused by the nonlinear stochastic scalar conservation law structure of the noise.
\end{abstract}

\setcounter{tocdepth}{2}
\tableofcontents
%

\section{Introduction\label{sec:intro}}
In this work, we consider the stochastic thin-film equation
\begin{equation}\label{stfe-general}
\diff u = - \D_x \left(M(u) \, \D_x^3 u\right) \diff t + \D_x \left(\sqrt{M(u)} \circ \diff W\right) \quad \mbox{in $Q_T$},
\end{equation}
where $u = u(t,x)$ denotes the height of a thin viscous film depending on the independent variables time $t \in [0,T]$, where $T \in (0,\infty)$ is fixed, and lateral position $x \in \bT$, where $\bT$ is the one-dimensional torus of length $L := \abs{\bT}$, and $Q_T := [0,T] \times \bT$. Equation~\eqref{stfe-general} describes the spreading of viscous thin films driven by capillary forces (acting at the liquid-air interface) and thermal noise and decelerated by friction (in the bulk or at the liquid-solid interface). The function $M \colon \R \to [0,\infty)$ is called mobility and the following results apply to the choice $M(r) = \abs{r}^n$ for $r \in \R$, where $n \in \left[\frac 8 3, 4\right)$. In particular, this covers the physically relevant case of a cubic mobility, that is, $n=3$, modelling no slip at the liquid-solid interface in the underlying stochastic Navier-Stokes equations of which \eqref{stfe-general} is an approximation. The symbol $W$ denotes a Wiener process in the Hilbert space $H^2(\bT)$. 

Since its introduction over 15 years ago,  by Davidovitch, Moro, and Stone in \cite{DMS2005}, and by the fourth author, Mecke, and Rauscher in \cite{GMR2006}, the existence of solutions to stochastic thin-film equations for cubic mobilities has been an open problem, even in the case of sufficiently regular noise $W$ in \eqref{stfe-general}. The solution of this problem is the main result of this work.

We refer to \cite{DeGennes1985,ODB1997,BEIMR2009} for details on the physical derivation by means of a lubrication approximation and on the relevance of \eqref{stfe-general} in the deterministic case, where $W = 0$ in $[0,T] \times \bT$. Stochastic versions of the thin-film equation have been proposed independently in \cite{DMS2005} and \cite{GMR2006}. The former paper is concerned with the question of how thermal fluctuations enhance the spreading of purely surface-tension driven flow. On the contrary, the paper \cite{GMR2006} considers the effect of noise on the stability of liquid films and time-scales of the de-wetting process. Therefore, the energy considered in \cite{GMR2006} differs from that one of \cite{DMS2005} by an additional effective interface potential -- giving rise to a so called \emph{conjoining-disjoining pressure} in the equation. We emphasize that the structure of the noise term in \eqref{stfe-general} is common to \cite{GMR2006} and \cite{DMS2005}. We further refer to \cite{DORKP2019} for a more recent derivation of the model including the discussion of detailed-balance conditions.

A first existence result of martingale solutions to stochastic thin-film equations has been obtained in \cite{FischerGruen2018} by Fischer and the fourth author of this paper, in the setting of quadratic mobility  $M(r) = r^2$, additional conjoining-disjoining pressure, and It\^o noise.
We also mention the paper \cite{Cornalba2018} by Cornalba who introduced additional nonlocal source terms and in this way obtained results for more general mobilities. In \cite{GessGnann2020}, the second and the third author of this paper have studied \eqref{stfe-general} with Stratonovich noise and quadratic mobility $M(r) = r^2$ without conjoining-disjoining pressure. It turns out that non-negative martingale solutions exist that allow for touch down of solutions with complete-wetting boundary conditions. The case of quadratic mobility is special and simpler since in this case the stochastic part in \eqref{stfe-general} becomes linear. This allows to separately treat the deterministic and stochastic parts in \eqref{stfe-general}, a fact crucial to the approach in \cite{GessGnann2020}, and which fails in the case of non-quadratic mobility. 

In this paper, we study the existence of weak (or martingale) solutions to \eqref{stfe-general} in the situation in which the gradient-noise term $\D_x \left(\sqrt{M(u)} \circ \diff W\right)$ is nonlinear in the film height $u$, in particular covering the situation $M(r) = \abs{r}^3$. This includes precisely the situation studied in \cite{DMS2005} in the complete-wetting regime.

The analysis of the present work is based on a combination of estimates of the surface (excess) energy $\frac 1 2 \int_{\bT} (\D_x u)^2 \, \diff x = \frac 1 2 \norm{\D_x u}_{L^2(\bT)}^2$ and the (mathematical) entropy $\int_{\bT} G_0(u) \, \diff x$, where
\begin{equation}\label{eq:entropy-eps-0}
G_0(r) = \begin{cases} \frac{r^{2-n}}{(2-n) (1-n)} & \mbox{ for } r > 0, \\ \infty & \mbox{ for } r \le 0. \end{cases}
\end{equation}
The main difficulty comes from the fact that - in contrast to the case of quadratic mobility and Stratonovich noise  - the energy estimate cannot be closed on its own. This is caused by the nonlinear, stochastic conservation law structure of the noise in  \eqref{stfe-general}. Indeed, this nonlinear structure may lead to the occurrence of shocks and, hence, to the  blow up of the energy $\frac 1 2 \norm{\D_x u}_{L^2(\bT)}^2$. In the light of this, the task becomes to understand if the thin-film operator, that is, the deterministic part in \eqref{stfe-general}, has a sufficiently strong regularity improving effect to compensate the possible energy blow up caused by the stochastic perturbation. Since the thin-film operator degenerates when $u \approx 0$, this requires a control on the smallness of $u$. Such a control is obtained by the entropy estimate, which explains its importance in the case of non-quadratic mobility. Indeed, in the present work we prove that a blow up of the energy can be ruled out by means of a combination of energy and entropy estimates.
Once this importance of the entropy estimate for the construction of weak solutions to \eqref{stfe-general} is understood, the next task is to find approximations to \eqref{stfe-general} which allow for uniform (energy) estimates. In light of the previous discussion,  these approximations are chosen in a careful  way,  compatible with  both energy and entropy estimates. 

We next give a brief account on the literature for the deterministic thin-film equation: A theory of existence of weak solutions for the deterministic thin-film equation has been developed in \cite{BF1990,BBDP1995,BP1996} and \cite{Otto1998,BGK2005,Mellet2015} for zero and nonzero contact angles at the intersection of the liquid-gas and liquid-solid interfaces, respectively, while the higher-dimensional version of \eqref{stfe-general} with $W = 0$ in $[0,T] \times \bT$ and zero contact angles has been the subject of \cite{PGG1998,Gruen2004}. For these solutions, a number of quantitative results has been obtained -- including optimal estimates on spreading rates of free boundaries, i.e.\ the triple lines separating liquid, gas, and solid, see \cite{HulshofShishkov98, BDPGG1998, Gruen2002, SupportPropagationThinFilm}, optimal conditions on the occurrence of waiting time phenomena \cite{WaitingTime}, as well as scaling laws for the size of waiting times \cite{LowerBounds, UpperBoundsThinFilmWeakSlippage}. We also  refer to \cite{NumericsGruen} for an existence result based on numerical analysis.

A corresponding theory of classical solutions, giving the existence and uniqueness for initial data close to generic solutions or short times, has been developed in \cite{GKO2008,GK2010,GGKO2014,Gnann2015,Gnann2016,GIM2019} for zero contact angles and in \cite{Knuepfer2011,KM2013,KM2015,Knuepfer2015,Esselborn2016} for nonzero contact angles in one space dimension, while the higher-dimensional version has been the subject of \cite{John2015,Seis2018,GP2018} and \cite{Degtyarev2017} for zero and nonzero contact angles, respectively.

The paper is structured as follows: In \S\ref{sec:main}, we introduce the necessary mathematical framework and state our main result on existence of martingale solutions. In \S\ref{sec:galerkin} we introduce a suitable approximation of \eqref{stfe-general} using a Galerkin scheme, a regularization of the mobility $M$ controlled by a small parameter $\eps$, and a cut-off in $\norm{u}_{L^\infty(\bT)}$. The Galerkin scheme only makes use of the energy inequality, which is valid also in the infinite-dimensional setting but ceases to hold as $\eps \searrow 0$. In \S\ref{sec:apriori} we then derive an energy-entropy estimate which is uniform in $\eps$ and the cut-off in $\norm{u}_{L^\infty(\bT)}$ (the latter is removed at the end of this section). Finally, in \S\ref{sec:eps-limit} the limit $\eps \searrow 0$ is carried out and the existence of martingale solutions to the original problem \eqref{stfe-general} is obtained.

\section{Setting and Main Result\label{sec:main}}
\subsection{Notation}
For a set $X$ and $A \subseteq X$ we write $\ind_A \colon X \to \{0,1\}$ for the indicator function of $A$, that is,
\[
\ind_A(x) := \begin{cases} 1 & \mbox{ for } x \in A, \\ 0 & \mbox{ for } x \in X \setminus A. \end{cases}
\]

For a measurable set $D\subseteq \R^d$, where $d \in \N$, we write $\abs{D}$ for its $d$-dimensional Lebesgue measure. We write $\bT := \R / (L \Z)$ for the one-dimensional torus of length $L > 0$. For any $T \in [0,\infty)$ we write $Q_T := [0,T] \times \bT$ for the corresponding parabolic cylinder.

For ${\alpha_1, \alpha_2} \in (0,1)$ and $T > 0$, we introduce the H\"older space $C^{{\alpha_1, \alpha_2}}({Q_T})$ to be the subset of all functions on $Q_T$, which satisfy 
$$ 
[u]_{C^{{\alpha_1, \alpha_2}}({Q_T})}:=\sup_{x\in\bT} \sup_{{\substack{t_1, t_2\in[0,T]\\ t_1 \ne t_2}}} \frac{|u(t_1,x)-u(t_2,x)|}{\abs{t_1-t_2}^{\alpha_1}} + \sup_{t\in[0,T]}\sup_{{\substack{x_1, x_2 \in \bT\\ x_1 \ne x_2}}}\frac{\abs{u(t,x_1)-u(t,x_2)}}{\abs{x_1-x_2}^{\alpha_2}}<\infty,$$
and we set 
$$
\|u\|_{C^{{\alpha_1, \alpha_2}}({Q_T})} := \sup_{(t,x) \in Q_T} |u(t,x)|+[u]_{C^{{\alpha_1, \alpha_2}}({Q_T})}.
$$
For $p \in [1,\infty]$, a measure space $\left(\Omega,\cA,\mu\right)$, and a Banach space $\left(X,\norm{\cdot}\right)$, we write $L^p\left(\Omega,\cA,\mu;X\right)$ for the $X$-valued Lebesgue space of $\mu$-measurable functions $\Omega \to X$ with separable range and finite norm $\norm{\cdot}_{L^p(\Omega,\cA,\mu;X)}$, where
\[
\norm{v}_{L^p(\Omega,\cA,\mu;X)} := \begin{cases} \left(\int_\Omega \abs{v(y)}^p \diff \mu(y)\right)^{\frac 1 p} & \mbox{ if } p \in [1,\infty), \\ \messsup{\mu}_{y \in \Omega} \abs{v(y)} & \mbox{ if } p = \infty, \end{cases} 
\]
where
\[
\messsup{\mu}_{y \in \Omega} \abs{v(y)} := \inf\left\{C \in [0,\infty] \colon \norm{v} \le C \mbox{ $\mu$-almost everywhere}\right\}
\]
denotes the essential supremum of $\abs{v}$. For $p = 2$ and a Hilbert space $\left(X,\inner{\cdot,\cdot}\right)$, we have $\norm{v}_{L^2(U,\cA,\mu;X)} := \sqrt{\inner{v,v}_{L^2(U,\cA,\mu;X)}}$, where the inner product is given by $\inner{w_1,w_2}_{L^2(U,\cA,\mu;X)} := \int_U \inner{w_1(y),w_2(y)}_X \diff\mu(y)$ for $w_1, w_2 \in L^2(U)$. We write $L^p(U,\cA,\mu) := L^p(U,\cA,\mu;\R)$ if $X = \R$. If $U \subseteq \R^d$ with $d \in \N$ is Borel measurable, $\cA$ is the Borel $\sigma$-algebra $\cB(U)$, and $\mu = \lambda_U$ the Lebesgue measure on $U$, we simply write $L^p(U;X) := L^p(U,\cB(U),\lambda_U;X)$ and if additionally $X = \R$, we write $L^p(U) := L^p(U;\R)$.  For $v \in L^1(U)$,  we write 
$$
\A( v ):= \frac{1}{\abs{U}} \int_U v(y) \, \diff y, 
$$ 
for its average value.

For $k \in \N$, $1 \le p \le \infty$, and $U \subset \R^d$ with $\D U \in C^\infty$, we write $W^{k,p}(U;X)$ for the Sobolev space of all $u \in L^p(U;X)$ such that $\D^\alpha u \in L^p(U;X)$ for all $\alpha \in \N_0^d$ with $\abs{\alpha} \le k$, where the norm is given by
\[
\norm{u}_{W^{k,p}(U;X)} := \sum_{\abs{\alpha} \le k} \norm{\D^\alpha u}_{L^p(U;X)}.
\]
For $s \in (0,1)$ and $u \colon U \to X$ measurable, we define
\[
[u]_{W^{s,p}(U;X)} := \left(\int_U \int_U \frac{\norm{u(x) - u(y)}^p}{\abs{x-y}^{s p}} \, \diff x \, \diff y\right)^{\frac 1 p}.
\]
For $s \in (0,\infty)$, we define the Sobolev-Slobodeckij space $W^{s,p}(U;X)$ as the space of all $u \in W^{\lfloor s \rfloor,p}(U;X)$ such thar $[\D^\alpha u]_{W^{s-\lfloor s \rfloor,p}(U;X)} < \infty$ for all $\alpha \in \N_0^d$ with $\abs{\alpha} = \lfloor s \rfloor$, where the norm is given by $\norm{u}_{W^{s,p}(U;X)} := \norm{u}_{W^{\lfloor s \rfloor,p}(U;X)} + \sum_{\abs{\alpha} = \lfloor s \rfloor} [\D^\alpha u]_{W^{s-\lfloor s \rfloor,p}(U;X)}$.

For $k \in \N_0$ we define the periodic Sobolev space $H^k(\bT)$ as the closure of all smooth $v \colon \bT \to \R$ such that the norm $\norm{v}_{H^k(\bT)}$ is finite, where $\norm{v}_{H^k(\bT)} := \sqrt{\inner{v,v}_{H^k(\bT)}}$ and the inner product is defined as $
\inner{w_1,w_2}_{H^k(\bT)} := \sum_{j = 0}^k \inner{\D_x^j w_1, \D_x^j w_2}_{L^2(\bT)}$ for all smooth $w_1, w_2 \colon \bT \to \R$. We define $H^{-k}(\bT) := \left(H^k(\bT)\right)'$ as the dual of $H^k(\bT)$ relative to $L^2(\bT)$. 

For $s \in \R \setminus \Z$ we introduce the fractional Sobolev space $H^s(\bT)$ as the closure of all smooth $v \colon \bT \to \R$ such that the norm $\norm{v}_{H^s(\bT)}$ is finite, where $\norm{v}_{H^s(\bT)} := \sqrt{\inner{v,v}_{H^s(\bT)}}$ and the inner product is defined as $\inner{w_1,w_2}_{H^s(\bT)} := \sum_{k \in \Z} \left(1+\lambda_k^s\right) \hat w_1(k) \, \hat w_2(k)$ for all smooth $w_1, w_2 \colon \bT \to \R$, where $\hat w_j(k) := \frac{1}{\sqrt L} \int_\bT e_k(x) \, w_j(x) \, \diff x$ is the discrete Fourier transform with respect to the family $(e_k)_{k \in \Z}$ defined in \eqref{basis-e-k} below and $\lambda_k \stackrel{\eqref{def-lam-k}}{=} \frac{4 \pi^2 k^2}{L^2}$.

For $s \in \R$ we write $H^s_\mathrm{w}(\bT)$ for the space $H^s(\bT)$ endowed with the weak topology.

\subsection{Setting\label{sec:setting}}
Suppose we are given a stochastic basis $\left(\Omega,\cF,\F,\P\right)$, that is, the triple $\left(\Omega,\cF,\P\right)$ is a complete  probability space and $\F=( \mathcal{F}_t)_{t \in [0,T]}$ is a filtration satisfying the usual conditions. Further suppose that  independent real-valued standard $\F$-Wiener processes $\left(\beta_k\right)_{k \in \Z}$ are given. For what follows, we write
\begin{equation}\label{def-m-f-0}
M(r) := F_0^2(r), \quad \mbox{where} \quad F_0(r) := \abs{r}^{\frac n 2} \quad \mbox{for} \quad r \in \R,
\end{equation}
and $n \ge 1$ is a fixed real constant called mobility exponent. Further assume that $\sigma := \left(\sigma_k\right)_{k \in \N}$ is an orthogonal family of eigenfunctions for the negative one-dimensional Laplacian $- \Delta = - \D_x^2$ on $\bT$ (i.e., periodic boundary conditions are employed). Specifically, we introduce the orthonormal basis $(e_k)_{k=-\infty}^\infty$ of $L^2(\bT)$ with
\begin{subequations}\label{ass-sigma-k}
\begin{equation}\label{basis-e-k}
e_k(x) := \sqrt{\frac{2}{L}} \begin{cases} \cos\left(\frac{2 \pi k x}{L}\right) & \mbox{ for } k \ge 1 \mbox{ and } x \in \bT, \\
\frac{1}{\sqrt 2} & \mbox{ for } k = 0 \mbox{ and } x \in \bT, \\ \sin\left(\frac{2 \pi k x}{L}\right) & \mbox{ for } k \le - 1 \mbox{ and } x \in \bT,
\end{cases}
\end{equation}
so that in particular
\begin{equation}\label{ev_lap}
\partial_x e_k = \underbrace{\frac{2 \pi k}{L}}_{= \mathrm{sign}(k) \sqrt{\lambda_k}} e_{-k} \quad \mbox{and} \quad - \partial_x^2 e_k = \underbrace{\frac{4 \pi^2 k^2}{L^2}}_{= \lambda_k} e_k \quad \mbox{for} \quad k \in \Z,
\end{equation}
where
\begin{equation}\label{def-lam-k}
\lambda_k := \frac{4 \pi^2 k^2}{L^2}.
\end{equation}
We then write
\begin{equation}\label{def-sigma-k}
\sigma_k =: \nu_k e_k \quad \mbox{with $\nu_k \in \R$}
\end{equation}
and assume
\begin{equation}\label{eq:reg-sigma-k}
\sum_{k\in\Z} \lambda_k^2 \nu_k^2 < \infty.
\end{equation}
Notice that because of \eqref{ev_lap}, this implies that
\begin{equation}\label{eq:winf-sigma-k}
\sum_{k \in \mathbb{Z}} \| \sigma_k \|_{W^{2,\infty}(\bT)}^2< \infty.
\end{equation}
\end{subequations}
Now, we introduce the $H^2(\bT)$-valued Wiener process
\begin{equation}\label{eq:def-wiener}
W(t,x) := \sum_{k \in \Z} \sigma_k(x) \beta^k(t) \quad \mbox{for} \quad (t,x) \in [0,T] \times \bT.
\end{equation}
The stochastic partial differential equation (SPDE) \eqref{stfe-general} thus attains the form
\[
\diff u =  \D_x \left( - F^2_0 (u) \D^3_x u\right) \diff t + \sum_{k \in \Z} \D_x \left( \sigma_k F_0 (u) \right) \circ \diff \beta^k \quad \mbox{in $[0,T] \times \bT$.}
\]
It is more convenient for the subsequent analysis to rewrite this equation using It\^o calculus, leading to a stochastic correction of the drift (in the physics literature sometimes referred to as the \emph{spurious drift}), that is,
\begin{equation} \label{eq:exact-equation}
\diff u =  \left[\D_x \left( - F^2_0 (u) \D^3_x u\right)  + \frac{1}{2} \sum_{k \in \Z} \D_x \left( \sigma_k F_0'(u) \D_x \left( \sigma_k F_0 (u) \right)\right)\right] \diff t + \sum_{k \in \Z} \D_x \left( \sigma_k F_0 (u) \right) \diff \beta^k
\end{equation}
in $[0,T] \times \bT$.

\subsection{Main Result and Discussion}
We have the following notion of weak (or martingale) solutions to \eqref{eq:exact-equation}:

\begin{definition}\label{def:exact-solution}
A weak (or martingale) solution to \eqref{eq:exact-equation} for $\mathcal{F}_0$-measurable initial data $u^{(0)} \in L^2(\Omega;  H^1(\bT;\R^+_0)) $ is a quadruple
\[ \left\{ ( \tilde \Omega,\tilde \cF,\tilde \F,\tilde \P), \ (\tilde \beta_k)_{k \in \Z},\ \tilde u^{(0)},  \tilde u \right\}
\]
such that $(\tilde \Omega,\tilde \cF,\tilde \F,\tilde \P)$ is a filtered probability space satisfying the usual conditions,  $\tilde u^{(0)}$ is $\tilde{\mathcal{F}}_0$-measurable and has the same distribution as  $u^{(0)}$,   $(\tilde \beta_k)_{k \in \Z}$ are  independent real-valued standard $\tilde \F$-Wiener processes, and $\tilde{u}$ is an $\tilde \F$-adapted continuous $H^1_\mathrm{w}(\bT)$-valued process, such that
\begin{enumerate}[(i)]
\item \label{it:exact-solution-1}
 $ \tilde \E \sup_{t \leq T}  \| \tilde{u}(t) \|^2_{H^1(\bT)}< \infty $

\item \label{it:exact-solution-2} For almost all $(\tilde{\omega}, t) \in \tilde\Omega \times [0,T]$, the weak derivative of third order $\D^3_x\tilde{u}$ exists on $\{ \tilde{u}(t) \neq 0\}$ and satisfies 
$ \tilde \E \| \ind_{\{\tilde{u}\neq 0\}} F_0(\tilde u) \, \D_x^3 \tilde u \|_{L^2(Q_T)}^2< \infty $,

\item \label{it:exact-solution-3} For all $\varphi \in C^\infty(\bT)$, $\diff\tilde\P$-almost surely, we have  
\begin{eqnarray}
\inner{\tilde u(t),\varphi}_{L^2(\bT)} &=& \inner{\tilde u^{(0)},\varphi}_{L^2(\bT)} + \int_0^t \int_{\left\{\tilde u(t') > 0\right\}} F_0^2(\tilde u(t')) \left(\D_x^3 \tilde u(t')\right) \left(\D_x \varphi\right) \diff x \, \diff t' \nonumber \\
&& - \frac 1 2 \sum_{k \in \Z} \int_0^t \inner{\sigma_k F_0'(\tilde u(t')) \D_x \left(\sigma_k F_0(\tilde u(t'))\right), \D_x\varphi}_{L^2(\bT)} \diff t' \nonumber \\
&& - \sum_{k \in \Z} \int_0^t \inner{\sigma_k F_0(\tilde u(t')), \D_x \varphi}_{L^2(\bT)} \diff\beta^k(t') \label{eq:exact-weak}
\end{eqnarray}
 for all $t \in [0,T]$. 
\end{enumerate}
\end{definition}

The main result of this paper reads as follows:
\begin{theorem}\label{th:existence}
     Let $T \in (0,\infty)$, $n \in \left[\frac 8 3,4\right)$,  $p> n+2$, $q > 1$ satisfying $q \ge \max\left\{\frac{1}{4-n},\frac{n-2}{2n-5}\right\}$. Suppose that 
\[
u^{(0)} \in L^{p}\left(\Omega,\cF_0,\P;H^1(\bT)\right)
\]
such that $u^{(0)} \ge 0$, $\diff\P$-almost surely, $\E \abs{\Aunull}^{2pq} < \infty$, and $\E \norm{G_0\left(u^{(0)}\right)}_{L^1(\bT)}^{pq} < \infty$.
  Then \eqref{eq:exact-equation} admits a weak solution
\[ \left\{ ( \tilde \Omega,\tilde \cF,\tilde \F,\tilde \P), \ (\tilde \beta_k)_{k \in \N},\ \tilde u^{(0)},  \tilde u \right\}
\]
in the sense of Definition~\ref{def:exact-solution} such that $\tilde u \ge 0$, $ \diff \tilde\P \otimes \diff t \otimes  \diff x$-almost everywhere. This solution satisfies the estimate
\begin{align}\label{apriori-main}
  & \tilde\E \Big[\sup_{t \in [0,T]} \norm{\D_x \tilde u(t)}_{L^2(\bT)}^{p} + \sup_{t \in [0,T]} \norm{G_0\left(\tilde u(t)\right)}_{L^1(\bT)}^{pq} + \|\ind_{\{\tilde{u}>0\}}\tilde{u}^{\frac n 2}(\partial_{x}^{3}\tilde{u})\|_{L^2(Q_T)}^{p}
    + \norm{\D_x^2 \tilde u}_{L^2(Q_T)}^{2pq}\Big] \nonumber \\
    & \le C \, \E \left[1 + \abs{\Aunull}^{2pq} + \norm{G_0(u^{(0)})}_{L^1(\bT)}^{pq} + \norm{\D_x u^{(0)}}_{L^2(\bT)}^{p} \right],
  \end{align}
where
$C < \infty$ is a constant depending only on $p,q,\sigma=(\sigma_k)_{k\in\Z}, n,L,$ and $T$.
Moreover, 
\begin{equation}\label{hoelder-main}
\tilde u\in L^{p'}(\tilde \Omega, \tilde\F, \tilde\P; {C^{\frac \gamma 4,\gamma}}(Q_T)),\quad \mbox{for all} \quad \gamma\in\left(0,\tfrac12\right) \quad \mbox{and} \quad p'\in \left[1,\tfrac{2p}{n+2}\right).
\end{equation}
\end{theorem}
The proof of Theorem ~\ref{th:existence} is given in Section \ref{sec:proof_main_result} below.

Theorem~\ref{th:existence} is a global existence result for weak solutions to the stochastic thin-film equation \eqref{eq:exact-equation} for a range of mobility exponents, including the cubic one $n = 3$, corresponding to a no-slip condition at the substrate of the underlying stochastic Navier-Stokes equations (see \cite{DMS2005} for details on the modelling and a non-rigorous derivation). Therefore, Theorem~\ref{th:existence} in particular applies to the physically relevant situation  considered in \cite{DMS2005}. We expect that the limitations $n \ge \frac 8 3$ and $n < 4$ are due to technical reasons and that these restrictions can be potentially removed in future work by making use of so-called $\alpha$-entropies as first introduced in \cite{BBDP1995}. Similarly, upgrading Theorem~\ref{th:existence} to cover higher dimensions, as done in \cite{PGG1998,Gruen2004}, would be an interesting direction for future research. Notably, our solutions are nonnegative as in \cite{GessGnann2020} but since $\norm{G_0\left(\tilde u(t)\right)}_{L^1(\bT)}$ is $\diff t \otimes \diff \tilde\P$-almost everywhere finite, by \eqref{eq:entropy-eps-0} it holds $\abs{\left\{\tilde u(t) = 0\right\}} = 0$ for all $t \in [0,T]$, $\diff \tilde\P$-almost everywhere. Since the arguments in \cite{GessGnann2020} are purely energetic, the support of the initial data in \cite{GessGnann2020} is not necessarily $\bT$ and in general this is  not the case for the corresponding solution of the SPDE, either. We expect that it is possible to overcome this constraint also in the situation of this paper by using a renormalization technique, which will be left as an endeavour for future research, too.

\section{Galerkin Approximation\label{sec:galerkin}}
In this section, we use the definitions and assumptions of \S\ref{sec:setting}.
\subsection{Setup}
We write $V_N= \text{span} \{e_{-N},...,e_N \}$, where the $(e_j)_{j \in \Z}$ are defined as in \eqref{basis-e-k}, $N \in \N$, and let $\Pi_N : L^2(\bT) \to V_N$ be the orthogonal projection given by
\begin{equation}\label{projection}
\Pi_N v = \sum_{j=-N}^N \inner{v, e_j}_{L^2(\bT)} e_j \quad \mbox{for any} \quad v \in L^2(\bT).
\end{equation} 
It is immediate from \eqref{ev_lap} that $\partial_x^2 \Pi_N = \Pi_N \partial_x^2$. Furthermore, we obtain for any $v \in L^2(\bT)$ through integration by parts and with our specific choice of eigenfunctions,
\begin{eqnarray}
\Pi_N (\partial_x v) &\stackrel{\eqref{projection}}{=}& \sum_{j = -N}^N \inner{\partial_x v, e_j}_{L^2(\bT)} e_j = - \sum_{j = -N}^N \inner{v,\partial_x e_j}_{L^2(\bT)} e_j \nonumber \\
&\stackrel{\eqref{ev_lap}}{=}& \sum_{j = -N}^N \inner{v, e_{-j}}_{L^2(\bT)} \frac{- 2 \pi j}{L} e_j \stackrel{\eqref{ev_lap}}{=} \sum_{j = -N}^N \inner{v, e_{-j}}_{L^2(\bT)} \partial_x e_{-j} \nonumber \\
&=& \partial_x \left(\Pi_N v\right). \label{Pi_N_pr_x}
\end{eqnarray}
Let $g \colon [0,\infty) \to [0,1]$ be a smooth function such that $g = 1$ on $[0,1]$ and $g = 0$ on $[2,\infty)$. Further define $g_R(s) := g(s/R)$ for $s \in [0,\infty)$ and set for $\eps \ge 0$
\begin{equs}\label{eq:def-f-eps}
F_\eps(r):= \left(r^2+\eps^2\right)^{\frac n 4} \quad \mbox{for $r \in \R$},
\end{equs}
where $n > 0$ is constant. Notably, for $\eps = 0$ the definition \eqref{eq:def-f-eps} is consistent with the corresponding expression in \eqref{def-m-f-0}, but we will assume $\eps > 0$ and thus that $F_\eps$ is smooth with $F_\eps(r) \ge \eps^{\frac n 2}$ for all $r \in \R$ until \S\ref{sec:eps-limit}. We consider the Galerkin scheme, i.e., the finite-dimensional stochastic differential equation (SDE)
\begin{eqnarray}\nonumber
\diff u_{\eps,R,N} &=&  \Pi_N \left[ \D_x\left(-F_\eps^2(u_{\eps,R,N}) \D_x^3 u_{\eps,R,N}\right) \right] \diff t \\
&& + \frac 1 2 g_R^2\left(\norm{u_{\eps,R,N}}_{L^\infty(\bT)}\right) \Pi_N \left[\sum_{k \in \Z} \partial_x \left(\sigma_k F_\eps'(u_{\eps,R,N}) \partial_x(\sigma_k F_\eps(u_{\eps,R,N}))\right) \right] \diff t \nonumber \\
&&+ g_R\left(\norm{u_{\eps,R,N}}_{L^\infty(\bT)}\right) \Pi_N \left[\sum_{k \in \Z} \partial_x\left(\sigma_k F_\eps(u_{\eps,R,N})\right) \diff \beta^k\right]. \label{galerkin_sde}
\end{eqnarray} 
 The approximation in \eqref{galerkin_sde} is three-fold. While applying the projection $\Pi_N$ yields a finite-dimensional SDE, additionally the mobility $F_0^2$ is regularized with $F_\eps^2$, so that the limiting equation as $N \to \infty$ is non-degenerate if $\eps > 0$. For technical reasons in what follows, we also cut off the noise with the pre-factor $g_R\left(\norm{u_{\eps,R,N}}_{L^\infty(\bT)}\right)$.

Notice that \eqref{galerkin_sde} is equivalent to the system on $\R^{2N+1}$
\begin{equs}   \label{eq:system}
\diff y= [A_1( y)+ A_2(y) ] \,  \diff t +\sum_{k\in \Z} B^k(y)  \, \diff \beta^k(t),  
\end{equs}
where, with the short-hand notation $v_y(x)= \sum_{j=-N}^N y^j e_j (x)$  for $y \in \R^{2N+1}$, 
\begin{equs}
A_1^i(y)& = \inner{ F^2_\eps (v_y)\sum_{j=-N}^N y^j \D^3_x e_j , \D_x e_i }_{L^2(\bT)},
\\
A_2^i(y)& =-\frac 1 2 g_R^2\left(\| v_y \|_{L^\infty(\bT)}\right)   \inner{  \sum_{k \in \Z} \sigma_k F_\eps'(v_y ) \partial_x(\sigma_k F_\eps(v_y ) , \partial_x e_i }_{L^2(\bT)},
\\
B^k(y)& =-  \inner{g_R\left(\norm{v_y}_{L^\infty(\bT)}\right) \sigma_k F_\eps(v_y), \partial_xe_i} _{L^2(\bT)}.
\end{equs}
Let us consider on $\R^{2N+1}$ the inner product 
\begin{equs}
\inner{ y^1, y^2 }_\lambda : = \sum_{j=_N}^N \lambda_j y^j_1 y^j_2, 
\end{equs}
and denote by $\| \cdot \|_\lambda$ the corresponding norm. By \eqref{ev_lap},  it is easy to see that for all $y \in \R^{2N+1}$ we have 
\begin{equs}
\inner{ A_1(y), y }_\lambda =- \|  F_\eps(v_y) \D^3_x v_y \|_{L^2(\bT)}  \leq 0. 
\end{equs}
In addition, because of the truncation in $R$ and the finite dimensionality, it is easy to see that there exists a constant $C= C(R, N)$ such that for all $y \in \R^{2N+1}$
\begin{equs}
\|A_2(y) \|_\lambda + \sum_{k=1}^ \infty \| B^k(y)\|_\lambda^2 \leq C.
\end{equs}
This shows that the system \eqref{eq:system} is coercive, which combined with the local Lipschitz continuity of the coefficients implies that for any $\mathcal{F}_0$-measurable  random variable in $\R^{2N+1}$, there exists a unique solution of \eqref{eq:system} starting from $y_0$. In particular, \eqref{galerkin_sde} has a unique solution starting from $u_N^{(0)} := \Pi_N u^{(0)}$. 
Finally, notice that with \eqref{Pi_N_pr_x} it follows that \eqref{galerkin_sde} is still in divergence form so that in particular $\A(u_{\eps,R,N}(t)) = \A(u_{\eps,R,N}(0))$ for any $t \in [0,T]$.

\subsection{Energy Estimate for the Galerkin Scheme}

%
\begin{lemma}           \label{lem:energy-R-dependent}
Suppose $p \in [2,\infty)$, $u^{(0)} \in L^p\left(\Omega,\mathcal{F}_0,\P;H^1(\bT)\right)$, and $n > 0$.  Let  $u_{\eps,R,N}$ be the unique solution to \eqref{galerkin_sde} with initial data $u_N^{(0)}$. Then $u_{\eps,R,N}$ satisfies
\begin{align}
&\E \sup_{t \leq T } \norm{\D_x u_{\eps,R,N}(t)}_{L^2(\bT)}^p + \E \norm{F_\eps (u_{\eps,R,N}) \, \D_x^3 u_{\eps,R,N}}_{L^2(Q_T)}^p \nonumber \\
& \quad \le C \left(1 + \E \norm{\D_x u^{(0)}}^p_{L^2(\bT)}\right),  \label{galerkin_energy}  
\end{align}
where $C < \infty$ is a constant depending only on $\eps$, $R$, $p$, $\sigma = \left(\sigma_k\right)_{k \in \Z}$, $n$, and $T$ (but not on $N$). 
\end{lemma}
\begin{proof}
For convenience, we drop the dependence on $\eps$, $R$, and $N$ in the notation and simply write $u$ and $\gamma_u(t) := g_R\left(\norm{u(t)}_{L^\infty(\bT)}\right)$. Applying It\^o's formula to \eqref{galerkin_sde}, we have, $\diff\P$-almost surely,
\begin{eqnarray*}
\lefteqn{\frac 1 2 \norm{\partial_x u(t)}_{L^2(\bT)}^2 - \frac 1 2 \norm{\partial_x u(0)}_{L^2(\bT)}^2} \\
&=& - \int_0^t \inner{\partial_x u(t'), \partial_x \left(\Pi_N \partial_x \left(F_\eps^2(u(t')) \partial_x^3 u(t')\right)\right)}_{L^2(\bT)} \diff t' \\
&& + \frac 1 2 \sum_{k \in \Z} \int_0^t {\gamma_u^2}(t') \inner{\partial_x u(t'), \partial_x \left(\Pi_N \partial_x \left(\sigma_k F_\eps'(u(t')) \partial_x \left(\sigma_k F_\eps(u(t'))\right)\right)\right)}_{L^2(\bT)} \diff t' \\
&& + \frac 1 2 \sum_{k \in \Z} \int_0^t {\gamma_u^2}(t') \norm{\partial_x \left(\Pi_N \partial_x \left(\sigma_k F_\eps(u(t'))\right)\right)}_{L^2(\bT)}^2 \diff t' \\
&& + \sum_{k \in \Z} \int_0^t \gamma_u(t') \inner{\partial_x u(t'), \partial_x \left(\Pi_N \partial_x \left(\sigma_k F_\eps(u(t'))\right)\right)}_{L^2(\bT)} \diff \beta^k(t')
\end{eqnarray*}
for all $t \in [0,T]$. 
Since $\Pi_N$ is an orthogonal projection, it furthermore holds $\inner{v, \Pi_N w}_{L^2(\bT)} = \inner{\Pi_N v, w}_{L^2(\bT)}$ for any $v, w \in L^2(\bT)$. Since $\Pi_N u = u$ and $\norm{\Pi_N v}_{L^2(\bT)} \le \norm{v}_{L^2(\bT)}$ for any $v \in L^2(\bT)$, we obtain with the help of \eqref{Pi_N_pr_x} the simplification
\begin{eqnarray*}
\lefteqn{\frac 1 2 \norm{\partial_x u(t)}_{L^2(\bT)}^2 - \frac 1 2 \norm{\partial_x u^{(0)}}_{L^2(\bT)}^2} \\
&\le& - \int_0^t \inner{\partial_x u(t'), \partial_x^2 \left(F_\eps^2(u(t')) \, \partial_x^3 u(t')\right)}_{L^2(\bT)} \diff t' \\
&& + \frac 1 2 \sum_{k \in \Z} \int_0^t {\gamma_u^2}(t') \inner{\partial_x u(t'), \partial_x^2 \left(\sigma_k F_\eps'(u(t')) \, \partial_x \left(\sigma_k F_\eps(u(t'))\right)\right)}_{L^2(\bT)} \diff t' \\
&& + \frac 1 2 \sum_{k \in \Z} \int_0^t {\gamma_u^2}(t') \norm{\partial_x^2 \left(\sigma_k F_\eps(u(t'))\right)}_{L^2(\bT)}^2 \diff t' \\
&& + \sum_{k \in \Z} \int_0^t \gamma_u(t') \inner{\partial_x u(t'), \partial_x^2 \left(\sigma_k F_\eps(u(t'))\right)}_{L^2(\bT)} \diff \beta^k(t').
\end{eqnarray*}
Integration by parts gives for the terms to the right of the inequality

\begin{eqnarray*}
\lefteqn{\inner{\partial_x u, \partial_x^2 \left(F_\eps^2(u) \, \partial_x^3 u\right)}_{L^2(\bT)} = \norm{F_\eps(u) \, \D_x^3 u}_{L^2(\bT)}^2,} \\
\lefteqn{\frac 1 2 \int_{\bT} (\partial_x u) \partial_x^2 \left(\sigma_k F_\eps'(u) \partial_x (\sigma_k F_\eps(u))\right) \diff x} \\
&=& \int_{\bT} \sigma_k^2 \left(- \tfrac 1 2 (F_\eps')^2(u) \, (\partial_x^2 u)^2 + \tfrac 1 6 \left((F_\eps')^2\right)''(u) \, (\partial_x u)^4\right) \diff x \\
&& + \int_{\bT} \left(\partial_x(\sigma_k^2)\right) \left(\tfrac{1}{16} (F_\eps^2)'''(u) + \tfrac{5}{12} \left((F_\eps')^2\right)'(u)\right) (\partial_x u)^3 \, \diff x \\
&& + \int_{\bT} \left(\partial_x^2 (\sigma_k^2)\right) \left(\tfrac 1 4 (F_\eps')^2(u) + \tfrac{3}{16} (F_\eps^2)''(u)\right) (\partial_x u)^2 \, \diff x \\
&& - \frac 1 8 \int_{\bT} \left(\partial_x^4 (\sigma_k^2)\right) F_\eps^2(u) \, \diff x, \\
\lefteqn{\frac 1 2 \int_{\bT} \left(\partial_x^2(\sigma_k F_\eps(u))\right)^2 \diff x} \\
&=& \int_{\bT} \sigma_k^2 \left(\tfrac 1 2 \, (F_\eps')^2(u) \, (\partial_x^2 u)^2 + \left(\tfrac 1 2 (F_\eps'')^2(u) - \tfrac 1 6 \left((F_\eps')^2\right)''(u)\right) (\partial_x u)^4\right) \diff x \\
&& - \frac 1 6 \int_{\bT} \left(\partial_x (\sigma_k^2)\right) \left((F_\eps')^2\right)'(u) \, (\partial_x u)^3 \, \diff x \\
&& + \int_{\bT} \left((\partial_x \sigma_k)^2 - 2 \sigma_k \, (\partial_x^2 \sigma_k)\right) (F_\eps')^2(u) \, (\partial_x u)^2 \, \diff x \\
&& + \frac 1 2 \int_{\bT} \sigma_k \, (\partial_x^4 \sigma_k) \, F_\eps^2(u) \, \diff x, \\
\lefteqn{\int_{\bT} (\partial_x u) \partial_x^2\left(\sigma_k F_\eps(u)\right) \diff x = \int_{\bT} \sigma_k \, F_\eps(u) \, \partial_x^3 u \, \diff x,}
\end{eqnarray*}
so that we can infer
\begin{eqnarray*}
\lefteqn{\frac 1 2 \norm{\partial_x u(t)}_{L^2(\bT)}^2 - \frac 1 2 \norm{\partial_x u(0)}_{L^2(\bT)}^2} \\
&\le& - \norm{F_\eps(u) \, \partial_x^3 u}_{L^2(Q_t)}^2 \\
&& +\frac 1 6 \sum_{k \in \Z} \int_0^t {\gamma_u^2} \int_{\bT} \sigma_k^2 \, (F_\eps'')^2(u) \, (\partial_x u)^4 \, \diff x \, \diff t' \\
&& + \frac{1}{16} \sum_{k \in \Z} \int_0^t {\gamma_u^2} \int_{\bT} \left(\partial_x (\sigma_k^2)\right) \left((F_\eps^2)'''(u) + 4 \left((F_\eps')^2\right)'(u)\right) (\partial_x u)^3 \, \diff x \, \diff t' \\
&& + \frac{3}{16} \sum_{k \in \Z} \int_0^t {\gamma_u^2} \int_{\bT} \left(8 \left((\partial_x \sigma_k)^2 - \sigma_k (\partial_x^2 \sigma_k)\right) (F_\eps')^2(u) + \left(\partial_x^2 (\sigma^2)\right) (F_\eps^2)''(u)\right) (\partial_x u)^2 \, \diff x \, \diff t' \\
&& + \frac 1 8 \sum_{k \in \Z} \int_0^t {\gamma_u^2} \int_{\bT} \left(4 \sigma_k \, \partial_x^4 \sigma_k - \partial_x^4 (\sigma_k^2)\right) F_\eps^2(u) \, \diff x \, \diff t' \\
&& + \sum_{k \in \Z} \int_0^t \gamma_u \int_{\bT} \sigma_k \, F_\eps(u) \, \partial_x^3 u \, \diff x \, \diff \beta^k(t').
\end{eqnarray*}
For $j, \ell \in \N_0$ with $j + \ell \le 4$ we have
\[
\sum_{k \in \Z} \norm{(\D_x^j \sigma_k) (\D_x^\ell \sigma_k)}_{L^\infty(\bT)} \stackrel{\eqref{basis-e-k}, \eqref{ev_lap}, \eqref{def-sigma-k}}{\le} \frac 2 L \sum_{k \in \Z} \lambda_k^{\frac{j+\ell}{2}} \nu_k^2 \le \frac 2 L \sum_{k \in \Z} \left(1+\lambda_k^2\right) \nu_k^2 \stackrel{\eqref{def-lam-k}, \eqref{eq:reg-sigma-k}}{<} \infty.
\]
This and our control of $\norm{u}_{L^\infty(\bT)}$ via the cut-off function $\gamma_u$ imply together with \eqref{eq:def-f-eps} that
\begin{eqnarray*}
\lefteqn{\frac 1 2 \norm{\partial_x u(t)}_{L^2(\bT)}^2 - \frac 1 2 \norm{\partial_x u^{(0)}}_{L^2(\bT)}^2} \\
&\le& - \norm{F_\eps(u) \, \partial_x^3 u}_{L^2(Q_t)}^2 + C_{\eps,R,\sigma,n} \left(1 + \int_0^t {\gamma_u^2} \int_{\bT} \left((\partial_x u)^4 + \abs{\partial_x u}^3 + (\partial_x u)^2\right) \diff x \, \diff t'\right) \\
&& + \sum_{k \in \Z} \int_0^t \gamma_u \int_{\bT} \sigma_k \, F_\eps(u) \, \partial_x^3 u \, \diff x \, \diff \beta^k(t').
\end{eqnarray*}
Now, note that $\abs{\D_x u}^3 \le \frac 1 2 (\D_x u)^4 + \frac 1 2 (\D_x u)^2$. Furthermore, if $\gamma_u>0$, then we have through integration by parts
\begin{eqnarray*}
\int_{\bT} (\D_xu )^4 \,  \diff x &=& - 3  \int_{\bT} u ( \D_x u)^2 \D^2_x u \, \diff x = - \frac 3 2 \int_{\bT} (\D_x u^2) \, (\D_x u) \, \D_x^2 u \, \diff x
\\
&=& \frac 3 2  \int_{\bT} u^2 \, ( \D_x^2 u)^2 \, \diff x 
+ \frac 3 2 \int_{\bT} u^2 \, (\D_x u) \, \D_x^3 u \, \diff x \\
&\le& C_R \left(\int_{\bT} (\D_x^2 u)^2 \, \diff x + \int_{\bT} \abs{\D_x u} \abs{\D_x^3 u} \diff x\right) \le C_R \int_{\bT} \abs{\D_x u} \abs{\D_x^3 u} \diff x.
\end{eqnarray*}
Consequently, by Young's inequality we have
\begin{equs}
C_{\eps,R,\sigma,n} \, \gamma_u \int_{\bT} (\D_xu )^4 \,  \diff x \leq \frac{\eps^n}{2} \,  \norm{\D^3_x u}_{L^2(\bT)}^2 + C_{\eps,R,\sigma,n} \, \gamma_u \norm{\D_x u}_{L^2(\bT)}^2,
\end{equs}
so that
\begin{eqnarray*}
\lefteqn{\frac 1 2 \norm{\partial_x u(t)}_{L^2(\bT)}^2 + \frac 1 4 \int_0^t \int_{\bT} F_\eps^2(u) \, (\partial_x^3 u)^2 \, \diff x \, \diff t'} \\
&\le& \frac 1 2 \norm{\partial_x u^{(0)}}_{L^2(\bT)}^2 + C_{\eps,R,\sigma,n} + \sum_{k \in \Z} \int_0^t \gamma_u \int_{\bT} \sigma_k \, F_\eps(u) \, \partial_x^3 u \, \diff x \, \diff \beta^k(t') \\
&& + C_{\eps,R,\sigma,n} \int_0^t \norm{\partial_x u(t')}^2 \diff t'.
\end{eqnarray*}
Let us set 
$$
\tau_m = \inf \{ t \geq 0 : \| \D_x u(t) \|^2_{L^2(\bT)} + \int_0^t \int_{\bT}  F_\eps^2(u) \, (\partial_x^3 u)^2 \, \diff x \, \diff t' >m   \}\wedge T .
$$
By replacing $t$ with $t \wedge \tau_m$ in the above inequality, raising to the power $\frac p 2$, taking expectations, and using Gr\"onwall's lemma, we conclude that 
\begin{equs}
\lefteqn{\E \sup_{t \in [0,\tau_m]} \norm{\partial_x u(t)}_{L^2(\bT)}^p + \E \norm{ \ind_{[0, \tau_m]} F_\eps(u) \, \partial_x^3 u}_{L^2(Q_T)}^p} \\
&\le& C_{\eps,R,p,\sigma,n,T} \left(1 + \E \norm{\partial_x u^{(0)}}_{L^2(\bT)}^p + \E \sup_{t \in [0,T]} \abs{\sum_{k \in \Z} \int_0^{t\wedge \tau_m} \gamma_u \int_{\bT} \sigma_k \, F_\eps(u) \, \partial_x^3 u \, \diff x \, \diff \beta^k(t')}^{\frac p 2}\right),    
\\
   \label{eq:en-gale-bef-fatou}
\end{equs}
The Burkholder-Davis-Gundy inequality and the Cauchy-Schwarz inequality imply
\begin{eqnarray*}
\lefteqn{C_{\eps,R,p,\sigma,n,T} \, \E \sup_{t \in [0,T]} \abs{\sum_{k \in \Z} \int_0^{t\wedge \tau_m} \gamma_u \int_{\bT} \sigma_k \, F_\eps(u) \, \partial_x^3 u \, \diff x \, \diff \beta^k}^{\frac p 2}} \\
&\le& C_{\eps,R,p,\sigma,n,T} \, \E \left(\sum_{k \in \Z} \int_0^{\tau_m} \gamma_u^2 \left(\int_{\bT} \sigma_k \, F_\eps(u) \, \partial_x^3 u \, \diff x\right)^2 \diff t'\right)^{\frac p 4} \\
&\stackrel{\eqref{eq:winf-sigma-k}}{\le}& C_{\eps,R,p,\sigma,n,T} \, \E \left(\int_0^{\tau_m} \gamma_u^2 \int_{\bT} F_\eps^2(u) \, (\partial_x^3 u)^2 \, \diff x \, \diff t'\right)^{\frac p 4} \\
&\le& C_{\eps,R,p,\sigma,n,T} + \frac 1 2 \, \E \norm{ \ind_{[0, \tau_m]} F_\eps(u) \, \partial_x^3 u}_{L^2(Q_T)}^p,
\end{eqnarray*}
which shows that the last term at the right hand side of \eqref{eq:en-gale-bef-fatou} can be dropped. The claim then follows by letting $m \to \infty$ and using Fatou's lemma. 
\end{proof}
%

\subsection{Passage to the Limit in the Galerkin Scheme}
Let us consider the equation
\begin{subequations}\label{eq:problem-eps-R-dependent}
\begin{eqnarray}
\diff u_{\eps,R} &=& \D_x\left(-F_\eps^2(u_{\eps,R}) \D_x^3u_{\eps,R}\right) \diff t \nonumber \\
&& + \frac 1 2 \sum_{k \in \Z} {g_R^2}\left(\norm{u_{\eps,R}}_{L^\infty(\bT)}\right)\partial_x \left(\sigma_k  F_\eps'(u_{\eps,R}) \partial_x(\sigma_k  F_\eps(u_{\eps,R}))\right) \diff t \nonumber \\           
&& + \sum_{k \in \Z} g_R\left(\norm{u_{\eps,R}}_{L^\infty(\bT)}\right) \left(\partial_x\left(\sigma_k  F_\eps(u_{\eps,R})\right)\right) \diff \beta^k,
\label{eq:equation-eps-R-dependent}\\
u_{\eps,R}(0,\cdot) &=& u^{(0)}.
\end{eqnarray}
\end{subequations}

\begin{definition}\label{def:weak-sol-eps-R}
Let $R\in(0,\infty]$. A weak (or martingale) solution to \eqref{eq:problem-eps-R-dependent}  is a quadruple
$$ \left\{ ( \hat \Omega, \hat\cF,\hat\F,\hat \P),  (\hat\beta_k)_{k \in \Z}, \hat u^{(0)},  \hat u_{\eps,R} \right\}
$$
such that $( \hat \Omega, \hat\cF,\hat\F,\hat \P)$ is a filtered probability space satisfying the usual conditions,  $\hat u^{(0)}$ is $\hat{\mathcal{F}}_0$-measurable and has the same distribution as  $u^{(0)}$,   $(\hat \beta_k)_{k \in \Z}$ are  independent real-valued standard $\hat \F$-Wiener processes, and $\hat u_{\eps,R}$ is  an $\hat \F$-adapted continuous $H^1(\bT)$-valued process, such that
\begin{enumerate}[(i)]

\item \label{item:integrability}$ \hat \E \| \hat u_{\eps,R}\|^2_{L^\infty(0,T; H^1(\bT))}< \infty  $ and for almost all $(\hat{\omega}, t) \in \hat\Omega \times [0,T]$,  the weak derivative of third order $\D^3_x\hat u_{\eps,R}$ exists  and satisfies 
$ \hat \E \|  F_\eps (\hat u_{\eps,R}) \, \D_x^3 \hat u_{\eps,R} \|^2_{L^2(Q_T)}< \infty $,

\item  \label{item:satisfying-the-equ} For all $\varphi \in C^\infty(\bT)$, $\diff \hat\P$-almost surely, we have  
\begin{align*}
& \inner{\hat u_{\eps,R}(t),\varphi}_{L^2(\bT)} \\
& \quad = \inner{\hat u^{(0)},\varphi}_{L^2(\bT)} + \int_0^t \int_{\left\{\hat u_{\eps,R}(t') > 0\right\}} F_\eps^2(\hat u_{\eps,R}(t')) \left(\D_x^3 \hat u_{\eps,R}(t')\right) \left(\D_x \varphi\right) \diff x \, \diff t' \nonumber \\
& \phantom{\quad =} - \frac 1 2 \sum_{k \in \Z} \int_0^t g_R^2\left(\norm{\hat u_{\eps,R}(t')}_{\Linfty}\right)\inner{\sigma_k F_\eps'(\hat u_{\eps,R}(t')) \D_x \left(\sigma_k F_\eps(\hat u_{\eps,R}(t'))\right), \D_x\varphi}_{L^2(\bT)} \diff t' \nonumber \\
& \phantom{\quad =} - \sum_{k \in \Z} \int_0^t g_R\left(\norm{\hat u_{\eps,R}(t')}_\Linfty\right)\inner{\sigma_k F_\eps(\tilde u_{\eps,R}(t')), \D_x \varphi}_{L^2(\bT)} \diff\beta^k(t')
\end{align*}
for all $t \in [0,T]$. 
\end{enumerate}
\end{definition}
\begin{remark}
\label{rem:conservation-of-mass}
  1. Note that Definition~\ref{def:weak-sol-eps-R}
  covers also the case that the cutoff by $g_R$ is not active -- just by formally setting $R=\infty.$
  \newline
  2. (Mass conservation) 
In the situation of Definition~\ref{def:weak-sol-eps-R} by setting $\varphi = 1$ in \eqref{item:satisfying-the-equ} it follows that
\begin{equs}
\int_{\bT} \hat u_{\eps,R}(t,x) \, \diff x  = \int_{\bT} \hat u^{(0)}(x)  \, \diff x =: L \A({\hat u}^{(0)})  \quad \mbox{for} \quad t \in [0,T], \quad \mbox{$\diff \hat\P$-almost surely.}
\end{equs}
Hence, by Poincar\'e's inequality there exists a constant $C_L < \infty$, only depending on $L$, such that we have
\begin{equation}\label{eq:poincares}
\norm{\hat u_{\eps,R}(t)}_{L^2(\bT)} \le C_L \left(\norm{\partial_x \hat u_{\eps,R}(t)}_{L^2(\bT)} +| \A(\hat u^{(0)})|
\right) \quad \mbox{for} \quad t \in [0,T], \quad \mbox{$\diff \hat\P$-almost surely.}
\end{equation}
\end{remark}
\begin{proposition}\label{prop:weak-eps-r}
For $n \in (0,4]$, $p \ge n+2$, and $u^{(0)} \in L^p\left(\Omega;\cF_0,\P;H^1(\bT)\right)$, problem~\eqref{eq:problem-eps-R-dependent} admits a weak solution in the sense of Definition~\ref{def:weak-sol-eps-R}.
\end{proposition}
\begin{proof} Let $(\Omega, \mathcal{F}, \mathbb{F}, \mathbb{P})$ be a filtered probability space carrying a sequence $\left(\beta^k\right)_{k=1}^\infty$ of independent $\mathbb{F}$-Wiener processes and on this probability space let $u_{\eps, R, N}$ be the unique --probabilistically-- strong  solution of \eqref{galerkin_sde}. 
From now on, since $\eps$ and $R$ are fixed, we drop them and we write $u_N$ instead of $u_{\eps, R, N}$ in order to simplify the notation. By Lemma \ref{lem:energy-R-dependent} we have that $u_N$ satisfies the bound
\begin{equs}      
\E \sup_{t \in [0,T]} \norm{\D_x u_N(t)}_{L^2(\bT)}^p + \E \norm{F_\eps (u_N) \D_x^3 u_N}_{L^2(Q_T)}^p \le C, \label{eq:energy-R-eps-dependent}
\end{equs}
where $C < \infty$ is independent of $N$. Let us introduce the notation $\gamma_w(t) := g_R\left(\norm{w(t)}_{L^\infty(\bT)}\right)$, and  let us decompose $u_N$ as 
$u_N= u^{(1)}_N+u^{(2)}_N$, where 
\begin{eqnarray*}
 u^{(1)}_N(t) &:=&  u_N^{(0)}+  \int_0^t  \D_x \left(\Pi_N \left(-F_\eps^2(u_N(t')) \D_x^3 u_N(t')\right) \right) \diff t' \\
&& + \frac 1 2 \sum_{k \in \Z} \int_0^t {\gamma_{\hat u_N}^2(t')} \, \partial_x \left(\Pi_N \left(\sigma_k  F_\eps'(u_N(t')) \partial_x(\sigma_k F_\eps(u_N(t')))\right) \right) \diff t'
\end{eqnarray*}
and
\begin{equs}
u^{(2)}_N(t') := \sum_{k \in \Z} \int_0^t \gamma_{\hat u_N}(t') \, \partial_x \left(\Pi_N \left(\sigma_k F_\eps(u_N(t'))\right)\right) \diff \beta^k(t'),
\end{equs}
(recall that we can interchange the projection operator and the derivative by virtue of \eqref{Pi_N_pr_x}).
Let $\alpha \in \left(0, \frac 1 2\right)$ such that $\alpha > \frac 1 p$. By Sobolev's embedding and Young's inequality, we have 
\begin{eqnarray*}
\lefteqn{\sup_{N \in \N} \E \norm{u^{(1)}_N}^2 _{W^{\alpha,p}\left(0,T;H^{-1}(\bT)\right)}} \\
&\le& C \sup_{N \in \N} \E \norm{u^{(1)}_N}^2 _{W^{1,2}\left(0,T;H^{-1}(\bT)\right)} 
\\
&\le& C \sup_{N \in \N} \left(\E \norm{u_N^{(0)}}_\Ltwo^2 + \E \norm{F^2_\eps(u_N) \D^3_x u_N}_{L^2(Q_T)}^2\right) \\
&& + C \sup_{N \in \N} \sum_{k \in \Z} \E \norm{\sigma_k  F_\eps'(u_N) \partial_x(\sigma_k F_\eps(u_N))}^2_{L^2(Q_T)} 
\\
&\le&  C \left(\norm{u^{(0)}}_{L^2(\bT)}^2 + \sup_{N \in \N} \E \left( \|F_\eps(u_N)\|_{L^\infty(\bT)}^2 \norm{F_\eps(u_N) \D^3_x u_N}_{L^2(Q_T)}^2\right)\right)  \\
&& + C \sup_{N \in \N} \sum_{k \in \Z} \E \norm{\sigma_k  F_\eps'(u_N) \partial_x(\sigma_k F_\eps(u_N))}^2_{L^2(Q_T)} 
\\
&\stackrel{\eqref{eq:def-f-eps}, \eqref{eq:winf-sigma-k}}{\le}& C \left( 1+ \sup_{N \in \N} \E \sup_{t \in [0,T]} \norm{u_N}_{H^1(\bT)}^{n+2}\right) + C \sup_{N \in \N} \E \norm{F_\eps(u_N) \D^3_x u_N}_{L^2(Q_T)}^{n+2} \stackrel{\eqref{eq:energy-R-eps-dependent}}{<} \infty,
\end{eqnarray*}
where we have used $2n-2 \le n+2$. By \cite[Lemma 2.1]{Flandoli} we get
\begin{eqnarray*}
\sup_{N \in \N} \E \norm{u^{(2)}_N}^p_{W^{\alpha, p}\left(0,T;H^{-1}(\bT)\right)} &\le& C \sup_{N \in \N} \int_0^T \E \norm{u_N(t)}_{H^1(\bT)}^p \, \diff t \\
&\le& C \sup_{N \in \N} \sup_{t \in [0,T]} \E \norm{u_N(t)}_{H^1(\bT)}^p \stackrel{\eqref{eq:energy-R-eps-dependent}}{<} \infty.
\end{eqnarray*}
From these two estimates we have that
\begin{equs}\label{eq:finite-expectation}
\sup_{N \in \N} \E \norm{u_N}_{W^{\alpha, p}\left(0,T;H^{-1}(\bT)\right) \cap L^\infty\left(0,T;H^1(\bT)\right)}< \infty.
\end{equs}
Let us set 
\[
\beta(t) := \sum_{k\in\Z} 2^{-\abs{k}} \beta^k(t) \, \mathfrak{e}_k,
\]
where $(\mathfrak{e}_k)_{k \in \Z}$ is the standard orthonormal basis of $\ell^2(\Z)$.
We now fix ${s} \in \left(\frac 1 2, 1\right)$. By \cite[\S8, Corollary 5]{Simon1986} we have  that the embedding
\begin{equs}
W^{\alpha, p}\left(0,T;H^{-1}(\bT)\right) \cap  L^\infty\left(0,T;H^1(\bT)\right) \hookrightarrow C\left([0,T];H^{{s}}(\bT)\right)
\end{equs}
is compact. Combining this with \eqref{eq:energy-R-eps-dependent} and \eqref{eq:finite-expectation}, it follows that
for each $\delta>0$ a compact set $K_\delta \subset \cZ:= C\left([0,T];H^{{s}}(\bT)\right) \times \cY \times \ell^2(\Z)$ exists, where $\cY$ denotes the linear space $L^2\left(0,T;H^3(\bT)\right)$ endowed with the weak topology, such that
\begin{equs}
\sup_{N \in \N} \P\left\{\left(u_N, u_N , \beta \right) \in K_\delta\right\} \ge 1-\delta. 
\end{equs}
By \cite[Theorem~2]{Jakubowski} (Prokhorov's theorem for non-metric spaces), there exist $\mathcal{Z}$-valued random variables $(\hat u_N, \hat\theta_N,  \hat\beta_N )$, $(\hat u,\hat\theta,  \hat\beta)$, for $N \in \N$, on a probability space $(\hat\Omega, \hat\cF, \hat\P)$ such that in $\cZ$,
\begin{equation}\label{eq:convergence-in-Z}
(\hat u_N, \hat\theta_N, \hat\beta_N ) \to (\hat u, \hat\theta, \hat\beta ) \quad \mbox{as} \quad N \to \infty, \quad \mbox{$\diff\hat\P$-almost surely},
\end{equation}
and for each $N \in \mathbb{N}$, as random variables in $\mathcal{Z}$
\begin{equation}\label{eq:distribution}
( \hat u_N, \hat \theta_N, \hat\beta_N) \sim (u_N, u_N, \beta ).
\end{equation}
It follows that
\begin{equs}   \label{eq:identification-J}
\hat\theta_N = \hat u_N \quad \mbox{and} \quad \hat\theta = \hat u.
\end{equs}
We set $\hat u^{(0)} := \hat u(0,\cdot)$. Let $\hat\F = (\hat\cF_t)_{t \in [0,T]}$ be the augmented filtration of
\[
\cG_t:= \sigma\left(\hat u(t'), \hat \beta(t'); t' \le t\right)
\]
and let
\[
\hat\beta^k(t):= 2^{\abs{k}} \big(\hat{\beta}(t),\mathfrak{e}_k\big)_{\ell^2(\Z)}.
\]
It follows that $\hat\beta^k$, $k \in \Z$, are mutually independent, standard, real-valued $\hat\cF_t$-Wiener processes (see, e.g., \cite[Proposition~5.3]{GessGnann2020} or \cite[Proof of Proposition~5.4]{DareiotisGess} or \cite[Lemma~5.7]{FischerGruen2018}). We claim that the probability space $(\hat\Omega, \hat\cF, \hat\F, \P)$ with $\hat \cF := \hat \cF_T$, together with the Wiener processes $(\hat{\beta}_k)_{k\in\Z}$ 
and the process $\hat{u}$ set up a weak solution of \eqref{eq:equation-eps-R-dependent}.

Notice that Definition \ref{def:weak-sol-eps-R} \eqref{item:integrability} is satisfied because of \eqref{eq:energy-R-eps-dependent}, \eqref{eq:convergence-in-Z}, \eqref{eq:distribution}, and Fatou's lemma. Hence, we only have to prove Definition \ref{def:weak-sol-eps-R} \eqref{item:satisfying-the-equ} and the continuity of $\hat u$ as a process with values in in $H^1(\bT)$. Let us set
\begin{equs}
\nonumber
M(\hat u, t) &:= \hat u(t)- \hat u(0,\cdot)- \int_0^t  \left(   \D_x\left(-F_\eps^2(\hat u(t')) \D_x^3 \hat u(t') \right) \right) \diff t' \\
 & + \frac 1 2 \sum_{k \in \Z} \int_0^t {\gamma_{\hat u}^2(t')}  \left(\partial_x \left(\sigma_k  F_\eps'(\hat u(t')) \partial_x(\sigma_k F_\eps(\hat u(t')))\right) \right) \diff t'
\end{equs}
and for $v \in \{ \hat u_N, u_N\}$
\begin{equs}
M_N(v, t) &:= v(t)- v(0)- \int_0^t \Pi_N \left(  \D_x\left(- F_\eps^2 (v(t')) \D_x^3 v(t') \right) \right) \diff t' \\
 & + \frac 1 2 \sum_{k \in \Z} \int_0^t {\gamma_{v}^2(t')} \Pi_N \left(\partial_x \left(\sigma_k  F_\eps'(v(t')) \partial_x(\sigma_k F_\eps(v(t')))\right) \right) \diff t'.
\end{equs}
Fix an arbitrary $l \in \Z$. We will show that for any $\varphi \in H^{-1}(\bT)$, the processes
\begin{eqnarray*}
 M^1(\hat u,  t) &:=& (M(\hat u,  t), \varphi)_{H^{-1}(\bT)},
\\ 
 M^2( \hat u , t) &:=& ( M( \hat u , t), \varphi)^2_{H^{-1}(\bT)} - \sum_{k\in\Z} \int_0^t \gamma_{\hat u}^2(t') \left(\D_x\left( \sigma_k F_\eps(\hat u(t'))\right) , \varphi \right)_{H^{-1}(\bT)}^2 \, \diff t',
\\
 M^{3}( \hat u , t) &:=& \hat \beta^l(t)(  M( \hat u , t), \varphi)_{H^{-1}(\bT)} - \int_0^t  \gamma_{\hat u}(t') \left(\D_x\left( \sigma_l F_\eps(\hat u(t'))\right) , \varphi \right)_{H^{-1}(\bT)} \diff t'
\end{eqnarray*}
are continuous $\hat{\mathcal{F}}_t$-martingales. We first show that they are continuous $\mathcal{G}_t$-martingales. Let us further assume for now that $\varphi \in \bigcup_{N \in \N} V_N$, and for $i=1,2,3$ and $v \in \{ u_N , \hat u_N\}$, let us also define the processes $M^i_N(v, t)$ as $M^i( \hat u, t)$, but with $\hat u$, $M( \hat u, t)$, $\D_x( \sigma_k F_\eps (\hat u))$ replaced by $v$, $M_N(v,t)$, $\Pi_N \D_x( \sigma_k  F_\eps (v))$, respectively. Let us fix $t' < t$ and let $\Phi$ be a bounded continuous function on $C\left([0,t'];H^{-1}(\bT)\right) \times C\left([0, t'] ; \ell^2(\Z)\right)$. We have that
\begin{equs}
\inner{M_N(u_N, t), \varphi}_{H^{-1}(\bT)} \stackrel{\eqref{galerkin_sde}}{=} \sum_{k \in \Z} \int_0^t \gamma_{u_N}(t'') \left(\Pi_N \D_x \left( \sigma_k F_\eps(u_N(t'')) \right) , \varphi \right)_{H^{-1}(\bT)} \, \diff \beta^k(t'').
\end{equs}
It follows that the $M^i_N(u_N, t)$ are continuous $\mathcal{F}_t$-martingales. Hence, 
\[
\E \left[\Phi (u_N|_{[0,t']}, \beta|_{[0,t']} )\big(M_N^i(u_N , t)-M_N^i(u_N, t')\big)\right] = 0,
\]
which combined with \eqref{eq:distribution} gives
\begin{equation}\label{eq:martingale-Ml}
\hat \E \big[\Phi (\hat u_N|_{[0,t']},\hat \beta_N |_{[0,t']} ) \big(M_N^i(\hat u_N, t)-  M_N^i(\hat u_N, t')\big)\big] =0.
\end{equation}
Next, notice that since $\varphi \in V_M$ for some $M$, we have for all $N >M$
\begin{eqnarray*}
\lefteqn{\int_0^t \inner{\Pi_N \D_x\left(-F_\eps^2(\hat u_N(t')) \D_x^3 \hat u_N(t')\right), \varphi}_{H^{-1}(\bT)} \diff t'} \\
&=& \int_0^T\int_{\bT} \ind_{[0,t]}(t') \, F_\eps^2 (\hat u_N(t')) (\D_x^3 \hat u_N(t')) \, \D_x (I-\Delta)^{-1} \varphi \, \diff x \, \diff t'.
\end{eqnarray*}
By \eqref{eq:convergence-in-Z} we have that $\diff \hat\P$-almost surely $\norm{\hat u_N - \hat u}_{C(Q_T)} \to 0$ as $N \to \infty$, which in particular implies that $\diff \hat\P$-almost surely in $L^2(Q_T)$
\[
\ind_{[0,t]} F_\eps^2(\hat u_N)\to \ind_{[0,t]} F^2_\eps(\hat u) \quad \mbox{as} \quad N \to \infty.
\]
Since in addition from \eqref{eq:convergence-in-Z} and \eqref{eq:identification-J} we have that $\diff \hat\P$-almost surely in $L^2(Q_T)$
\[
\D^3_x \hat u_N \rightharpoonup \D^3_x \hat u \quad \mbox{as} \quad N \to \infty,
\]
one easily deduces that for each $t \in [0,T]$, $\diff \hat \P$-almost surely
\begin{equation}\label{eq:MltoM-in-probability}
\inner{M_N(\hat u_N ,t), \varphi}_{H^{-1}(\bT)} \to \inner{M(\hat u , t), \varphi}_{H^{-1}(\bT)} \quad \mbox{as} \quad N \to \infty.
\end{equation}
In addition, we have
\[
\gamma_{\hat u_N}^2 \inner{\Pi_N \D_x \left( \sigma_k F_\eps \left( \hat u_N \right)\right) , \varphi}_{H^{-1}(\bT)}^2 = \gamma_{\hat u_N}^2 \inner{\sigma_k F_\eps \left( \hat u_N \right) , \D_x (I-\Delta)^{-1}\varphi}_{L^2(\bT)}^2,
\]
which combined with \eqref{eq:convergence-in-Z} (uniform convergence in $(t,x)$)
implies that $\diff \hat\P$-almost surely as $N \to \infty$
\begin{eqnarray*}
\lefteqn{\int_0^t \gamma_{\hat u_N}^2(t') \inner{\Pi_N \D_x \left( \sigma_k F_\eps \left( \hat u_N(t') \right)\right) , \varphi}_{H^{-1}(\bT)}^2  \, \diff t'}
\\
&\to& \int_0^t \gamma_{\hat u_N}^2(t') \inner{\sigma_k F_\eps \left( \hat u(t') \right) , \D_x (I-\Delta)^{-1}\varphi}_{L^2(\bT)}^2 \diff t' 
 \\
&=& \int_0^t \gamma_{\hat u_N}^2(t') \left(\D_x \left( \sigma_k F_\eps \left( \hat u(t') \right)\right) , \varphi\right)_{H^{-1}(\bT)}^2  \, \diff t'.
\end{eqnarray*}
Hence, we have in particular that $M^2_N(\hat u_N, t) \to M^2(\hat u, t)$ as $N \to \infty$ in probability. Similarly one shows  that  $M^3_N(\hat u_N, t) \to M^3(\hat u, t)$. Therefore, for each $t \in [0,T]$ we have that $M^i_N(\hat u_N, t) \to M^i(\hat u, t)$ in probability for $i\in\{1,2,3\}$. Moreover, for $q := \frac{2 p}{n} > 2$ we have
\begin{eqnarray*}
\lefteqn{\sup_{N\in\N} \hat\E \abs{\inner{M_N(\hat u_N ,t), \varphi}_{H^{-1}(\bT)}}^q}
\\
&=& \sup_{N\in\Z} \E \abs{\sum_{k \in \Z}  \int_0^t \gamma_{u_N}^2(t') \inner{\sigma_k F_\eps(u_N(t')) , \D_x (I-\Delta)^{-1}\varphi}_{L^2(\bT)} \diff\beta^k(t')}^q
\\
&\le&C \sup_{N\in\N} \E \left(\int_0^t \sum_{k=1}^\infty \gamma_{u_N}^4(t') \inner{\sigma_k F_\eps(u_N(t')) , \D_x (I-\Delta)^{-1}\varphi)}_{L^2(\bT)}^2 \diff t' \right) ^{\frac q 2}
\\
&\le& C \left(\sum_{k \in \Z} \norm{\sigma_k}_{L^2(\bT)}^2\right)^{\frac q 2} \norm{\D_x (I-\Delta)^{-1}\varphi}_{L^\infty(\bT)}^q \left( 1+ \sup_{N\in\N} \E \norm{u_N}_{L^\infty(Q_T)}^{\frac{q n}{2}} \right)
\\
&\stackrel{\eqref{eq:reg-sigma-k}}{\le}& C \left( 1+ \sup_{N\in\N} \E \sup_{t' \in [0,T]} \norm{u_N(t')}_{H^1(\bT)}^p\right) \stackrel{\eqref{eq:energy-R-eps-dependent}}{<} \infty,
\end{eqnarray*}
where in the last step we have used $\frac{q n}{2} = p$, Sobolev's inequality, and \eqref{eq:equation-eps-R-dependent} combined with conservation of mass. Similarly, for $q := \frac{2 p}{n} > 2$,
\begin{eqnarray*}
\lefteqn{\sup_{N\in\N} \hat \E \left(\sum_{k\in\Z}  \int_0^t \gamma_{\hat u_N}^2(t') \inner{\D_x\left( \sigma_k F_\eps(\hat u_N(t'))\right) , \varphi }_{H^{-1}(\bT)}^2 \diff t'\right)^{\frac q 2}} \\
&\le& C \sup_{N\in\N} \E \left(\sum_{k \in \Z} \int_0^t \gamma_{\hat u_N}^2(t') \inner{\sigma_k F_\eps(u_N(t')) , \D_x (I-\Delta)^{-1}\varphi}_{L^2(\bT)}^2 \diff t'\right)^{\frac q 2} \\
&\stackrel{\eqref{eq:reg-sigma-k}}{\le}& C \left( 1+ \sup_{N\in\N} \E \sup_{t' \in [0,T]} \norm{u_N(t')}_{H^1(\bT)}^p\right) \stackrel{\eqref{eq:energy-R-eps-dependent}}{<} \infty,
\end{eqnarray*}
from which one deduces that for each $i =1,2,3$ and $t \in [0,T]$, the $M^i_N(\hat u_N, t)$ are uniformly integrable in $\hat\omega \in \hat\Omega$. 
 Hence, we can pass to the limit in \eqref{eq:martingale-Ml} to obtain
\begin{equation}\label{eq:martingale-property}
\hat{\E}\left[\Phi\big(\hat u|_{[0,t']}, \hat \beta|_{[0,t']}\big) \big(M^i(\hat u, t)- M^i(\hat u, t')\big)\right] = 0.
\end{equation}
In addition, using the continuity of $ M^i(\hat u, t)$ in $\varphi$, the uniform integrability in $\hat\Omega$, and the fact that $\bigcup_N V_N$ is dense in $H^{-1}(\bT)$, it follows that  \eqref{eq:martingale-property} holds also for all $\varphi \in H^{-1}(\bT)$. Hence, for all $\varphi \in H^{-1}(\bT)$, $i = 1,2,3$, one can see that the $\hat {M}^i(\hat u , t)$ are continuous $\mathcal{G}_t$-martingales having finite $\frac q 2$-moments, where $q := \frac{2 p}{n}$. In particular, by Doob's maximal inequality, they are uniformly integrable (in $t \in [0,T]$), which combined with continuity (in $t \in [0,T]$) implies that they are also $\hat{\mathcal{F}}_t$-martingales. By \cite[Proposition A.1]{HOF2} we obtain that $\diff \hat\P$-almost surely, for all $\varphi \in H^{-1}(\bT)$, $t \in [0,T]$,
\begin{eqnarray}
\lefteqn{\inner{\hat {u}(t), \varphi}_{H^{-1}(\bT)}} \nonumber \\
&=& \inner{\hat {u}(0,\cdot), \varphi}_{H^{-1}(\bT)}+\int_0^t \inner{\D_x \left( -F_\eps^2(\hat{u}(t')) \D_x^3  \hat{u}(t') \right) , \varphi}_{H^{-1}(\bT)} \diff t' \nonumber \\
&&+ \frac 1 2 \sum_{k \in \Z} \int_0^t {\gamma_{\hat u}^2(t')} \left(\partial_x \left(\sigma_k  F_\eps'(\hat u(t')) \partial_x(\sigma_k F_\eps(\hat u(t')))\right), \varphi \right)_{H^{-1}(\bT)} \diff t' \nonumber \\
&&+\sum_{k\in\Z}   \int_0^t \gamma_{\hat u}(t') \left(\partial_x\left(\sigma_k  F_\eps(\hat{u}(t'))\right), \varphi\right)_{H^{-1}(\bT)} \diff \hat\beta^k(t')
\label{eq:solving-the-equation}    
\end{eqnarray}
Choosing $\varphi := (I-\Delta) \psi$  in \eqref{eq:solving-the-equation} for $\psi \in C^\infty(\bT)$, we obtain that for $\diff \hat\P \otimes \diff t$-almost all $(\hat {\omega}, t) \in \hat\Omega \times [0,T]$
\begin{eqnarray*}
\inner{\hat {u}(t), \psi}_{L^2(\bT)} &=& (u^{(0)}, \psi)_{H^{-1}(\bT)} + \int_0^t \inner{F_\eps^2 (\hat{u}(t')) \D_x^3 \hat{u}(t') ,\D_x \psi}_{L^2(\bT)} \diff t' 
\\
&& - \frac 1 2 \int_0^t {\gamma_{\hat u}^2(t')} \left( \sigma_k  F_\eps'(\hat u(t')) \partial_x(\sigma_k F_\eps(\hat u(t'))), \partial_x \psi \right)_{L^2(\bT)} \diff t'
\\
&& + \sum_{k\in\N} \int_0^t \gamma_{\hat u}(t') \inner{\sigma_k  F_\eps(\hat{u}(t')), \partial_x \psi}_{L^2(\bT)} \diff \hat\beta^k(t').    
\end{eqnarray*}
By \cite[Theorem 3.2]{KR79} we have that $\hat u $ is an $ \hat{\mathbb{F}}$-adapted continuous $L^2(\bT)$-valued process and therefore the above equality is satisfied $\diff \hat\P$-almost surely, for all $t \in [0,T]$. Moreover, from the above and the fact that $\hat u$ satisfies Definition ~\ref{def:weak-sol-eps-R} \eqref{item:integrability}, it follows that for all $\psi \in C^\infty(\bT)$, for almost all $(\hat \omega, t ) \in \hat \Omega \times (0,T)$, we have 
\[
\inner{\D_x \hat{u}(t), \psi}_{L^2(\bT)} = \inner{\D_x u^{(0)}, \psi}_{L^2(\bT)} + \int_0^t {}_{H^{-{2}}(\bT)}\left\langle v^*(t'), \psi \right\rangle_{H^2(\bT)} \diff t' + \inner{M( \hat{u}, t), \psi}_{L^2(\bT)},
\]
where 
$$
v^*:= \Delta \left( -F_\eps^2 (\hat{u}) \D_x^3 \hat{u} \right) 
$$
is a predictable $H^{-2}(\bT)$-valued process such that with probability one  $v^* \in L^2((0, T) ; H^{-2}(\bT))$,   $M(\hat{u}, \cdot)$ is an $L^2(\bT)$-valued martingale, and the duality between $H^2(\bT)$ and $H^{-2}(\bT)$ is given by means of the inner product in $L^2(\bT)$. Hence, $\D_x \hat{u}$ also satisfies the conditions of   \cite[Theorem 3.2]{KR79} with the choices $V = H^2(\bT)$ and $H=L^2(\bT)$. Consequently, $\D_x \hat u$ is also continuous $L^2(\bT)$-valued. This finishes the proof.
\end{proof}
%

\section{A-priori Estimates\label{sec:apriori}}
In this section, we use the definitions and assumptions of \S\ref{sec:setting}. 
\subsection{Entropy Estimate}
 For $r \in \R$,  let us set 
\begin{equs}\label{eq:def-G-eps}
G_\eps(r)= \int_r^\infty \int_{r'}^\infty \frac{1}{F_\eps^2 (r'')} \, \diff r'' \diff r' \quad \mbox{and} \quad  H_\eps(r)= \int_r^\infty \frac{1}{F_\eps(r')}\, \diff r',
\end{equs}
where $F_\eps(r)$ was introduced in \eqref{eq:def-f-eps}. We collect some properties of $F_\eps$, $G_\eps$, and $H_\eps$ that we will need later on. 

\begin{lemma}\label{lem:lnF}
Let $n > 2$. Then there exists a constant $C_n < \infty$, only depending on $n$, such that for all $r \in \R$ and all $\eps \in (0,1)$ we have
\begin{equs}
\abs{\ln F_\eps(r)} \le C_n \left(G_\eps(r) + \abs{r} + 1\right). 
\end{equs}
\end{lemma}

\begin{proof}
Suppose first that $r \ge 0$. We have 
\begin{equs}
\abs{\ln F_\eps(r)} \stackrel{\eqref{eq:def-f-eps}}{\le} \frac n 4 \abs{\ln \left(r^2+\eps^2\right)} & \le \frac n 4 \left( \ln 2+ 2 \abs{\ln (r+\eps)} \right) \le C_n \left(r+\eps + (r+\eps)^{2-n} + 1\right).
\end{equs}
Then, notice that 
\begin{equs}\label{eq:r+eps}
(r+\eps)^{2-n}= (n-1)(n-2)\int_r^\infty \int_{r'}^\infty \frac{1}{(r''+\eps)^n} \, \diff r'' \, \diff r' \stackrel{\eqref{eq:def-G-eps}}{\le} C_n G_\eps(r),
\end{equs}
since $(r''+\eps)^{-n} \le \left((r'')^2+\eps^2\right)^{-\frac n 2} = F_\eps^{-2}(r'')$. This proves the inequality when $r \ge 0$. If $r <0$, let us first consider the case $r^2+\eps^2 \ge 1$. In this case, we have
\begin{equs}
\abs{\ln F_\eps(r)} \stackrel{\eqref{eq:def-f-eps}}{\le} \frac n 4 \abs{\ln \left(r^2+\eps^2\right)} \le \frac n 4 \ln\left(\abs{r}+\eps\right)^2  \le \frac n 2 \left(\abs{r}+\eps\right)
\end{equs}
due to $\left(\abs{r}+\eps\right)^2 \ge r^2+\eps^2 \ge 1$. This again shows the desired inequality.  If $0 \leq r^2+\eps^2 \leq 1$, then
\begin{equs}
\abs{\ln F_\eps(r)} \stackrel{\eqref{eq:def-f-eps}}{\le} \frac n 4 \abs{\ln \eps^2} \le C_n \eps^{2-n} \leq C_n \int_0^\infty \int_0^\infty \frac{1}{(r''+\eps)^n} \, \diff r'' \, \diff r' \stackrel{\eqref{eq:def-G-eps}}{\le} C_n G_\eps(0) \le C_n G_\eps(r), 
\end{equs}
since $G_\eps$ is decreasing. This finishes the proof.
\end{proof}
\begin{lemma}\label{lem:H}
Let $n> 2$. Then there exists a constant $C_n < \infty$, only depending on $n$, such that for all $r \in \R$ and all $\eps \in (0,1)$ we have 
\begin{equs}
H_\eps^2(r) \le C_n  G_\eps(r).
\end{equs} 
\end{lemma}
\begin{proof}
Let us first look at the case $r\geq 0$. We have 
\begin{eqnarray}
H^2_\eps(r) &\stackrel{\eqref{eq:def-f-eps}}{\le}& 2^{\frac n 2} \left(\int_r^\infty \frac{1}{(r'+\eps)^{\frac n 2}} \, \diff r'\right)^2 = \frac{2^{\frac n 2 + 2}}{(n-2)^2} (r+\eps)^{2-n} \nonumber \\
&\le& 2^{\frac n 2 + 2} \frac{n-1}{n-2} \int_r^\infty \int_{r'}^\infty \frac{1}{(r''+\eps)^n} \, \diff r'' \, \diff r' \nonumber \\
&\stackrel{\eqref{eq:def-f-eps}}{\le}& 2^{\frac n 2 + 2} \frac{n-1}{n-2}  G_\eps(r), \label{eq:H-eps}
\end{eqnarray}
where for the last inequality we have used \eqref{eq:r+eps}. Hence, we only have to check the case $r<0$. In this case we have 
\begin{equs}
H_\eps(r) = \int_r^0 \frac{1}{F_\eps(r')} \, \diff r'+ H_\eps(0) \stackrel{\eqref{eq:def-f-eps}}{\le} 2 H_\eps(0)
\end{equs}
since $F_\eps$ is even. Therefore,
\begin{equs}
H_\eps^2(r) \le 4 H_\eps^2(0) \le 2^{\frac n 2 + 4} \frac{n-1}{n-2} G_\eps(r),
\end{equs}
where we have used \eqref{eq:H-eps} and the fact that $G_\eps$ is decreasing. This finishes the proof. 
\end{proof}
\begin{lemma}[Entropy Estimate]\label{lem:entropy}
Suppose that  $n \in (2,4]$, $T \in (0,\infty)$, $p \ge 1$, and $u^{(0)} \in L^p\left(\Omega;\cF_0,\P;H^1(\bT)\right)$. For a weak solution of problem~\eqref{eq:problem-eps-R-dependent} in the sense of Definition~\ref{def:weak-sol-eps-R} it holds
\begin{eqnarray}
\lefteqn{\hat\E \sup_{t \in [0,T]} \norm{G_\eps\left(\hat u_{\eps,R}(t)\right)}_{L^1(\bT)}^p + \hat \E  \norm{\D^2_x \hat u_{\eps,R}}_{L^2(Q_T)}^{2p}} \nonumber \\
&\le& C \, \hat\E \left( \norm{G_\eps\left(\hat u^{(0)}\right)}_{L^1(\bT)}^p + \abs{ \A\left({\hat u}^{(0)}\right)}^{2p}+1\right), \label{eq:entropy-estimate}
\end{eqnarray}
where $C< \infty$ is a constant depending only on $p$, $\sigma = \left(\sigma_k\right)_{k \in \Z}$, and $T$. 
\end{lemma}
\begin{proof}
For the convenience of the reader, we simply write $\hat u$ instead of $\hat u_{\eps,R}$. By It\^o's formula \cite{Krylov2013} we have
\begin{eqnarray*}
\int_{\bT} G_\eps\left(\hat u(t)\right) \, \diff x &\stackrel{\eqref{eq:equation-eps-R-dependent}}{=}& \int_{\bT} G_\eps\left(\hat u^{(0)}\right) \diff x + \int_0^t \int_{\bT} G_\eps''( \hat u)  F_\eps^2(\hat u) (\D^3_x \hat u) \, \D_x \hat u \,  \diff x \, \diff t' \\
&& + \frac 1 2 \sum_{k \in \Z} \int_0^t \gamma_{\hat u}^2 \int_{\bT} G_\eps'\left(\hat u\right) \D_x \left(\sigma_k F_\eps' (\hat u) \, \D_x (\sigma_k F_\eps (\hat u))\right) \diff x \, \diff t' \\
&& + \frac 1 2 \sum_{k \in \Z} \int_0^t \gamma_{\hat u}^2 \int_{\bT} G_\eps''(\hat u) \left(\D_x (\sigma_k F_\eps(\hat u))\right)^2 \diff x \, \diff t' \\
&& + \sum_{k \in \Z} \int_0^t \gamma_{\hat u} \int_{\bT} G_\eps'(\hat u) \left(\D_x (\sigma_k F_\eps(\hat u))\right) \diff x \, \diff \hat\beta^k(t'),
\end{eqnarray*}
$\diff \hat\P$-almost surely, where $\gamma_{\hat u}(t') := g_R\left(\norm{\hat u(t')}_{L^\infty(\bT)}\right)$, so that after integration by parts we get 
\begin{eqnarray}
\int_{\bT} G_\eps\left(\hat u(t)\right) \, \diff x &\stackrel{\eqref{eq:def-G-eps}}{=}& \int_{\bT} G_\eps\left(\hat u^{(0)}\right) \diff x - \int_0^t \int_{\bT} (\D^2_x \hat u)^2 \, \diff x \, \diff t' \nonumber \\
&& + \frac 1 2 \sum_{k \in \Z} \int_0^t \gamma_{\hat u}^2 \int_{\bT} (\partial_x \sigma_k) \, F_\eps^{-1}(\hat u) \left(\D_x (\sigma_k F_\eps(\hat u))\right) \diff x \, \diff t' \nonumber \\
&& - \sum_{k \in \Z} \int_0^t \gamma_{\hat u}^2 \int_{\bT}\sigma_k F_\eps^{-1}(\hat u) \, (\D_x \hat u) \, \diff x \, \diff \hat\beta^k(t'), \label{eq:Entropy-Ito}
\end{eqnarray}
$\diff \hat\P$-almost surely. Then we have for all $\delta >0$
\begin{eqnarray*}
\sum_{k \in \Z} \int_{\bT} (\partial_x \sigma_k) \, F_\eps^{-1}(\hat u) \left(\D_x (\sigma_k F_\eps(\hat u))\right) \diff x &=& \sum_{k \in \Z} \int_{\bT} (\D_x \sigma_k)^2 \, \diff x - \frac 1 2 \sum_{k \in \Z} \int_{\bT} (\D^2_x\sigma_k^2) \, \ln F_\eps(\hat u) \, \diff x
\\
&\stackrel{\eqref{ass-sigma-k}}{\le}& C_{\sigma,\delta} \left(1 + \norm{G_\eps(\hat u)}_{L^1(\bT)}\right)+ \delta \norm{\hat u}_{L^2(\bT)}^2,
\end{eqnarray*}
where for the last inequality we have used Lemma~\ref{lem:lnF}. Moreover, since $\D_x \hat u$ has zero average, we get from Poincar\'e's inequality using conservation of mass (cf.~Remark~\ref{rem:conservation-of-mass}),
\begin{equs}
\norm{\hat u(t)}_{L^2(\bT)} \le C_L \left(\norm{\D_x \hat u(t)}_{L^2(\bT)}+ \abs{\int_{\bT} \hat u(t) \, \diff x}\right) \le C_L \left(\norm{\D_x^2 \hat u(t)}_{L^2(\bT)}+ \abs{ \A({\hat u}^{(0)})}\right).
\end{equs}
Consequently, for any $\delta > 0$ we have
\begin{eqnarray*}
\lefteqn{\sum_{k \in \Z} \int_{\bT} (\partial_x \sigma_k) \, F_\eps^{-1}(\hat u) \left(\D_x (\sigma_k F_\eps(\hat u))\right) \diff x} \\
&\le& C_{L,\sigma,\delta} \left(1 + \abs{ \A({\hat u}^{(0)})}^2 + \norm{G_\eps(\hat u(t))}_{L^1(\bT)}\right) + \delta \norm{\D_x^2 u}_{L^2(\bT)}^2 ,
\end{eqnarray*}
Using this in \eqref{eq:Entropy-Ito}, choosing $\delta > 0$ small, rearranging, and taking the $p$-th power gives 
\begin{eqnarray*}
\norm{G_\eps (\hat u(t))}_{L^1(\bT)}^p +  \norm{\D^2_x \hat u}_{L^2(Q_t)}^{2p} &\leq& C_{L,T,\sigma,p} \left( \norm{G_\eps\left(\hat u^{(0)}\right)}_{L^1(\bT)}^p+ \abs{\A\left(\hat u^{(0)}\right)}^{2p}+1\right) \\
&& + C_p \int_0^t \norm{G_\eps (\hat u(t'))}_{L^1(\bT)}^p \, \diff t' + C_p \abs{M(t)}^p,
\end{eqnarray*}
$\diff \hat\P$-almost surely, where
\[
M(t) := - \sum_{k \in \Z} \int_0^t \gamma_{\hat u}^2 \int_{\bT}\sigma_k F_\eps^{-1}(\hat u) \, (\D_x \hat u) \, \diff x \, \diff \hat\beta^k(t')
\]
is the martingale from \eqref{eq:Entropy-Ito}. Notice that $G_\eps ( \hat u (t))$ is a continuous $L^1(\bT)$-valued process and let us set 
$$
\tau_m= \inf\{ t >0 : \| G_\eps( \hat u(t)) \|_{L^1(\bT)} + \int_0^t \| \D^2_x \hat u \|^2_{L^2(\bT)} \, \diff t' > m \}\wedge T . 
$$
Taking suprema up to $\tau_m \wedge t'$, for $t' \in [0,T]$,   in the above inequality and expectation, we obtain by virtue of the Burkholder-Davis-Gundy inequality
\begin{equation}\begin{split}             
&\hat\E \sup_{t \in[0, t']} \norm{G_\eps (\hat u(t\wedge \tau_m))}_{L^1(\bT)}^p  + \hat\E \norm{\D^2_x \hat u}_{L^2(Q_{t' \wedge \tau_m})}^{2p}\\ 
&\le  C_{L,\sigma,p} \, \hat\E \left( \norm{G_\eps\left(\hat u^{(0)}\right)}_{L^1(\bT)}^p+ \abs{ \A({\hat u}^{(0)})}^{2p}+1\right) \\
& + C_p \, \hat\E  \int_0^{t'}  \norm{G_\eps (\hat u(t'' \wedge \tau_m))}_{L^1(\bT)}^p \, \diff t''   + C_p \, \hat\E  \langle M \rangle_{t' \wedge \tau_m} ^{\frac p 2}. \label{eq:before-gronwall}
\end{split}\end{equation}
Next, by integration by parts, we have
 \begin{eqnarray*}
 \hat\E \langle M \rangle_{t' \wedge \tau_m}^{\frac p 2} &=& \hat\E \left(\sum_{k \in \Z} \int_0^{t' \wedge \tau_m} \gamma_{\hat u}^2 \left(  \int_{\bT} (\D_x\sigma_k) \, H_\eps(\hat u) \, \diff x \right)^2 \, \diff t'' \right)^{\frac p 2} \\
&\stackrel{\eqref{ass-sigma-k}}{\le}& C_\sigma \, \hat\E \left(\sum_{k \in \Z} \int_0^{t' \wedge \tau_m} \int_{\bT} \left(H_\eps(\hat u)\right)^2 \diff x \, \diff t'' \right)^{\frac p 2}.
\end{eqnarray*}
By Lemma~\ref{lem:H} we see that 
\begin{equs}
\hat\E \langle M \rangle_{t' \wedge \tau_m}^{\frac p 2} & \le C_{\sigma,n} \, \hat\E \left(\int_0^{t' \wedge \tau_m} \norm{G_\eps(\hat u(t''))}_{L^1(\bT)} \, \diff t''\right)^{\frac p 2} 
\\
& \le C_{\sigma,n,p,T} \left(1 + \hat\E \int_0^{t'} \norm{G_\eps(\hat u(t'' \wedge \tau_m))}_{L^1(\bT)}^p \, \diff t''\right).
\end{equs}
Using this and rearranging in \eqref{eq:before-gronwall}, we have the desired inequality by virtue of Gr\"onwall's inequality and Fatou's lemma. 
\end{proof}

\subsection{Uniform Energy Estimate}\label{sec:eps-uniform}

The following auxiliary result is convenient for deriving an energy estimate.
\begin{lemma}\label{lem:int-energy}
For $\eps \in (0,1)$, $n \in (2,\infty)$, and $r \in \R$ we have
\begin{equation}\label{entropy_abc}
\abs{\int_1^r (F_\eps'')^2(r') \, \diff r'} \le C_n \begin{cases} 1 + \abs{r}^{n-3} & \text{ if } n>3, \\
C_{\vartheta}\left(1 + \abs{r}^{\vartheta} + G_\eps^{{\vartheta}}(r)\right) & \text{ if } n=3 \\
1 + G_\eps^{\frac{3-n}{n-2}}(r) & \text { if } n \in \left(2, 3\right).
\end{cases}
\end{equation}
for any ${\vartheta} > 0$, where $C_n < \infty$ only depends on $n$ and $C_{\vartheta}$ only depends on ${\vartheta}$.
\end{lemma}
\begin{proof}
First note that
\[
F_\eps'(r) \stackrel{\eqref{eq:def-f-eps}}{=} \tfrac n 2 r (r^2+\eps^2)^{\frac n 4 -1}, \quad F_\eps''(r) = \tfrac n 2 (r^2+\eps^2)^{\frac n 4 -1} +n( \tfrac n 4 -1) r^2 (r^2+\eps^2)^{\frac n 4 -2},
\]
so that because of $\frac{r^2}{r^2+\eps^2} \le 1$ we have
\[
(F_\eps'')^2(r) \le C_n (r^2+\eps^2)^{\frac n 2 -2}.
\]
This implies
\begin{align*}
\abs{\int_1^r (F_\eps'')^2(r') \, \diff r'} &\le C_n \int_1^r (r')^{n-4} \, \diff r' \le C_n
\begin{cases} 1+ r^{n-3} & \text{ if } n \ne 3,\\
\ln r & \text{ if } n=3,
\end{cases} 
\quad \mbox{for $r \ge 1$}, \\
\abs{\int_1^r (F_\eps'')^2(r') \, \diff r'} &\le C_n \int_{r }^1 (r')^{n-4} \, \diff r' \le C_n \begin{cases} 1 + r^{n-3} & \text{ if } n\ne 3,\\
- \ln r & \text{ if } n=3, \end{cases}
 \quad \mbox{for $\eps \le r < 1$,} \\
\abs{\int_1^r (F_\eps'')^2(r') \, \diff r'} &\le C_n \left(\int_\eps^1 (r')^{n-4} \, \diff r' + \int_{-\eps}^\eps \eps^{n-4} \, \diff r' \right) \\
&\le C_n \begin{cases} 1+\eps^{n-3} & \text{ if } n\ne 3,\\
-\ln \eps + 2 & \text{ if } n=3,\end{cases}
 \quad \mbox{for $-\eps \le r < \eps$,} \\
\abs{\int_1^r (F_\eps'')^2(r') \, \diff r'} &\le C_n \left(\int_\eps^1 (r')^{n-4} \, \diff r' + \int_{-\eps}^\eps \eps^{n-4} \, \diff r' + \int_{\eps}^{-r} (r')^{n-4} \, \diff r'\right) \\
&\le C_n \begin{cases} 1 + \eps^{n-3} + \abs{r}^{n-3} & \text{ if } n \ne 3,\\
-2\ln\eps +2+\ln(-r) & \text{ if } n=3, \end{cases} \quad \mbox{for $r < -\eps$,}
\end{align*}
so that because of $\eps \in (0,1)$,
\[
\abs{\int_1^r \left(F_\eps''(r')\right)^2 \diff r'} \le C_n 
\begin{cases}
\abs{r}^{n-3} \ind_{\left\{\abs{r} \ge \eps\right\}} + 1
& \text{ if } n > 3, \\
  \abs{\ln\abs{r}} \ind_{\left\{\abs{r} \ge \eps\right\}} + (1-\ln\eps) \ind_{\left\{r < \eps\right\}} & \text{ if } n=3,\\
  \left(1+\abs{r}^{n-3}\right) \ind_{\left\{r \ge \eps\right\}} + \left(1+\eps^{n-3}\right) \ind_{\left\{r < \eps\right\}} & \text{ if } n< 3.
\end{cases}
\]
Further notice that
\[
G_\eps(r) \stackrel{\eqref{eq:def-G-eps}}{=} \int_r^\infty \int_{r'}^\infty \frac{\diff r' \, \diff r''}{\left((r'')^2 + \eps^2\right)^{\frac n 2}} \ge 2^{- \frac n 2} \int_r^\infty \int_{r'}^\infty \frac{\diff r'' \, \diff r'}{(r'')^n} \ge C_n r^{2-n} \quad \mbox{for $r \ge \eps$}
\]
and
\[
G_\eps(r) = \eps^{2-n} \int_{\frac r \eps}^\infty \int_{r'}^\infty \frac{\diff r' \, \diff r''}{\left((r'')^2 + 1\right)^{\frac n 2}} > \eps^{2-n} \int_1^\infty \int_{r'}^\infty \frac{\diff r' \, \diff r''}{\left((r'')^2 + 1\right)^{\frac n 2}} \ge C_n \eps^{2-n} \quad \mbox{for $r < \eps$},
\]
so that we may infer that \eqref{entropy_abc} holds true.
\end{proof}
\begin{lemma}\label{lem:int-energy-2}
For $\eps \in (0,1)$, $n \in (2,\infty)$, and $r \in \R$ we have
\begin{equation}\label{entropy-ux3}
\abs{(F_\eps^2)'''(r) + 4 \left((F_\eps')^2\right)'(r)} \le C_n \begin{cases} 1 + \abs{r}^{n-3} & \mbox{ if } n \in [3,\infty), \\
G_\eps^{\frac{3-n}{n-2}}(r) & \mbox{ if } n \in \left[\frac 5 2,3\right), \end{cases}
\end{equation}
where $C_n < \infty$ only depends on $n$.
\end{lemma}
\begin{proof}
We compute
$$
(F_\eps')^2(r) \stackrel{\eqref{eq:def-f-eps}}{=} \tfrac{n^2}{4} \, r^2 \, (r^2+\eps^2)^{\frac n 2 -2} $$
which implies that 
$$
 \left((F_\eps')^2\right)'(r) = \tfrac{n^2}{2} \, r \, (r^2+\eps^2)^{\tfrac n 2 -2} + (\tfrac{n^3}{4}-n^2) \, r^3 \, (r^2 + \eps^2)^{\frac n 2 -3},
$$
so that
\[
\left| \left((F_\eps')^2\right)'(r) \right| \le C_n r(r^2+\eps^2)^{\frac{n}{ 2}-2} \le C_n (r^2+\eps^2)^{\frac{n-3}{2}}.
\]
Furthermore,
$$
 (F_\eps^2)'(r) = n r \, (r^2+\eps^2)^{\frac n 2 -1} 
$$
which gives 
$$
 (F_\eps^2)''(r) = n \, (r^2+\eps^2)^{\frac n 2 -1} + n(n-2) \, r^2 \, (r^2+\eps^2)^{\frac n 2 -2} 
 $$
 and 
 $$ (F_\eps^2)'''(r) = 3n(n-2) \, r \, (r^2+\eps^2)^{\frac n 2 -2} +n(n-2)(n-4) \, r^3 \, (r^2+\eps^2)^{\frac n 2 -3},
$$
whence
\[
\abs{(F_\eps^2)'''(r)} \le C_n (r^2+\eps^2)^{\frac{n-3}{2}}.
\]
Because of $(r^2+\eps^2)^{\frac{n-3}{2}} \le C_n \left(1 + \abs{r}^{n-3}\right)$, the first part of \eqref{entropy-ux3} is immediate. For $n \in \left(\frac 5 2,3\right)$ we use
\begin{eqnarray*}
(r^2+\eps^2)^{\frac{2-n}{2}} &=& \left(n-2\right) \int_r^\infty r' \left((r')^2+\eps^2\right)^{-\frac n 2} \diff r' \le  \left(n-2\right) \int_r^\infty \left((r')^2+\eps^2\right)^{\frac{1-n}{2}} \diff r' \\
&=& (n-2) \, (n-1) \int_r^\infty \int_{r'}^\infty r'' \left((r'')^2+\eps^2\right)^{-\frac{n+1}{2}} \diff r'' \, \diff r' \\
&\stackrel{\eqref{eq:def-f-eps}}{\le}& (n-2) \, (n-1) \int_r^\infty \int_{r'}^\infty \frac{1}{F_\eps^2(r'')} \, \diff r'' \, \diff r' \stackrel{\eqref{eq:def-G-eps}}{=} (n-2) \, (n-1) \, G_\eps(r), 
\end{eqnarray*}
so that \eqref{entropy-ux3} also holds true in this parameter range, too.
\end{proof}
\begin{lemma}\label{lem:energy-estimate}
Suppose $n \in \left[\frac 8 3, 4\right)$, $T \in (0,\infty)$, $p \ge 1$,  $u^{(0)} \in L^p\left(\Omega;\cF_0,\P;H^1(\bT)\right)$,  $\eps \in (0,1]$ and let $q>1$ such that $q \ge \max\left\{\frac{1}{4-n},\frac{n-2}{2n-5}\right\}$. Then, for any weak solution of problem~\eqref{eq:problem-eps-R-dependent} in the sense of Definition~\ref{def:weak-sol-eps-R} it holds
\begin{eqnarray}
\lefteqn{\hat\E \left[\sup_{t \in [0,T]} \norm{\D_x \hat u_{\eps,R}(t)}_{L^2(\bT)}^p + \norm{F_\eps(\hat u_{\eps,R}) \, \D_x^3 \hat u_{\eps,R}}_{L^2(Q_T)}^p\right]} \nonumber \\  
&\le&C  \ \hat\E \left(1 +| \A(\hat u_{\eps, R}^{(0)})|^{\frac{np}{2}}+  \norm{\D_x \hat u^{(0)}}_{L^2(\bT)}^p\right) \nonumber \\
&& + C \left(\hat\E \sup_{t \in [0,T]} \norm{G_\eps(\hat u_{\eps,R}(t))}_{L^1(\bT)}^{p} \ind_{\left[\frac 8 3,3\right]}(n) + \norm{\D_x^2 \hat u_{\eps,R}}_{L^2(Q_T)}^{2pq}\right), \label{eq:energy-est-eps}
\end{eqnarray}
where $C< \infty$ is a constant depending only on $p$, $q$, $\sigma = (\sigma_k)_{k \in \Z}$, $n$,  $L$, and $T$.
\end{lemma}

\begin{proof}
For convenience of the reader, we write $\hat u$ instead of $\hat u_{\eps,R}$. By It\^o's formula (see, e.g., \cite{Krylov2013})  we have
\begin{eqnarray*}
\frac 1 2 \norm{\partial_x \hat u(t)}_{L^2(\bT)}^2 &\stackrel{\eqref{eq:equation-eps-R-dependent}}{=}& \frac 1 2 \norm{\partial_x \hat u^{(0)}}_{L^2(\bT)}^2 - \int_0^t  \int_{\bT} F_\eps^2(\hat u) (\partial_x^3 \hat u)^2 \,\diff x \, \diff t' \\
&& + \sum_{k \in \Z} \int_0^t \gamma_{\hat u}^2 \int_{\bT}  \frac{1}{2}(\partial_x \hat u) \, \partial_x^2 (\sigma_k F_\eps'(\hat u) \partial_x (\sigma_k F_\eps(\hat u))) \, \diff x \, \diff t' \\
&& + \frac 1 2 \sum_{k \in \Z} \int_0^t\gamma_{\hat u}^2 \int_{\bT} \left(\partial_x^2 (\sigma_k F_\eps(\hat u))\right)^2 \diff x \, \diff t' \\
&& + \sum_{k \in \Z} \int_0^t \gamma_{\hat u} \int_{\bT} (\partial_x \hat u) \, \partial_x^2(\sigma_k F_\eps(\hat u)) \, \diff x \, \diff \hat\beta^k(t'),
\end{eqnarray*}
$\diff \hat\P$-almost surely, where we write $\gamma_{\hat u}(t) := g_R\left(\norm{\hat u(t)}_{L^\infty(\bT)}\right)$. The same reasoning as in the proof of Lemma~\ref{lem:energy-R-dependent} leads to
\begin{small}
\begin{eqnarray}
\nonumber
\lefteqn{\frac 1 2 \norm{\partial_x \hat u(t)}_{L^2(\bT)}^2} \\
\nonumber
&=& \frac 1 2 \norm{\partial_x \hat u^{(0)}}_{L^2(\bT)}^2 - \int_0^t \int_{\bT} F_\eps^2(\hat u) \, (\partial_x^3 \hat u)^2 \, \diff x \, \diff t' \\
\nonumber
&& +\frac 1 6 \sum_{k \in \Z} \int_0^t \gamma_{\hat u}^2 \int_{\bT} \sigma_k^2 \, (F_\eps'')^2(\hat u) \, (\partial_x \hat u)^4 \, \diff x \, \diff t' \\
\nonumber
&& + \frac{1}{16} \sum_{k \in \Z} \int_0^t \gamma_{\hat u}^2 \int_{\bT} \left(\partial_x (\sigma_k^2)\right) \left((F_\eps^2)'''(\hat u) + 4 \left((F_\eps')^2\right)'(\hat u)\right) (\partial_x \hat u)^3 \, \diff x \, \diff t' \\
\nonumber
&& + \frac{3}{16} \sum_{k \in \Z} \int_0^t \gamma_{\hat u}^2 \int_{\bT} \left(8 \left((\partial_x \sigma_k)^2 - \sigma_k (\partial_x^2 \sigma_k)\right) (F_\eps')^2(\hat u) + \left(\partial_x^2 (\sigma_k^2)\right) (F_\eps^2)''(\hat u)\right) (\partial_x \hat u)^2 \, \diff x \, \diff t' \\
\nonumber
&& + \frac 1 8 \sum_{k \in \Z} \int_0^t \gamma_{\hat u}^2 \int_{\bT} \left(4 \sigma_k \, \partial_x^4 \sigma_k - \partial_x^4 (\sigma_k^2)\right) F_\eps^2(\hat u) \, \diff x \, \diff t' \\
&& + \sum_{k \in \Z} \int_0^t \gamma_{\hat u} \int_{\bT} \sigma_k \, F_\eps(\hat u) \, \partial_x^3 \hat u \, \diff x \, \diff \hat\beta^k(t'), \label{energy_eps_0}
\end{eqnarray}
\end{small}
$\diff \hat\P$-almost surely.

\paragraph{$(\partial_x \hat u)^4$-term}
We first focus on estimating the term $\sum_{k \in \Z} \int_0^t \gamma_{\hat u}^2 \int_{\bT} \sigma_k^2 \, (F_\eps'')^2(\hat u) \, (\partial_x \hat u)^4 \, \diff x \, \diff t'$ and note that through integration by parts we have
\begin{eqnarray}
\lefteqn{\sum_{k \in \Z}  \int_0^t \gamma_{\hat u}^2 \int_{\bT} \sigma_k^2 \, (F_\eps'')^2(\hat u) \, (\partial_x \hat u)^4 \, \diff x \, \diff t'} \nonumber \\
&\stackrel{\eqref{ass-sigma-k}}{\le}& C_\sigma \int_0^t \int_{\bT} (F_\eps'')^2(\hat u) \, (\partial_x \hat u)^4 \, \diff x \, \diff t' \nonumber\\
&=& - 3 C_\sigma \int_0^t \int_{\bT} \left(\int_1^{\hat u} (F_\eps'')^2(r) \, \diff r\right) (\partial_x \hat u)^2 \, (\partial_x^2 \hat u) \, \diff x \, \diff t'. \label{split_vx4}
\end{eqnarray}
With help of the Cauchy-Schwarz and H\"older's inequality we have
\begin{eqnarray*}
\lefteqn{- 3 \int_{\bT}  \left(\int_1^{\hat u} \left(F_\eps''(r)\right)^2 \diff r\right) (\partial_x \hat u)^2 \, (\partial_x^2 \hat u) \, \diff x} \\
&\le& 3 \left(\int_{\bT} \left(\int_1^{\hat u} \left(F_\eps''(r)\right)^2 \diff r\right)^2 \diff x\right)^{\frac 1 2} \norm{\partial_x \hat u}_{L^\infty(\bT)}^2 \norm{\partial_x^2 \hat u}_{L^2(\bT)}.
\end{eqnarray*}
Next, since $\partial_x \hat u$ has zero average, we use that
\begin{equation}\label{interp_vxinf}
\norm{\partial_x \hat u}_{L^\infty(\bT)} \le C \norm{\partial_x \hat u}_{L^2(\bT)}^{\frac 1 2} \norm{\partial_x^2 \hat u}_{L^2(\bT)}^{\frac 1 2}.
\end{equation}
In the case $n \in \left[\frac 8 3,3\right]$, we deduce with help of Lemma~\ref{lem:int-energy} that
\begin{eqnarray*}
\lefteqn{\sum_{k \in \Z} \int_{\bT} \sigma_k^2 (F_\eps'')^2(\hat u) \, (\partial_x \hat u)^4 \, \diff x} \\
&\stackrel{\eqref{split_vx4}, \eqref{interp_vxinf}}{\le}& C_\sigma \left(\int_{\bT} \left(\int_1^{\hat u} (F_\eps'')^2(r) \, \diff r\right)^2 \diff x\right)^{\frac 1 2} \norm{\partial_x \hat u}_{L^2(\bT)} \norm{\partial_x^2 \hat u}_{L^2(\bT)}^2 \\
&\stackrel{\eqref{entropy_abc}}{\le}& C_{n,{\vartheta}} \left(1 + \norm{\hat u}_{L^1(\bT)} + \norm{G_\eps(\hat u)}_{L^1(\bT)}\right)^{\max\left\{\frac{3-n}{n-2}, {\vartheta}\right\}} \norm{\partial_x \hat u}_{L^2(\bT)} \norm{\partial_x^2 \hat u}_{L^2(\bT)}^2 \\
&\le& C_{n,{\vartheta}, L} \left(1 + \norm{\partial_x \hat u}_{L^2(\bT)}+ | \A(\hat u^{(0)})|  + \norm{G_\eps(\hat u)}_{L^1(\bT)}\right)^{\max\left\{\frac{3-n}{n-2}, {\vartheta}\right\}} \norm{\partial_x \hat u}_{L^2(\bT)} \\
&& \times \norm{\partial_x^2 \hat u}_{L^2(\bT)}^2,
\end{eqnarray*}
$\diff \hat\P$-almost surely, where ${\vartheta} > 0$ and we have applied conservation of mass (cf.~Remark~\ref{rem:conservation-of-mass}) and the second Poincar\'e inequality. Integration in time yields for any ${\vartheta} > 0$,
\begin{eqnarray*}
\lefteqn{\sum_{k \in \Z} \int_0^t \gamma_{\hat u}^2 \int_{\bT} \sigma_k^2 (F_\eps'')^2(\hat u) \, (\partial_x \hat u)^4 \, \diff x \, \diff t'} \\
&\le& C_{\sigma,n,{\vartheta}, L } \left(1 + \sup_{t' \in [0,t]} \norm{\partial_x \hat u(t')}_{L^2(\bT)} + | \A(\hat u^{(0)})| + \sup_{t' \in [0,t]} \norm{G_\eps(\hat u(t'))}_{L^1(\bT)}\right)^{\max\left\{\frac{3-n}{n-2}, {\vartheta}\right\}} \\
&& \times \sup_{t' \in [0,t]} \norm{\partial_x \hat u(t')}_{L^2(\bT)}  \norm{\partial_x^2 \hat u}_{L^2(Q_t)}^2,
\end{eqnarray*}
$\diff \hat\P$-almost surely. Applying Young's inequality and confining ${\vartheta}$ to the interval $(0,1)$ yields for any $\delta > 0$
\begin{eqnarray}\nonumber
\lefteqn{\sum_{k \in \Z} \int_0^t \gamma_{\hat u}^2 \int_{\bT} \sigma_k^2 (F_\eps'')^2(\hat u) \, (\partial_x \hat u)^4 \, \diff x \, \diff t'} \\
&\le& \delta \sup_{t' \in [0,t]} \norm{\partial_x \hat u(t')}_{L^2(\bT)}^2 \nonumber \\
&& + C_{\sigma,n,{\vartheta},\delta} \left(1 + | \A(\hat u^{(0)})| ^2 + \sup_{t' \in [0,t]} \norm{G_\eps(\hat u(t'))}_{L^1(\bT)}^2 + \norm{\partial_x^2 \hat u}_{L^2(Q_t)}^{\max\left\{\frac{4(n-2)}{2n-5}, \frac{4}{1-{\vartheta}}\right\}}\right), \nonumber \\
&& \label{eq:nl3vx44}
\end{eqnarray}
$\diff \hat\P$-almost surely.

In the case $n\in(3,4)$ we get with Lemma~\ref{lem:int-energy} and H\"older's inequality
\begin{eqnarray*}
\lefteqn{\abs{\int_{\bT}\left(\int_1^{\hat u} (F_\eps''(r))^2 \diff r\right) (\partial_x \hat u)^2 \, (\partial_x^2 \hat u) \, \diff x}} \\
&\stackrel{\eqref{entropy_abc}}{\le}& C_n \left(\intort \left(1+\abs{\hat u}^{2n-6}\right) \diff x\right)^{\frac 1 2} \norm{\D_x \hat u}_{L^{\infty}(\bT)}^2 \norm{\D_x^2 \hat u}_\Ltwo \\
&\stackrel{\eqref{interp_vxinf}}{\le}& C_{n,L} \left(1 + \norm{\hat u}_{L^\infty(\bT)}^{n-3}\right) \norm{\D_x \hat u}_\Ltwo \norm{\D_x^2 \hat u}_\Ltwo^2 \\
&\le & C_{n,L} \left(1+\norm{\D_x \hat u}_\Ltwo^{n-3}+ | \A(\hat u^{(0)})| ^{n-3}\right) \norm{\D_x \hat u}_\Ltwo \norm{\D_x^2 \hat u}_\Ltwo^2,
\end{eqnarray*}
$\diff \hat\P$-almost surely, where we have employed mass conservation (Remark~\ref{rem:conservation-of-mass}) and the Sobolev embedding in the last line. Hence,
\begin{eqnarray}
\lefteqn{\sum_{k \in \Z} \int_0^t \gamma_{\hat u}^2 \int_{\bT} \sigma_k^2 (F_\eps'')^2(\hat u) \, (\partial_x \hat u)^4 \, \diff x \, \diff t'} \nonumber \\
&\stackrel{\eqref{split_vx4}}{\le}& C_{n,L} \left(1+ \sup_{t' \in [0,t]} \norm{\D_x \hat u(t')}_\Ltwo^{n-2} +|\A(\hat u^{(0)})| ^{n-2}\right) \norm{\D_x^2 \hat u}_{L^2(Q_t)}^2 \nonumber \\
&\le& \delta \sup_{t' \in [0,t]} \norm{\D_x \hat u(t')}_\Ltwo^2 + C_{n,L,\delta} \left(1+ |\A(\hat u^{(0)})| ^2+ \norm{\D_x^2 \hat u}_{L^2(Q_t)}^{\frac{4}{4-n}}\right),
\label{eq:ng3vx42}
\end{eqnarray}
$\diff \hat\P$-almost surely, where we have used Young's inequality in the last estimate. Combining \eqref{eq:nl3vx44} and \eqref{eq:ng3vx42}, we end up with the bound
\begin{eqnarray}
\lefteqn{\sum_{k \in \Z} \int_0^t \gamma_{\hat u}^2 \int_{\bT} \sigma_k^2 (F_\eps'')^2(\hat u) \, (\partial_x \hat u)^4 \, \diff x \, \diff t'} \nonumber \\
&\le& \delta \sup_{t' \in [0,T]} \norm{\D_x \hat u(t')}_\Ltwo^2 + C_{\sigma,n,\delta, {\vartheta} } \left(1 +|\A(\hat u^{(0)})| ^2+ \sup_{t' \in [0,t]} \norm{G_\eps(\hat u(t'))}_{L^1(\bT)}^2 \ind_{\left[\frac 8 3,3\right]}(n)\right) \nonumber \\
&&  + C_{\sigma,n,\delta, {\vartheta} } \norm{\D_x^2 \hat u}_{L^2(Q_t)}^{\max\left\{\frac{4}{4-n},\frac{4(n-2)}{2n-5},\frac{4}{1-{\vartheta}}\right\}}, \label{eq:ng3vx44}
\end{eqnarray}
$\diff \hat\P$-almost surely, where $\delta > 0$ and ${\vartheta} \in (0,1)$ are free parameters.

\paragraph{$(\partial_x \hat u)^3$-term}
We have
\begin{eqnarray*}
\lefteqn{\sum_{k \in \Z} \int_0^t \gamma_{\hat u}^2 \int_{\bT} \left(\partial_x (\sigma_k^2)\right) \left((F_\eps^2)'''(\hat u) + 4 \left((F_\eps')^2\right)'(\hat u)\right) (\partial_x \hat u)^3 \, \diff x \, \diff t'} \\
&\stackrel{\eqref{ass-sigma-k}}{\le}& C_\sigma \int_0^t  \int_{\bT} \abs{(F_\eps^2)'''(\hat u) + 4 \left((F_\eps')^2\right)'(\hat u)} \abs{\partial_x \hat u}^3 \diff x \, \diff t'.
\end{eqnarray*}
This implies
\begin{eqnarray*}
\lefteqn{\int_{\bT} \abs{(F_\eps^2)'''(\hat u) + 4 \left((F_\eps')^2\right)'(\hat u)} \abs{\partial_x \hat u}^3 \, \diff x} \\
&\le& \int_{\bT} \abs{(F_\eps^2)'''(\hat u) + 4 \left((F_\eps')^2\right)'(\hat u)} \diff x \norm{\D_x \hat u}_{L^\infty(\bT)}^3 \\
&\stackrel{\eqref{interp_vxinf}}{\le}& C \int_{\bT} \abs{(F_\eps^2)'''(\hat u) + 4 \left((F_\eps')^2\right)'(\hat u)} \diff x \norm{\D_x \hat u}_{L^2(\bT)}^{\frac 3 2} \norm{\D_x^2 \hat u}_{L^2(\bT)}^{\frac 3 2}
\end{eqnarray*}
and Lemma~\ref{lem:int-energy-2} yields
\[
\int_{\bT} \abs{(F_\eps^2)'''(\hat u) + 4 \left((F_\eps')^2\right)'(\hat u)} \diff x \le C_{n,L} \begin{cases} \left(1 + \int_{\bT} \abs{\hat u}^{n-3} \diff x\right) & \mbox{ if } n \in [3,4), \\
\int_{\bT} G_\eps^{\frac{3-n}{n-2}}(\hat u) \, \diff x & \mbox{ if } n \in \left[ \tfrac 8 3, 3\right), \end{cases}
\]
i.e.,
\[
\int_{\bT} \abs{(F_\eps^2)'''(\hat u) + 4 \left((F_\eps')^2\right)'(\hat u)} \diff x \le C_{n,L} \begin{cases} \left(1 + \norm{\hat u}_{L^1(\bT)}^{n-3}\right) & \mbox{ if } n \in [3,4), \\
\norm{G_\eps(\hat u)}_{L^1(\bT)}^{\frac{3-n}{n-2}} & \mbox{ if } n \in \left[\tfrac 8 3, 3\right). \end{cases}
\]
We further use
\begin{equation}\label{interp_vx_v_vxx}
\norm{\partial_x \hat u}_{L^2(\bT)}^2 = \int_{\bT} (\partial_x \hat u)^2 \, \diff x = - \int_{\bT} \hat u \, \partial_x^2 \hat u \, \diff x \le \norm{\hat u}_{L^2(\bT)} \, \norm{\partial_x^2 \hat u}_{L^2(\bT)}.
\end{equation}

For $n \in [3,4)$, by employing mass conservation (Remark~\ref{rem:conservation-of-mass}) and the Sobolev embedding, we get for any $\delta > 0$
\begin{eqnarray*}
\lefteqn{\int_0^t  \int_{\bT} \abs{(F_\eps^2)'''(\hat u) + 4 \left((F_\eps')^2\right)'(\hat u)} \abs{\partial_x \hat u}^3 \diff x \, \diff t'} \\
&\le& C_{n,L} \int_0^t \left(1 + \norm{\hat u(t')}_{L^1(\bT)}^{n-3}\right) \norm{\D_x \hat u(t')}_{L^2(\bT)}^{\frac 3 2} \norm{\D_x^2 \hat u(t')}_{L^2(\bT)}^{\frac 3 2} \diff t' \\
&\stackrel{\eqref{interp_vx_v_vxx}}{\le}& C_{n,L} \int_0^t \left(1 + \norm{\hat u(t')}_{L^1(\bT)}^{n-3}\right) \norm{u(t')}_{L^2(\bT)}^{\frac 1 2} \norm{\D_x \hat u(t')}_{L^2(\bT)}^{\frac 1 2} \norm{\D_x^2 \hat u(t')}_{L^2(\bT)}^2 \diff t' \\
&\le& C_{n,L} \left(1 + \sup_{t' \in [0,t]} \norm{\D_x \hat u(t')}_{L^2(\bT)}^{n-2}+ |\A(\hat u^{(0)})| ^{n-2} \right) \norm{\D_x^2 \hat u}_{L^2(Q_t)}^2 \\
&\le& \delta  \sup_{t' \in [0,t]} \norm{\D_x \hat u(t')}_{L^2(\bT)}^2 + C_{n,L,\delta}\left(1+|\A(\hat u^{(0)})| ^2+ \norm{\D_x^2 \hat u}_{L^2(Q_t)} ^{\frac{4}{4-n}} \right) ,
\end{eqnarray*}
$\diff \hat\P$-almost surely, where we have applied Young's inequality in the last step.

For $n \in \left[\frac 8 3, 3\right)$, with an analogous reasoning we obtain for any $\delta > 0$
\begin{eqnarray*}
\lefteqn{\int_0^t  \int_{\bT} \abs{(F_\eps^2)'''(\hat u) + 4 \left((F_\eps')^2\right)'(\hat u)} \abs{\partial_x \hat u}^3 \diff x \, \diff t'} \\
&\le& C_{n,L} \int_0^t \norm{G_\eps(\hat u(t'))}_{L^1(\bT)}^{\frac{3-n}{n-2}} \norm{\D_x \hat u(t')}_{L^2(\bT)}^{\frac 3 2} \norm{\D_x^2 \hat u(t')}_{L^2(\bT)}^{\frac 3 2} \diff t' \\
&\stackrel{\eqref{interp_vx_v_vxx}}{\le}& C_{n,L} \int_0^t \norm{G_\eps(\hat u(t'))}_{L^1(\bT)}^{\frac{3-n}{n-2}} \norm{u(t')}_{L^2(\bT)}^{\frac 1 2} \norm{\D_x \hat u(t')}_{L^2(\bT)}^{\frac 1 2} \norm{\D_x^2 \hat u(t')}_{L^2(\bT)}^2 \diff t' \\
&\le& C_{n,L} \, \sup_{t' \in [0,t]} \norm{G_\eps(\hat u(t'))}_{L^1(\bT)}^{\frac{3-n}{n-2}} \left(1 + \sup_{t' \in [0,t]} \norm{\D_x \hat u(t')}_{L^2(\bT)}+ | \A(\hat u ^{(0)})|\right) \norm{\D_x^2 \hat u}_{L^2(Q_t)}^2 \\
&\le& \delta  \sup_{t' \in [0,t]} \norm{\D_x \hat u(t')}_{L^2(\bT)}^2 \\
&& + C_{n,L,\delta} \left(1+\sup_{t' \in [0,t]} \norm{G_\eps(\hat u(t'))}_{L^1(\bT)}^2 +| \A(\hat u ^{(0)})|^2+ \norm{\D_x^2 \hat u}_{L^2(Q_t)}^{\frac{4(n-2)}{2n-5}}\right),
\end{eqnarray*}
$\diff \hat\P$-almost surely, where we have applied Young's inequality again.

Altogether, for $n\in \left[\frac 8 3,4\right)$, we obtain
\begin{eqnarray}
\lefteqn{\sum_{k \in \Z} \int_0^t \gamma_{\hat u}^2 \int_{\bT} \left(\partial_x (\sigma_k^2)\right) \left((F_\eps^2)'''(\hat u) + 4 \left((F_\eps')^2\right)'(\hat u)\right) (\partial_x \hat u)^3 \, \diff x \, \diff t'} \nonumber \\
&\le&  \delta  \sup_{t' \in [0,t]} \norm{\D_x \hat u(t')}_{L^2(\bT)}^2 + C_{n,L,T,\delta} \left(1 +|\A(\hat u^{(0)})| ^2+ \sup_{t' \in [0,t]} \norm{G_\eps(\hat u(t'))}_{L^1(\bT)}^2 \ind_{\left[\frac 8 3,3\right]}(n)\right) \nonumber \\
&& + C_{n,L,T,\delta} \norm{\D_x^2 \hat u}_{L^2(Q_t)}^{\max\left\{\frac{4}{4-n},\frac{4(n-2)}{2n-5}\right\}}, \label{eq:ng3vx34}
\end{eqnarray}
$\diff \hat\P$-almost surely.

\paragraph{$(\partial_x \hat u)^2$-term}
We start out with
\begin{align*}
&\sum_{k \in \Z} \int_0^t \gamma_{\hat u}^2 \int_{\bT} \left(8 \left((\partial_x \sigma_k)^2 - \sigma_k (\partial_x^2 \sigma_k)\right) (F_\eps')^2(\hat u) + \left(\partial_x^2 (\sigma_k^2)\right) (F_\eps^2)''(\hat u)\right) (\partial_x \hat u)^2 \, \diff x \, \diff t' \\
&\quad \stackrel{\eqref{ass-sigma-k}}{\le} C_\sigma \int_0^t \int_{\bT} \left((F_\eps')^2(\hat u) + \abs{(F_\eps^2)''(\hat u)}\right) (\partial_x \hat u)^2 \, \diff x \, \diff t'.
\end{align*}
Next, we compute
\[
(F_\eps')^2(r) \stackrel{\eqref{eq:def-f-eps}}{=} \tfrac{ n^2}{4} \, r^2 \, (r^2+\eps^2)^{\frac n 2 -2} \les |r|^{n-2}
\]
and
\[
\abs{(F_\eps^2)''(r)} \stackrel{\eqref{eq:def-f-eps}}{=} \abs{n (r^2+\eps^2)^{\frac n 2 -1} + n(n-2) \, r^2 \, (r^2+\eps^2)^{\frac n 2 -2}} \le C_n(\abs{r}^{n-2} + 1),
\]
so that
\[
\int_{\bT} \left((F_\eps')^2(\hat u) + \abs{(F_\eps^2)''(\hat u)}\right) (\partial_x \hat u)^2 \, \diff x \le C_n \int_{\bT} \left(1+\abs{\hat u}^{n-2}\right) (\partial_x \hat u)^2 \, \diff x.
\]
Similarly as before, we estimate using conservation of mass (Remark~\ref{rem:conservation-of-mass}) and that $\D_x\hat u$ has zero average,
\begin{eqnarray*}
\lefteqn{\int_{\bT} \left(1+\abs{\hat u}^{n-2}\right) (\partial_x \hat u)^2 \, \diff x} \\
&\le& \left(1 + \norm{\hat u}_{L^\infty(\bT)}^{n-2}\right) \norm{ \partial_x \hat u}_{L^2(\bT)}^2 \\
&{\le}& C_{L} \left(1 + \norm{\D_x \hat u}_{L^2(\bT)}^{n-2}+|\A(\hat u ^{(0)} ) |^{n-2}\right)\norm{ \partial_x^2 \hat u}_{L^2(\bT)}^2,
\end{eqnarray*}
$\diff \hat\P$-almost surely, so that 
\begin{eqnarray*}
\lefteqn{\int_0^t \int_{\bT} \left(1+\abs{\hat u}^{n-2}\right) (\partial_x \hat u)^2 \, \diff x \, \diff t'} \\
&\le& C_{L,T} \left(1 + \sup_{t' \in [0,t]} \norm{\D_x \hat u(t')}_{L^2(\bT)}^{n-2}+|\A(\hat u ^{(0)} ) |^{n-2}\right) \left(1 + \norm{\D_x^2 \hat u}_{L^2(Q_t)}^2\right),
\end{eqnarray*}
$\diff \hat\P$-almost surely. Hence, by Young's inequality we arrive at
\begin{align}
&\sum_{k \in \Z} \int_0^t \gamma_{\hat u}^2 \int_{\bT} \left(8 \left((\partial_x \sigma_k)^2 - \sigma_k (\partial_x^2 \sigma_k)\right) (F_\eps')^2(\hat u) + \left(\partial_x^2 (\sigma_k^2)\right) (F_\eps^2)''(\hat u)\right) (\partial_x \hat u)^2 \, \diff x \, \diff t' \nonumber \\
& \quad \le \delta\sup_{t' \in [0,t]}\norm{\D_x \hat u(t')}_{\Ltwo}^2+ C_{\sigma,L,T,\delta} \left(1+|\A(\hat u ^{(0)} ) |^2+\norm{\D_x^2 \hat u}_{L^2(Q_t)}^{\frac{4}{4-n}}\right), \label{eq:ng3vx24}
\end{align} 
$\diff \hat\P$-almost surely, for any $\delta > 0$.

\paragraph{$(\partial_x \hat u)^0$-term}
We first use
\[
\sum_{k \in \Z} \int_0^t \gamma_{\hat u}^2 \int_{\bT} \left(4 \sigma_k \, \partial_x^4 \sigma_k - \partial_x^4 (\sigma_k^2)\right) F_\eps^2(\hat u) \, \diff x \, \diff t' \stackrel{\eqref{ass-sigma-k}}{\le} C_\sigma \int_0^t \int_{\bT} F_\eps^2(\hat u) \, \diff x \, \diff t'
\]
and
\[
\int_0^t  \int_{\bT} F_\eps^2(\hat u) \, \diff x \, \diff t' \stackrel{\eqref{eq:def-f-eps}}{\le} C_{L,T} \left(1 + \int_0^t  \int_{\bT} \abs{\hat u}^n \diff x \, \diff t'\right).
\]
Then, we estimate, using conservation of mass (Remark~\ref{rem:conservation-of-mass}) and the Sobolev embedding,
\begin{eqnarray*}
\int_{\bT} \abs{\hat u}^n \diff x &\le& C_{L} \norm{\hat u}_{L^\infty(\bT)}^n \le C_{L,n} \left(1 + \norm{\D_x \hat u}_{L^2(\bT)}^n+ |\A(\hat u ^{(0)} ) |^n \right) \\
&\le& C_{L,n} \left(1 + \norm{\D_x \hat u}_{L^2(\bT)}^{n-2} \| \D^2_x \hat u \|_{L^2(\bT)}^2 + |\A(\hat u ^{(0)} ) |^n\right) 
\end{eqnarray*}
$\diff \hat\P$-almost surely, where we have also employed that $\D_x \hat u$ has zero average. Integrating in time yields
\[
\int_0^t  \int_{\bT} \abs{\hat u}^n \diff x \, \diff t' \le C_{L,T,n} \left(1 + \sup_{t' \in [0,t]} \norm{\D_x \hat u(t')}_{L^2(\bT)}^{n-2} \norm{\D_x^2 \hat u}_{L^2(Q_t)}^2+|\A(\hat u ^{(0)} ) |^n \right),
\]
$\diff \hat\P$-almost surely, so that by Young's inequality it follows that for any $\delta > 0$,
\begin{align}
&\sum_{k \in \Z} \int_0^t \gamma_{\hat u}^2 \int_{\bT} \left(4 \sigma_k \, \partial_x^4 \sigma_k - \partial_x^4 (\sigma_k^2)\right) F_\eps^2(\hat u) \, \diff x \, \diff t' \nonumber \\
& \quad \le \delta\sup_{t' \in [0,t]}\norm{\D_x \hat u(t')}_{\Ltwo}^2+ C_{\sigma,L,T,\delta} \left(1+|\A(\hat u ^{(0)} ) |^n +\norm{\D_x^2 \hat u}_{L^2(Q_t)}^{\frac{4}{4-n}}\right), \label{eq:ng0vx24}
\end{align} 
$\diff \hat\P$-almost surely.

\paragraph{Closing the estimate}
Inserting all the previous estimates \eqref{eq:ng3vx44}, \eqref{eq:ng3vx34}, \eqref{eq:ng3vx24}, and \eqref{eq:ng0vx24} in \eqref{energy_eps_0} and choosing $\delta$ sufficiently small and an appropriate ${\vartheta}$, we arrive for $n \in \left[\frac 8 3,4\right)$ at 
  \begin{align}\label{eq:before-closing}
    \begin{split}
&\norm{\D_x \hat u(t)}_{L^2(\bT)}^2 + \int_0^t \int_{\bT} F_\eps^2(\hat u) \, (\D_x^3 \hat u)^2 \, \diff x \, \diff t' \\
&\qquad\le 2 \norm{\D_x \hat u^{(0)}}_{L^2(\bT)}^2 + 2 \abs{M(t)} \\
      & \qquad\quad + C_{q,\sigma,n,L,T,\delta} \left(1 + |\A(\hat u ^{(0)} ) |^n+\sup_{t' \in [0,t]} \norm{G_\eps(\hat u(t'))}_{L^1(\bT)}^{2} \ind_{\left[\frac 8 3,3\right]}(n) + \norm{\D_x^2 \hat u}_{L^2(Q_t)}^{4q}\right),
      \end{split}
\end{align}
$\diff \hat\P$-almost surely, where $q \ge \max\left\{\frac{1}{4-n},\frac{n-2}{2n-5}\right\}$ and $q > 1$ with a constant $C_{q,\sigma,n,,L,T,\delta} < \infty$. Here,
\[
M(t) :=\sum_{ k \in \mathbb{Z}} \int_0^t \gamma_{\hat u} \int_{\bT} \sigma_k F_\eps(\hat u) \, \D^3_x \hat u \, \diff x \, \diff \hat\beta^k(t')
\]
denotes the martingale in the last line of \eqref{energy_eps_0}. Let us set as usual 
$$
\tau_m = \inf \{ t >0 :  \int_0^t \int_{\bT}  F_\eps (\hat u ) (\D_x^3 \hat u )^2 \, \diff x \, \diff t' >m \} \wedge T 
$$
We now discard the second term on the left-hand side of \eqref{eq:before-closing}, we take suprema in time up to $ T \wedge \tau_m $, we raise to the power $\frac p 2$, and we take expectation to obtain with help of the Burkholder-Davis-Gundy inequality
\begin{eqnarray}\nonumber   
\lefteqn{\hat\E \sup_{t \in [0,T]} \norm{\D_x \hat u(t\wedge \tau_m)}_{L^2(\bT)}^p} \\
&\le& C_p \left(\hat\E \norm{\D_x \hat u^{(0)}}_{L^2(\bT)}^p + \hat\E \langle M \rangle_{T \wedge \tau_m} ^{\frac p 4}\right) \nonumber \\
&& + C_{p,q,\sigma,n,L,T} \left(1 + \hat\E \sup_{t \in [0,T]} \norm{G_\eps(\hat u(t))}_{L^1(\bT)}^{p} \ind_{\left[\frac 8 3,3\right]}(n) + \norm{\D_x^2 \hat u}_{L^2(Q_T)}^{2pq}\right), \nonumber \\
&& \label{eq:before-closing-1}
\end{eqnarray}
$\diff \hat\P$-almost surely. Next, notice that
\begin{equs}\label{eq:qud}
\E \langle M \rangle_{T \wedge \tau_m}^{\frac p 4} \stackrel{\eqref{eq:winf-sigma-k}}{\le} C_p \, \hat\E \left( \int_0^{T \wedge \tau_m} \int_{\bT} F_\eps^2(\hat u) \, (\D_x^3 \hat u)^2 \, \diff x \, \diff t  \right)^{\frac p 4}.
\end{equs}

Now we go back to \eqref{eq:before-closing}, we discard the first term at the left-hand side and we conclude that 
\begin{align*}    
& \E \left(\int_0^{T \wedge \tau_m} \int_{\bT} F_\eps^2(\hat u) \, (\D_x^3 \hat u)^2 \, \diff x \, \diff t \right)^{\frac p 2} \\
&\ \ \quad\le \ \ C_p \left(\hat\E \norm{\D_x \hat u^{(0)}}_{L^2(\bT)}^p + \hat\E \langle M \rangle_{T \wedge \tau_m}^{\frac p 4}\right) \\
&\ \ \quad\phantom{\le} \ \ + C_{p,q,\sigma,n,L,T} \left(1 +|\A(\hat u ^{(0)} ) |^n+ \hat\E \sup_{t \in [0,T]} \norm{G_\eps(\hat u(t))}_{L^1(\bT)}^{p} \ind_{\left[\frac 8 3,3\right]}(n) + \norm{\D_x^2 \hat u}_{L^2(Q_T)}^{2pq}\right) \\
&\quad\stackrel{\eqref{eq:qud}}{\le} C_p \left(\hat\E \norm{\D_x \hat u^{(0)}}_{L^2(\bT)}^p + \E \left(\int_0^{T \wedge \tau_m} \int_{\bT} F_\eps^2(\hat u) \, (\D_x^3 \hat u)^2 \, \diff x \, \diff t \right)^{\frac p 4}\right) \\
&\quad \phantom{\stackrel{\eqref{eq:qud}}{\le}} + C_{p,q,\sigma,n,L,T} \left(1 +|\A(\hat u ^{(0)} ) |^{\frac{np}{2}}+ \hat\E \sup_{t \in [0,T]} \norm{G_\eps(\hat u(t))}_{L^1(\bT)}^{p} \ind_{\left[\frac 8 3,3\right]}(n) + \norm{\D_x^2 \hat u}_{L^2(Q_T)}^{2pq}\right) ,
\end{align*}
$\diff \hat\P$-almost surely. From this it follows with Young's inequality that 
\begin{eqnarray*}    
\lefteqn{\E \left(\int_0^{T \wedge \tau_m} \int_{\bT} F_\eps^2(\hat u) \, (\D_x^3 \hat u)^2 \, \diff x \, \diff t \right)^{\frac p 2}} \\  
&\le&C_{p,q,\sigma,n,L,T} \left(1 +|\A(\hat u ^{(0)} ) |^{\frac{np}{2}}+ \hat\E \norm{\D_x \hat u^{(0)}}_{L^2(\bT)}^p\right) \\
&& + C_{p,q,\sigma,n,L,T} \left(\hat\E \sup_{t \in [0,T]} \norm{G_\eps(\hat u(t))}_{L^1(\bT)}^{p} \ind_{\left[\frac 8 3,3\right]}(n) + \norm{\D_x^2 \hat u}_{L^2(Q_T)}^{2pq}\right),
\end{eqnarray*}
which combined with \eqref{eq:before-closing-1} and \eqref{eq:qud} yields by virtue of Fatou's lemma
\begin{eqnarray*}    
\lefteqn{\hat\E \left[\sup_{t \in [0,T]} \norm{\D_x \hat u(t)}_{L^2(\bT)}^p + \left(\int_0^T \int_{\bT} F_\eps^2(\hat u) \, (\D_x^3 \hat u)^2 \, \diff x \, \diff t \right)^{\frac p 2}\right]} \\  
&\le&C_{p,q,\sigma,n,L,T} \left(1 + |\A(\hat u ^{(0)} ) |^{\frac{np}{2}}+\hat\E \norm{\D_x \hat u^{(0)}}_{L^2(\bT)}^p\right) \\
&& + C_{p,q,\sigma,n,L,T} \left(\hat\E \sup_{t \in [0,T]} \norm{G_\eps(\hat u(t))}_{L^1(\bT)}^{p} \ind_{\left[\frac 8 3,3\right]}(n) + \norm{\D_x^2 \hat u}_{L^2(Q_T)}^{2pq}\right),
\end{eqnarray*}
which was stated in \eqref{eq:energy-est-eps}.
\end{proof}
%

\subsection{Passage to the Limit to remove the Cut Off}
In this section we consider the SPDE
\begin{subequations}\label{eq:approx-equation-2}
\begin{eqnarray}
\diff u_{\eps} &=& \D_x\left(-F_\eps^2(u_{\eps}) \D_x^3u_{\eps}\right) \diff t +\frac 1 2 \sum_{k \in \Z}\partial_x \left(\sigma_k  F_\eps'(u_{\eps}) \partial_x(\sigma_k  F_\eps(u_{\eps}))\right) \diff t \nonumber \\           
&& + \sum_{k \in \Z} \left(\partial_x\left(\sigma_k  F_\eps(u_{\eps})\right)\right) \diff \beta^k,
\label{eq:equation-approx-equation-2}\\
u_{\eps}(0,\cdot) &=& u^{(0)}.
\end{eqnarray}
\end{subequations}
The definition of a  weak solution $\{ (\check \Omega, \check\cF,\check\F,\check \P), \ (\check\beta_k)_{k \in \Z},\ \check u^{(0)},  \check u \}$ of equation 
\eqref{eq:approx-equation-2} is covered by Definition~\ref{def:weak-sol-eps-R}
taking for $g_R$ the function $g_\infty\equiv 1.$

\begin{proposition}\label{prop:weak-eps}
Suppose that $n \in \left[\frac 8 3,4\right)$, $\mathfrak{p} \ge n+2$, $\eps\in (0,1)$,  and $q > 1$ satisfying $q \ge \max\left\{\frac{1}{4-n},\frac{n-2}{2n-5}\right\}$. Suppose that  $u^{(0)} \in L^{\mathfrak{p}}\left(\Omega,\mathcal{F}_0,\P;H^1(\bT)\right)$ such that 
\begin{equs}
\mathcal{K}(u^{(0)}, \mathfrak{p}, q, \eps):= \E\abs{\A(u^{(0)})}^{2\mathfrak{p}q} + \E \norm{G_\eps (u^{(0)})}_{L^1(\bT)}^{\mathfrak{p}q}+ \E \norm{\D_x u^{(0)}}^{\mathfrak{p}} < \infty. 
\end{equs}
Then there exists a weak solution $\{ (\check \Omega, \check\cF,\check\F,\check \P), \ (\check\beta_k)_{k \in \Z},\ \check u^{(0)}_\eps,  \check u_\eps  \}$ to \eqref{eq:approx-equation-2} in the sense of Definition~\ref{def:weak-sol-eps-R}, satisfying 
\begin{align}
& \check \E \left[\sup_{t \in [0,T]} \norm{\D_x \check u_\eps(t)}_{L^2(\bT)}^{\mathfrak{p}} + \norm{F_\eps(\check u_\eps) \D_x^3 \check u_\eps}_{L^2(Q_T)}^{\mathfrak{p}} + \sup_{t \in [0,T]} \norm{G_\eps(\check u_\eps(t))}_{L^1(\bT)}^{\mathfrak{p}q} + \norm{\D_x^2\check  u_\eps}_{L^2(Q_T)}^{2\mathfrak{p}q}\right] \nonumber \\  
& \quad \le C \left(1 + \mathcal{K}(u^{(0)}, \mathfrak{p}, q, \eps)\right), \label{eq:energy-entropy}
\end{align}
where  $C < \infty$ is a constant depending only on $\mathfrak{p}$, $q$, $\sigma = (\sigma_k)_{k \in \Z}$, $n$, $L$, and $T$.
\end{proposition}

\begin{proof}

Let $\hat u_{\eps, R}$ be a weak solution of \eqref{eq:equation-eps-R-dependent}. By Lemmata~\ref{lem:entropy} and \ref{lem:energy-estimate} we have 
\begin{align*}
& \hat \E \left[\sup_{t \in [0,T]} \norm{\D_x  \hat u_{\eps, R}(t)}_{L^2(\bT)}^{\mathfrak{p}} + \norm{F_\eps( \hat u_{\eps, R}) \D_x^3  \hat u_{\eps, R}}_{L^2(Q_T)}^{\mathfrak{p}}\right] \\
& + \hat\E\left[\sup_{t \in [0,T]} \norm{G_\eps( \hat u_{\eps, R}(t))}_{L^1(\bT)}^{\mathfrak{p}q} + \norm{\D_x^2 \hat  u_{\eps, R}}_{L^2(Q_T)}^{2\mathfrak{p}q}\right] \le  C_{p,q,\sigma,n,L,T} \left(1 + \mathcal{K}(u^{(0)}, \mathfrak{p}, q, \eps)\right), 
\end{align*}
and notice that the constant does not depend on $R$. From this estimate, the construction of a weak solution $\{ (\check \Omega, \check\cF,\check\F,\check \P), \ (\check\beta_k)_{k \in \Z},\ \check u^{(0)},  \check u \}$ is very similar to the construction in Proposition \ref{prop:weak-eps-r} (in fact, easier) and is left to the reader. Estimate \eqref{eq:energy-entropy} follows from the above estimate and Fatou's lemma. 
\end{proof}
\begin{remark}\label{rem:entropychain}
  The right-hand side of  \eqref{eq:energy-entropy} can be formulated independently of $\eps$, just noting the inequality
\[
G_\eps(r) \stackrel{\eqref{eq:def-G-eps}}{=} \int_r^\infty \int_{r'}^\infty \frac{1}{F_\eps^2(r'')} \, \diff r'' \, \diff r' \stackrel{\eqref{eq:def-f-eps}}{\le} \int_r^\infty \int_{r'}^\infty \frac{1}{F_0^2(r'')} \, \diff r'' \, \diff r' \stackrel{\eqref{eq:def-G-eps}}{=} G_0(r) \qedhere
\]
and choosing $\eps=0$ in $\mathcal K(u^{(0)}, p,q,\eps).$
\end{remark}
  
\section{The Degenerate Limit\label{sec:eps-limit}}
In order to prove Theorem~\ref{th:existence}, we first prove additional regularity in time in order to obtain $\diff \tilde{\P}$-almost surely uniform convergence in the limit $\eps\searrow0$ using a version of Prokhorov's theorem (cf.~\cite[Theorem~2]{Jakubowski}) and a compactness argument. Subsequently, we prove that \eqref{apriori-main} is recovered in this limit by employing the energy-entropy estimate, Proposition~\ref{prop:weak-eps} . The proof is concluded by showing that the weak formulation \eqref{eq:exact-weak} is valid, which follows by applying \cite[Proposition A.1]{HOF2} to characterize the martingale. 

For $\eps \in \{ n^{-1} \}_{n=1}^\infty$,  we  denote by  $\{ (\check \Omega_\eps, \check\cF_\eps,\check\F_\eps,\check \P_\eps), \ (\check\beta^k_\eps)_{k \in \Z},\ \check u^{(0)}_\eps,  \check u_\eps \}$  the weak solution of  \eqref{eq:equation-approx-equation-2} constructed in Proposition  \ref{prop:weak-eps}. In order to drop the $\eps$-dependence from the probability space we will be considering $ ( (\check\beta^k_\eps)_{k \in \Z},\ \check u^{(0)}_\eps,  \check u_\eps)$ on a common probability space given by 
\begin{equs}
 (\check \Omega, \check\cF,\check\F,\check \P):= \prod _{\eps}  (\check \Omega_\eps, \check\cF_\eps,\check\F_\eps,\check \P_\eps).
\end{equs}

\subsection{Compactness}\label{sec:compactness}
The reasoning of this section uses techniques of \cite{FischerGruen2018} and of \cite[\S4]{GessGnann2020}. 
\begin{lemma}[Regularity in time]\label{lem:reg-time} Suppose that $T\in(0,\infty)$, $\eps\in(0,1]$, $n\in[8/3,4)$, $\fp>1$, $q > 1$ satisfying $q \ge \max\left\{\frac{1}{4-n},\frac{(n-2)}{2n-5}\right\}$, and 
\[
u^{(0)}\in L^{(n+2)\fp}\left(\Omega,\F_{0},\P;H^{1}(\bT)\right)
\]
such that $u^{(0)}\ge0$ $\diff \P$-almost surely, $\E\abs{\Aunull}^{2(n+2)\fp q}<\infty$, and $\E\norm{G_{0}\left(u^{(0)}\right)}_{L^{1}(\bT)}^{(n+2)\fp q}<\infty$. Then, the weak solutions $\check{u}_{\eps}$ constructed in Proposition~\ref{prop:weak-eps} satisfy for any  $\fpp\in[1,2\fp)$, 
\begin{subequations}
\label{eq:reg-time} 
\begin{equation}
\check u_ \eps\in L^{\fpp}\left(\check\Omega,\check \cF,\check \P;C^{\frac 1 4}\left([0,T];\Ltwo\right)\right)\label{eq:reg-time-1}
\end{equation}
with 
\begin{equation}\begin{split}
&\norm{\check{u}_\eps}_{L^{\fpp}(\check \Omega;C^{\frac 1 4}([0,T];\Ltwo))}\\
&\le C\,\left[\E\left(1+\abs{\Aunull}^{2(n+2)\fp q}+\norm{G_{0}(u^{(0)})}_{L^{1}(\bT)}^{(n+2)\fp q}+\norm{\D_{x}u^{(0)}}_{L^{2}(\bT)}^{(n+2)\fp}\right)\right]^{\frac{1}{2\fp}},\label{eq:reg-time-2}
\end{split}\end{equation}
\end{subequations}
where $C$ is a constant depending only on $\fp,\fpp,q,(\sigma_k)_{k\in\Z},L,$ and $T$.
 \end{lemma} 
\begin{proof}
Starting from the weak formulation (cf.~Definition~\ref{def:weak-sol-eps-R}~\eqref{item:satisfying-the-equ})
\begin{align}
\begin{split} & \inner{\check{u}_\eps(t_{2})-\check{u}_\eps(t_{1}),\varphi}_{\Ltwo}+\int_{t_{1}}^{t_{2}}\intort\Fepsq(\check{u}_\eps)\D_{x}^{3}\check{u}_\eps\D_{x}\varphi \ \diff x \, \diff t\\
 & \qquad+\tfrac{1}{2}\sum_{k\in\Z}\int_{t_{1}}^{t_{2}}\sigma_{k}\Fepsp(\check{u}_\eps)\D_{x}\left(\sk\Feps(\check{u}_\eps)\right)\D_{x}\varphi \, \diff x \, \diff t = \sum_{k\in\Z}\int_{t_{1}}^{t_{2}}\intort\D_{x}\left(\sk\Feps(\check{u}_\eps)\right)\varphi \, \diff x \, \diff\check \beta^{k}_\eps\qquad
\end{split}
\label{eq:greg-1}
\end{align}
for all $\varphi\in H^{1}(\bT)$, $t_{1},t_{2}\in[0,T]$ with $t_{1}\leq t_{2}$ and $\diff \check \P$-almost surely, we obtain, by an approximation argument based on the separability of $H^1(\bT)$ that the $\hat\P$-zero set can be chosen independently of $\varphi$, that is, $\diff \check \P$-almost surely, 
\begin{align*}
\begin{split} & \left(\check{u}_\eps(t_{2})-\check{u}_\eps(t_{1}),\varphi\right)_{\Ltwo}+\int_{t_{1}}^{t_{2}}\intort\Fepsq(\check{u}_\eps)\D_{x}^{3}\check{u}_\eps\D_{x}\varphi \, \diff x \, \diff t\\
 & \qquad+\tfrac{1}{2}\sum_{k\in\Z}\int_{t_{1}}^{t_{2}}\sigma_{k}\Fepsp(\check{u}_\eps)\D_{x}\left(\sk\Feps(\check{u}_\eps)\right)\D_{x}\varphi \, \diff x \, \diff t\leq\sup_{\norm{\psi}_{\Ltwo}\leq1}\abs{(I_{\eps}(t_{2})-I_{\eps}(t_{1}),\psi)}
\end{split}
\end{align*}
for all $\varphi\in H^{1}(\bT)$ with $\norm{\varphi}_{\Ltwo}=1.$ Here, we have used the abbreviation 
\begin{equation}
I_{\eps}(t):=\sum_{k\in\Z}\int_{0}^{t}\D_{x}\left(\sk\Feps(\check{u}_\eps)\right) \diff \check \beta^{k}.\label{eq:greg-2}
\end{equation}
Following the lines of the proof of Lemma~4.10 in \cite{FischerGruen2018}, the choice $\varphi:=\frac{\check{u}_\eps(t_{2})-\check{u}_\eps(t_{1})}{\norm{\check{u}_\eps(t_{2})-\check{u}_\eps(t_{1})}_{\Ltwo}}$ and Young's inequality imply that there is a finite constant $C$ independent of $\eps>0$ such that, $\diff \hat\P$-almost surely, 
\begin{align}
& \norm{\check{u}_\eps(t_{2})-\check{u}_\eps(t_{1})}_{\Ltwo}^2 \nonumber \\
& \quad \leq C\Bigg(\abs{\int_{t_{1}}^{t_{2}}\intort\Fepsq(\check{u}_\eps)\D_{x}^{3}\check{u}_\eps\D_{x}\left(\check{u}_\eps(t_{2})-\check{u}_\eps(t_{1})\right) \diff x \, \diff t} \nonumber \\
& \quad \qquad+\tfrac{1}{2}\abs{\sum_{k\in\Z}\int_{t_{1}}^{t_{2}}\intort\sk\Fepsp(\check{u}_\eps)\D_{x}\left(\check{u}_\eps(t_{2})-\check{u}_\eps(t_{1})\right)\D_{x}\left(\sk\Feps(\check{u}_\eps)\right) \diff x \, \diff t}\Bigg) \nonumber \\
& \quad \qquad+\norm{I_{\eps}(t_{2})-I_{\eps}(t_{1})}_{\Ltwo}^2 =: R_{1}^2(t_1,t_2)+R_{2}^2(t_1,t_2)+R_{3}^2(t_1,t_2). \label{eq:greg-3}
\end{align}
By \cite[Theorem~3.2~(vi)]{OV20}, for all $0<\sigma<1/2$, $1\leq \fpp < 2\fp$,
\begin{eqnarray*}
\lefteqn{\norm{I_{\eps}}_{L^{\fpp}(\check \Omega;C^{\frac{1}{2}-\sigma}\left([0,T];L^{2}(\bT)\right))}}\\
 & \le & C_{T,\fpp,\fp,\sigma}\|(\D_{x}\left(\sk\Feps(\check{u}_\eps)\right))_{k\in\Z}\|_{L^{2\fp}(\Omega;L^{\infty}(0,T;L_{2}(\ell^2(\Z);L^2(\bT))))}\\
 & = & C\,\left[\check \E\ \sup_{t\in[0,T]}\left(\sum_{k\in\Z}\|\D_{x}(\sigma_{k}F_{\eps}
(\check{u}_\eps))\|_{L^{2}(\bT)}^{2}\right)^{\fp}\right]^{\frac{1}{2\fp}}\\
 & \le & C\,\left[\check\E\ \sup_{t\in[0,T]}\left(\sum_{k\in\Z}\|\D_{x}\sigma_{k}\|_{L^{\infty}(\bT)}^{2}\|F_{\eps}(\check{u}_\eps)\|_{L^{2}(\bT)}^{2}\right)^{\fp}\right]^{\frac{1}{2\fp}}\\
 && + C\,\left[\check\E\ \sup_{t\in[0,T]}\left(\sum_{k\in\Z}\|\sigma_{k}\|_{L^{2}(\bT)}^{2}\|F_{\eps}'(\check{u}_\eps)\|_{L^{\infty}(\bT)}^{2}\|\D_{x}\check{u}_\eps\|_{L^{2}(\bT)}^{2}\right)^{\fp}\right]^{\frac{1}{2\fp}}\\
 & \stackrel{\eqref{eq:reg-sigma-k}}{\le} & C\,\left[\check\E\ \sup_{t\in[0,T]}\left(\|F_{\eps}(\check{u}_\eps)\|_{L^{2}(\bT)}^{2\fp}+\|F_{\eps}'(\check{u}_\eps)\|_{L^{\infty}(\bT)}^{2\fp}\|\D_{x}\check{u}_\eps\|_{L^{2}(\bT)}^{2\fp}\right)\right]^{\frac{1}{2\fp}}.
\end{eqnarray*}
We then use that
\begin{align*}
\|F_{\eps}(\check{u}_\eps)\|_{L^{2}(\bT)}^{2\fp}=\left(\int_{\bT}F_{\eps}^{2}(\check{u}_\eps)\right)^{\fp} & \stackrel{\eqref{eq:def-f-eps}}{\le}C\left(\abs{\A(\check u^{(0)}_\eps)}^{n}+\|\D_{x}\check{u}_\eps\|_{L^{2}}^{n}+1\right){}^{\fp}\\
 & \le C\left(1+\abs{\A(\check u^{(0)}_\eps)}^{\fp n}+\|\D_{x}\check{u}_\eps\|_{L^{2}}^{\fp n}\right)
\end{align*}
and
\begin{align*}
\|F_{\eps}'(\check{u}_\eps)\|_{L^{\infty}(\bT)}^{2\fp}\|\D_{x}\check{u}_\eps\|_{L^{2}(\bT)}^{2\fp} & \stackrel{\eqref{eq:def-f-eps}}{\le} \left(1+\|\check{u}_\eps\|_{L^{\infty}(\bT)}\right)^{2\left(\frac{n}{2}-1\right)\fp}\|\D_{x}\check{u}_\eps\|_{L^{2}(\bT)}^{2\fp}\\
 & \le C\left(1+\abs{\A(\check u^{(0)}_\eps)}^{\fp n}+\|\D_{x}\check{u}_\eps\|_{L^{2}}^{\fp n}\right)
\end{align*}
to get for $R_3$ the estimate
\begin{equation}
\label{eq:stoch_integral}
\check\E\sup_{t_1,t_2\in[0,T]}\left(\frac{R_3(t_1,t_2)}{|t_2-t_1|^{(1/2-\sigma)}}\right)^\fpp
\le  C\,\left[\check\E\ \sup_{t\in[0,T]}\left(1+\abs{\A(\check u^{(0)}_\eps)}^{\fp n}+\|\D_{x}\check{u}_\eps(t)\|_{L^{2}}^{\fp n}\right)\right]^{\frac{\fpp}{2\fp}},
\end{equation}
valid for any $\sigma\in\left(0,\frac 1 2\right)$.
For $R_{1}$, we estimate using $\Feps(r) \stackrel{\eqref{eq:def-f-eps}}{\le} (r^{2}+\eps^{2})^{\frac n 4}$, 
\begin{align}
\begin{split}& \abs{R_{1}(t_1,t_2)} \\
& \quad \leq C\Bigg(\int_{t_{1}}^{t_{2}}\intort\Fepsq(\check{u}_\eps) \left(\D_{x}(\check{u}_\eps(t_{2})-\check{u}_\eps(t_{1}))\right)^{2} \diff x \, \diff t\Bigg)^{\frac 1 4}\\
 & \quad \qquad\times\Bigg(\inttime\intort\Fepsq(\check{u}_\eps) \, (\D_{x}^{3}\check{u}_\eps)^{2} \, \diff x \, \diff t\Bigg)^{\frac 1 4}\\
 & \quad \leq C |t_{2}-t_{1}|^{\frac 1 4}\left(1+|\A(\check u_\eps^{(0)})|^{\frac{n+2}{4}}+\sup_{t\in[0,T]}\norm{\D_{x}\check{u}_\eps}_{\Ltwo}^{\frac{n+2}{4}}\right)\\
 & \quad \qquad\times\left(\int_{0}^{T}\intort\Fepsq(\check{u}_\eps) \, (\D_{x}^{3}\check{u}_\eps)^{2} \, \diff x \, \diff t\right)^{\frac 1 4} \\
   & \quad \leq C |t_{2}-t_{1}|^{\frac 1 4} \\
 & \quad \qquad\times\left(1+|\A(\check u_\eps^{(0)})|^{\frac{n+2}{2}}+\sup_{t\in[0,T]}\norm{\D_{x}\check{u}_\eps}_{\Ltwo}^{\frac{n+2}{2}}+\left( \int_{0}^{T}\intort\Fepsq(\check{u}_\eps) \, (\D_{x}^{3}\check{u}_\eps)^{2} \, \diff x \, \diff t\right)^{\frac 1 2}\right).
\end{split}
\label{eq:greg-6}
\end{align}
Hence,
\begin{align}
\begin{split}
&\check \E\left(\sup_{t_1,t_2\in[0,T]}\frac{|R_{1}(t_1,t_2)|}{|t_2-t_1|^{\frac{1}{4}}}\right)^{\fpp} \\
& \leq C \, \check \E\left(1+|\A(\check u_\eps^{(0)})|^{\frac{(n+2)\fpp}{2}}+\sup_{t\in[0,T]}\norm{\D_{x}\check{u}_\eps}_{\Ltwo}^{\frac{(n+2)\fpp}{2}}+\left( \int_{0}^{T}\intort\Fepsq(\check{u}_\eps) \, (\D_{x}^{3}\check{u}_\eps)^{2} \, \diff x \, \diff t\right)^{\frac{\fpp}{2}}\right).
\end{split}
\label{eq:greg-6-1}
\end{align}
The term $R_{2}^2$ is split as follows
\begin{align}
2R_{2}^2(t_1,t_2) & =\sum_{k\in\Z}\inttime\intort \sigma_{k}^{2} \Fepsp(\check{u}_\eps)^{2}\D_{x}\left(\check{u}_\eps(t_{2})-\check{u}_\eps(t_{1})\right)\D_{x}\check{u}_\eps(t) \, \diff x \, \diff t \nonumber \\
 & \qquad+\sum_{k\in\Z} \inttime\intort \sigma_{k}(\D_x\sigma_{k}) \Fepsp(\check{u}_\eps)\Feps(\check{u}_\eps)\D_{x}\left(\check{u}_\eps(t_{2})-\check{u}_\eps(t_{1})\right) \diff x \, \diff t \nonumber \\
 & =:R_{21}^2+R_{22}^2. \label{eq:greg-61}
\end{align}
For $R_{21}$, we estimate
\begin{align}
|R_{21}| &\stackrel{\eqref{eq:def-f-eps}}{\le} C\left(\sum_{k\in\Z} \norm{\sigma_{k}}_{\Linfty}^{2} \inttime\sup_{t\in[0,T]}\left(1+\norm{\check{u}_\eps(t)}_{\Linfty}^{n-2}\right)\sup_{t\in[0,T]}\norm{\D_{x}\check{u}_\eps(t)}_{\Ltwo}^{2}\right)^{\frac 1 2} \nonumber \\
 &\stackrel{\eqref{eq:reg-sigma-k}}{\le} C T^{\frac 1 4}\abs{t_{2}-t_{1}}^{\frac 1 4} \left(1+\abs{\Aunull}^{\frac{n}{2}}+\sup_{t\in[0,T]}\norm{\D_{x}\check{u}_\eps(t)}_{\Ltwo}^{\frac{n}{2}}\right). \label{eq:greg-7}
\end{align}
Finally, 
\begin{align}
\begin{split}\abs{R_{22}} & \leq C\left(\sum_{k\in\Z} \norm{\sigma_k}_{\Linfty} \norm{\D_x \sigma_k}_{\Linfty} \inttime\intort \abs{\Fepsp(\check{u}_\eps)\Feps(\check{u}_\eps)\D_{x}\left(\check{u}_\eps(t_{2})-\check{u}_\eps(t_{1})\right)} \diff x \, \diff t\right)^{\frac 1 2}\\
 &\stackrel{\eqref{eq:reg-sigma-k}, \eqref{eq:def-f-eps}}{\le} C T^{\frac 1 4} \abs{t_{2}-t_{1}}^{\frac 1 4} \left(1+\abs{\A(\check u_\eps^{(0)})}^{\frac{n}{2}}+\sup_{t\in[0,T]}\norm{\D_{x}\check{u}_\eps(t)}_{\Ltwo}^{\frac{n}{2}}\right).
\end{split}
\label{eq:greg-8}
\end{align}
Hence, we get
\begin{align}
\begin{split}
\check \E\left(\sup_{t_1,t_2\in[0,T]}\frac{|R_{2}(t_1,t_2)|}{|t_2-t_1|^{\frac{1}{4}}}\right)^{\fpp}
\le &
C\check \E\left(1+\abs{\A(\check u_\eps^{(0)})}^{\frac{n\fpp}{2}}+\sup_{t\in[0,T]}\norm{\D_{x}\check{u}_\eps(t)}_{\Ltwo}^{\frac{n\fpp}{2}}\right).
\end{split}
\label{eq:greg-61-1}
\end{align}
Altogether, combining \eqref{eq:stoch_integral}, \eqref{eq:greg-6-1},  \eqref{eq:greg-61-1} and choosing $\sigma = \frac 1 4$, we get  
\begin{eqnarray*}
 \lefteqn{\|\check{u}_\eps\|_{L^{\fpp}(\check \Omega;C^{\frac{1}{4}}
([0,T];\Ltwo))}^{\fpp}} \\
&\le&  C\,\left[\check \E\ \sup_{t\in[0,T]}\left(1+\abs{\Aunull}^{\fp n}+\|\D_{x}\check{u}_{\eps}(t)\|_{L^{2}}^{\fp n}\right)\right]^{\frac{\fpp}{2\fp}}\\
&&
+C\left[\check \E\left(1+\abs{\A(\check u_\eps^{(0)})}^{(n+2)\fp}+\sup_{t\in[0,T]}\norm{\D_{x}\check{u}_\eps}_{\Ltwo}^{(n+2)\fp}\right)\right]^{\frac{\fpp}{2\fp}} \\
&&
+C\left[\check \E\left(\int_{0}^{T}\intort\Fepsq(\check{u}_\eps) \, (\D_{x}^{3}\check{u}_\eps)^{2} \, \diff x \, \diff t\right)^{\fp}\right]^{\frac{\fpp}{2\fp}} \\
&& 
+C\left[\check \E\left(1+\abs{\A(\check u_\eps^{(0)})}^{n\fp}+\sup_{t\in[0,T]}\norm{\D_{x}\check{u}_\eps(t)}_{\Ltwo}^{n\fp}\right)\right]^{\frac{\fpp}{2\fp}}\\
&\le& C\left[\check \E\left(1+\abs{\A(\check u_\eps^{(0)})}^{(n+2)\fp}+\sup_{t\in[0,T]}\norm{\D_{x}\check{u}_\eps}_{\Ltwo}^{(n+2)\fp}\right)\right]^{\frac{\fpp}{2\fp}}\\
&& + C\left[\check \E\left(\int_{0}^{T}\intort\Fepsq(\check{u}_\eps) \, (\D_{x}^{3}\check{u}_\eps)^{2} \, \diff x \, \diff t\right)^{\fp}\right]^{\frac{\fpp}{2\fp}}\\
 &\stackrel{\eqref{eq:energy-entropy}}{\le}&
 C\left[1+\E\abs{\Aunull}^{2(n+2)\fp q}+\E\norm{G_{0}(u^{(0)})}_{L^{1}(\bT)}^{(n+2)\fp q}+\E\norm{\D_{x}u^{(0)}}_{L^{2}(\bT)}^{(n+2)\fp}\right]^{\frac{\fpp}{2\fp}}
\end{eqnarray*}
which gives the desired estimate \eqref{eq:reg-time-2}. Note that we have used Proposition~\ref{prop:weak-eps} to get an estimate in terms of the initial data. 
\end{proof}

By  interpolation, we get the following result on H\"older regularity with respect to space and time (for details, see \cite[Lemma 4.11, p.\ 437]{FischerGruen2018}). Note that in Corollary~\ref{cor:hoelder}, we control the moment of order $p'$ of $\norm{\check{u}_\eps}_{{C^{\frac 1 8, \frac 1 2}}(Q_T)}$ for any $p'\in[1,2p)$ while in \cite{FischerGruen2018} only estimates for second moments have been provided. This is due to \eqref{eq:reg-time-1} as the analogous estimate in \cite{FischerGruen2018} has been formulated only for $p'=2.$

\begin{corollary}[H\"older-continuity]\label{cor:hoelder}
Under the assumptions of Lemma~\ref{lem:reg-time}, the solutions $\check{u}_\eps$ constructed in Proposition~\ref{prop:weak-eps} are space-time H\"older-continuous, $\diff\check \P$-almost surely. In particular, there is a finite constant $C$ independent of $\eps>0$ such that 
  \begin{align}
    \label{eq:greg-10}
    \check\E\left[\norm{\check{u}_\eps}^{p'}_{{C^{\frac 1 8, \frac 1 2}}(Q_T)}\right]\leq C
  \end{align}
  for any $p'\in [1,2p)$.
\end{corollary}

In the next proposition we will consider the space of all functions $u : Q_T \to \R$ such that $u \in  {C^{\frac \gamma 4,\gamma}}(Q_T)$ for all $\gamma \in (0, 1/2)$, that is, 
$$
{C^{\frac{1}{8} -, \frac{1}{2} -}}(Q_T):=\bigcap\limits_{\gamma \in (0, 1/2)} {C^{\frac \gamma 4,\gamma}}(Q_T).
$$
We endow   ${C^{\frac{1}{8} -, \frac{1}{2} -}}(Q_T)$ with the  topology generated by the metric  
$$
d(u, v) := \sum_{n=1}^\infty 2^{-n}  \big( \|u -v \|_{C^{\frac{\gamma_n}{4}, \gamma_n,}(Q_T)} \wedge 1 \big),
$$
where $\gamma_n= \frac 1 2 - \frac{1}{2^{n+1}}$. 

\begin{remark}      \label{rem:Polish}
Despite the fact that for each $\gamma< \frac 1 2$, ${C^{\frac \gamma 4,\gamma}}(Q_T)$ is not separable, the space ${C^{\frac{1}{8} -, \frac{1}{2} -}}(Q_T)$ is separable: 
 If $u \in C^{ \frac{1}{8}-, \frac{1}{2}-}$ there exists $u_n \in C^\infty(Q_T)$ such that for all $\gamma < 1/2$
 $$
\lim_{n \to \infty}  \|u_n-u\|_{{C^{\frac \gamma 4,\gamma}}(Q_T)}=0.
 $$
Moreover, there exists a countable set $\mathcal{D} \subset C^\infty(Q_T)$ such that for all $v \in C^\infty(Q_T)$ and all $\eps>0$, there exists $\tilde{v} \in \mathcal{D}$ such that
$\| \tilde{v}- v \|_{C^1(Q_T)} \leq \eps$. It follows that $\mathcal{D}$ is dense in ${C^{\frac{1}{8} -, \frac{1}{2} -}}(Q_T)$.

In addition, it is complete,  since ${C^{\frac{\gamma_n}{4}, \gamma_n}(Q_T)} $ is complete for each $n$. Therefore it is a Polish space. 

 Finally,  since every bounded sequence in ${C^{\frac{1}{8}, \frac{1}{2}}}(Q_T)$ has a subsequence that converges in ${C^{\frac \gamma 4,\gamma}}(Q_T)$, for all $\gamma < \frac 1 2$, it follows that the embedding $ C^{\frac{1}{8}, \frac{1}{2}}(Q_T) \subset {C^{\frac{1}{8} -, \frac{1}{2} -}}(Q_T)$ is compact.
\end{remark}

\begin{proposition}[Point-wise convergence]\label{prop:point-eps}
Let $T \in (0,\infty)$, $n \in \left[\frac 8 3,4\right)$, $\eps \in (0,1]$, $p> 1$, $q > 1$ satisfying $q \ge \max\left\{\frac{1}{4-n},\frac{n-2}{2n-5}\right\}$. Suppose that 
\[
u^{(0)} \in L^{(n+2)p}\left(\Omega,\cF_0,\P;H^1(\bT)\right)
\]
such that $u^{(0)} \ge 0$, $\diff\P$-almost surely, $\E \abs{\Aunull}^{2(n+2)pq} < \infty$, and $\E \norm{G_0\left(u^{(0)}\right)}_{L^1(\bT)}^{(n+2)pq} < \infty$. Let $\check{u}_\eps$ be the weak solution constructed in Proposition~\ref{prop:weak-eps}. Further define $\check{J} _\eps := F_\eps(\check{u}_\eps) \, \D_x^3 \check{u}_\eps$ (pseudo-flux density) and $\check W_\eps  := \sumkz\sigma_k \check  \beta^k_\eps $. Then, up to taking a subsequence of $u_\eps$, on the probability space $([0,1], \mathcal{B}([0,1]),  \lambda_{[0,1]})$, there exist random variables $\tilde u, \tilde u_\eps \colon [0,1] \to  C^{\frac{1}{2}-, \frac{1}{8}- }(Q_T)$, $\tilde J, \tilde J_\eps \colon [0,1] \to L^2(Q_T)$, $\tilde W, \tilde W_\eps \colon [0,1] \to C\left([0,T];H^2(\bT)\right)$, with
\begin{equs}      \label{eq:same_distribution}
(\tilde u_\eps, \tilde J_\eps, \tilde W_\eps) \sim (\check{u}_\eps, \check{J}_\eps, \check{W}_\eps ),
\end{equs}
such that
\begin{subequations}
\begin{align}
\tilde u_\eps(\omega) &\to \tilde u(\omega) \quad \mbox{as} \quad \eps \searrow 0 \quad \mbox{in}  \quad {C^{\frac{1}{8}-, \frac{1}{2}- }}(Q_T), \label{conv-u-eps-0} \\
\tilde J_\eps(\omega) &\rightharpoonup \tilde J(\omega) \quad \mbox{as} \quad \eps \searrow 0 \quad \mbox{in} \quad L^2\left(Q_T\right), \label{conv-j-eps-0} \\
\tilde W_\eps(\omega) &\to \tilde W(\omega) \quad \mbox{as} \quad \eps \searrow 0 \quad \mbox{in} \quad C\left([0,T];H^2(\bT)\right) \label{conv-w-eps-0}
\end{align}
\end{subequations}
for every $\omega \in [0,1]$. It holds
\begin{equation}\label{hoelder-point}
\tilde u\in L^{p'}(\tilde \Omega, \tilde\F, \tilde\P; C^{\gamma, \frac \gamma 4}(Q_T)),\quad \mbox{for all} \quad \gamma\in\left(0,\tfrac12\right) \quad \mbox{and} \quad p'\in \left[1,2p\right).
\end{equation}
\end{proposition}
\begin{proof}
It suffices to  show the tightness of the laws $\mu_{\check{u}_\eps}$, $\mu_{\check{J}_\eps}$, and $\mu_{\check W_\eps }$ corresponding to the families $\check{u}_\eps$, $\check{J}_\eps$, and  $\check{W}_\eps$. The proposition then follows by applying \cite[Theorem~2]{Jakubowski}.

Tightness of the law $\mu_{\check W}$ follows because $\mu_{\check W}$ is a Radon measure in the Polish space $C\left([0,T];H^2(\bT)\right)$ implying regularity from interior and thus tightness.

Tightness for $\mu_{\check{u}_\eps}$ as a family of measures on ${C^{\frac{1}{8}-, \frac{1}{2}- }}(Q_T)$ is a direct consequence of Corollary~\ref{cor:hoelder}, in particular estimate \eqref{eq:greg-10} (see also Remark \ref{rem:Polish}) . 

By Markov's inequality we have for any $R \in (0,\infty)$, using conservation of mass (cf.~Remark~\ref{rem:conservation-of-mass}) and the Sobolev embedding theorem,
\begin{eqnarray*}
\lefteqn{\check \P\left\{\norm{J_\eps}_{L^2(Q_T)} > R\right\}} \\
&\le& \frac{1}{R}\check  \E \norm{\check{J}_\eps}_{L^2(Q_T)} \\
&\le& \frac{C}{R} \, \check \E \left[1 + \norm{F_\eps(\check{u}_\eps) \D_x^3 \check{u}_\eps}_{L^2(Q_T)}^{n+2}\right] \\
&\stackrel{\eqref{eq:energy-entropy}}{\le}& \frac{C}{R} \, \E \left[1 + \abs{\Aunull}^{2(n+2)q} + \norm{G_0(u^{(0)})}_{L^1(\bT)}^{(n+2)q} + \norm{\D_x u^{(0)}}_{L^2(\bT)}^{n+2}\right] \to 0 \quad \mbox{as $R \to \infty$,}
\end{eqnarray*}
where we have used Proposition~\ref{prop:weak-eps}. 

Finally, \eqref{hoelder-point} holds by virtue of \eqref{conv-u-eps-0}  and Corollary \ref{cor:hoelder}. 
This finishes the proof. 
\end{proof}

For what follows, we define $\tilde\F = (\tilde \cF_t)_{t \in [0,T]}$ as the augmented filtration of $(\tilde \cF_t')_{t \in [0,T]}$, where
\begin{equation}\label{cal-f-t-prime}
\tilde \cF_t' := \sigma \big(\tilde u(t'),  \tilde W(t') \colon 0 \le t' \big ).
\end{equation}
We further define
\begin{equation}\label{beta-k-eps-beta-k}
\tilde \beta^k_\eps(t) := \frac{ \big(\sigma_k,\tilde W_\eps(t) \big)_{H^2(\bT)}}{ \norm{\sigma_k}_{H^2(\bT)}^2} \quad \mbox{and} \quad \tilde \beta^k(t) := \frac{\big( \sigma_k,\tilde W(t) \big) _{H^2(\bT)}}{ \norm{\sigma_k}_{H^2(\bT)}^2}.
\end{equation}
Then, we work in the filtered probability space
\[
\big(\tilde\Omega,\tilde\cF,\tilde\F,\tilde\P\big ) := \big([0,1],\cB([0,1]), (\tilde \cF_t )_{t \in [0,T]},\lambda_{[0,1]}\big).
\]
\begin{lemma}     \label{lem:bm}
The $\tilde \beta^k$ are mutually independent, standard, real-valued $\big(\tilde\cF_t\big)$-Wiener processes.
\end{lemma}
\begin{proof}
The reasoning is quite standard and contained in detail for instance in \cite[Proposition~5.3]{GessGnann2020} or \cite[Proof of Proposition~5.4]{DareiotisGess} or \cite[Lemma~5.7]{FischerGruen2018}.
\end{proof}
\begin{proposition}[Weak convergence, a-priori estimate, non-negativity, continuity]\label{prop:conv-weak-eps}
Let $T \in (0,\infty)$, $n \in \left[\frac 8 3,4\right)$, $\eps \in (0,1]$, $p>1$, $q > 1$ satisfying $q \ge \max\left\{\frac{1}{4-n},\frac{n-2}{2n-5}\right\}$. Suppose that 
\[
u^{(0)} \in L^{(n+2)p}\left(\Omega,\cF_0,\P;H^1(\bT)\right)
\]
such that $u^{(0)} \ge 0$, $\diff\P$-almost surely,  $\E \abs{\Aunull}^{2(n+2)pq} < \infty$, and $\E \norm{G_0\left(u^{(0)}\right)}_{L^1(\bT)}^{(n+2)pq} < \infty$.
  With the notation of Proposition~\ref{prop:point-eps}, up to taking subsequences, it holds that
\begin{align*}
\partial_x \tilde u_\eps \overset{*}{\rightharpoonup} \partial_x \tilde u \quad & \mbox{in} \quad L^{(n+2)p}\big(\tilde\Omega,\tilde\cF,\tilde\P;L^\infty (0,T;L^2(\bT))\big),
 \\
\partial_x^2 \tilde u_\eps \overset{*}{\rightharpoonup} \partial_x^2 \tilde u \quad & \mbox{in} \quad L^{2(n+2)pq}\big(\tilde\Omega,\tilde\cF,\tilde\P;L^2(Q_T)\big), 
\\
\tilde J_\eps \overset{*}{\rightharpoonup} \tilde J \quad & \mbox{in} \quad L^{(n+2)p}\big(\tilde\Omega,\tilde\cF,\tilde\P;L^2(Q_T)\big)
\end{align*}
as $\eps \searrow 0$, and the energy-entropy estimate
\begin{align}
& \tilde\E \left[\sup_{t \in [0,T]} \norm{\D_x \tilde u(t)}_{L^2(\bT)}^{(n+2)p} + \sup_{t \in [0,T]} \norm{G_0\left(\tilde u(t)\right)}_{L^1(\bT)}^{(n+2)pq} + \norm{\D_x^2 \tilde u}_{L^2(Q_T)}^{2(n+2)pq}+\|\tilde J\|_{L^2(Q_T)}^{(n+2)p}\right] \nonumber \\
& \quad \le C \, \E \left[1 + \abs{\Aunull}^{2(n+2)pq} + \norm{G_0(u^{(0)})}_{L^1(\bT)}^{(n+2)pq} + \norm{\D_x u^{(0)}}_{L^2(\bT)}^{(n+2)p}\right], \label{eq:energy-entropy-lim}
\end{align}
is satisfied, with  a constant  $C < \infty$ depending only on  $p,q,\sigma=(\sigma_k)_{k\in\Z}, n,L,$ and $T$. It holds $\tilde u \ge 0$ and $\abs{\{\tilde u = 0\}} = 0$, $\diff\P$-almost surely. Furthermore, $\tilde u$ is a continous $H^1_\mathrm{w}(\bT)$-valued process.
\end{proposition}
\begin{proof}
From Proposition~\ref{prop:point-eps} we derive that, up to taking subsequences, we have $\tilde u_\eps(\omega) \to \tilde u(\omega)$ in $C^{\gamma, \frac \gamma 4}\left(Q_T\right)$ for $\gamma < \frac 1 2$, and $\tilde J_\eps(\omega) \rightharpoonup \tilde J(\omega)$ in $L^2\left(Q_T\right)$, $\diff\tilde\P$-almost surely. From \eqref{eq:energy-entropy} of Proposition~\ref{prop:weak-eps} we can conclude by compactness that, up to taking subsequences once more, we have
\begin{align*}
\partial_x \tilde u_\eps &\overset{*}{\rightharpoonup} \tilde u_1 \quad \mbox{in} \quad L^{(n+2)p}\big(\tilde\Omega,\tilde\cF,\tilde\P;L^\infty (0,T;L^2(\bT))\big), 
\\
\partial_x^2 \tilde u_\eps &\overset{*}{\rightharpoonup} \tilde u_2 \quad \mbox{in} \quad L^{2(n+2)pq}\big(\tilde\Omega,\tilde\cF,\tilde\P;L^2(Q_T)\big), 
\\
\tilde J_\eps &\overset{*}{\rightharpoonup} \tilde J_1 \quad \mbox{in} \quad L^{(n+2)p}\big(\tilde\Omega,\tilde\cF,\tilde\P;L^2(Q_T)\big)
\end{align*}
as $\eps \searrow 0$. From \eqref{eq:greg-10} of Corollary~\ref{cor:hoelder} we have that $\tilde\E \norm{\tilde u_\eps}_{C^{\gamma,\frac\gamma4}(Q_T)}^{p'}$ for $p' \in [1,2p)$ is uniformly bounded in $\eps$. Hence, we obtain with Vitali's convergence theorem $\tilde\E \norm{\tilde u_\eps - \tilde u}_{C^{\gamma,\frac\gamma4}(Q_T)} \to 0$ as $\eps \searrow 0$.
Thus, we have for $j \in \{1,2\}$ and $\tilde\phi \in L^\infty\big(\tilde\Omega,\tilde\cF,\tilde\P;C^2(Q_T)\big)$,
\[
\tilde\E \left\langle\partial_x^j \tilde u_\eps - \tilde u_j,\tilde\phi\right\rangle_{L^1\left(0,T;L^1(\bT)\right) \times L^\infty\left(0,T;L^\infty(\bT)\right)} \to 0 \quad \mbox{as} \quad \eps \searrow 0
\]
by weak-$*$-convergence while by the Cauchy-Schwarz inequality
\begin{eqnarray*}
\lefteqn{\abs{\tilde\E \left\langle\tilde u_\eps - \tilde u,\partial_x^j \tilde\phi\right\rangle_{L^1\left(0,T;L^1(\bT)\right) \times L^\infty\left(0,T;L^\infty(\bT)\right)}}} \\
&\le& \tilde\E \norm{\tilde u_\eps - \tilde u}_{C(Q_T)} \norm{\D_x^j \tilde\phi}_{L^\infty(\tilde\Omega\times Q_T)} \to 0 \quad \mbox{as} \quad \eps \searrow 0.
\end{eqnarray*}
This implies $\tilde u_j = \D_x^j \tilde u$, $\diff\tilde \P$-almost surely.

Since  $\tilde J_\eps(\omega) \rightharpoonup \tilde J(\omega)$ in $L^2\left(Q_T\right)$, $\diff\tilde\P$-almost surely, and $\tilde J_\eps \in L^{(n+2)p}\big(\tilde\Omega,\tilde\cF,\tilde\P;L^2(Q_T)\big)$ is uniformly bounded, we have%
\[
\big(\tilde J_\eps,\phi\big)_{L^2(Q_T)} \to \big(\tilde J,\phi\big)_{L^2(Q_T)}
\]
strongly in $L^{(n+2)p'}\big(\tilde\Omega,\tilde\cF,\tilde\P\big)$ for any $p'<p$ and $\phi\in L^2\left(Q_T\right)$. Since also $\big(\tilde J_\eps,\phi\big)_{L^2(Q_T)} \rightharpoonup \big(\tilde J_1,\phi\big)_{L^2(Q_T)}$ weakly in $L^{(n+2)p'}\big(\tilde\Omega,\tilde\cF,\tilde\P\big)$ we have $\tilde J_1 = \tilde J$.

Estimate~\eqref{eq:energy-entropy-lim} follows from \eqref{eq:energy-entropy} of Proposition~\ref{prop:weak-eps} by weak lower-semicontinuity of the appearing norms and Fatou's lemma.

The fact that $\tilde u \ge 0$ and $\abs{\{\tilde u = 0\}} = 0$, $\diff\tilde\P$-almost surely, is a consequence of $\tilde u \in C(Q_T)$, $\diff\tilde\P$-almost surely, $\tilde u(0,\cdot) \ge 0$, $\diff\tilde\P$-almost surely, and finiteness of $\tilde\E \sup_{t \in [0,T]} \norm{G_0\left(\tilde u\right)}_{L^1(\bT)}^{(n+2)pq}$ because of \eqref{eq:energy-entropy-lim}.

The fact that $\tilde u$ is an $H^1_\mathrm{w}(\bT)$-valued process follows at once from $\tilde u \in L^\infty\left(0,T;H^1(\bT)\right) \cap C^{\gamma, \frac\gamma 4}(Q_T)$, $\diff\tilde\P$-almost surely. 
\end{proof}

\begin{proposition}\label{prop:identify-flux} Assume the conditions of Proposition~\ref{prop:conv-weak-eps} and $u^{(0)} \in L^{2np}\left(\Omega,\cF_0,\P;H^1(\bT)\right)$. Then, the distributional derivative $\partial_{x}^{3}\tilde{u}$ fulfills $\partial_{x}^{3}\tilde{u}\in L_{\mathrm{loc}}^{2}\left(\left\{ \tilde{u}>0\right\} \right)$ and we can identify $\tilde{J}_{\eps}=F_{\eps}\left(\tilde{u}_{\eps}\right) (\partial_{x}^{3}\tilde{u}_{\eps})$ and $\tilde{J}=\ind_{\{\tilde{u}>0\}}\tilde{u}^{\frac n 2}(\partial_{x}^{3}\tilde{u})$. In particular, the following energy-dissipation estimate holds 
\begin{align}
& \tilde\E \Big[\|\ind_{\{\tilde{u}>0\}}\tilde{u}^{\frac n 2}(\partial_{x}^{3}\tilde{u})\|_{L^2(Q_T)}^{(n+2)p}\Big] \nonumber \\
& \quad \le C \, \E \left[1 + \norm{\D_x u^{(0)}}_{L^2(\bT)}^{(n+2)p} + \abs{\Aunull}^{2(n+2)pq} + \norm{G_0(u^{(0)})}_{L^1(\bT)}^{(n+2)pq}\right]. \label{eq:flux-dissipation-lim}
\end{align}
\end{proposition}
\begin{proof}
The proof follows the lines of the proof of \cite[Proposition~5.6]{GessGnann2020}, with the additional complication of taking care of the approximation $F_\varepsilon(r)$ of the square root of the mobility $F(r)=\abs{r}^{\frac n 2}$.

Since the laws of $(\check{J}_{\eps}, \check{u}_\eps)$ and $(\tilde{J}_{\eps}, \tilde{u}_\eps)$ coincide (cf.~Proposition~\ref{prop:point-eps}), it holds for $\phi\in C^{\infty}(Q_{T})$, 
\[
0=\E\abs{\inner{\check{J}_{\eps}-F_{\eps}(\check{u}_\eps)\,\D_{x}^{3}\check{u}_\eps,\phi}_{L^{2}(Q_{T})}}=\tilde{\E}\abs{\inner{\tilde{J}_{\eps}-F_{\eps}\left(\tilde{u}_{\eps}\right)\D_{x}^{3}\tilde{u}_{\eps},\phi}_{L^{2}(Q_{T})}},
\]
which implies $\tilde{J}_{\eps}=F_{\eps}\left(\tilde{u}_{\eps}\right)\D_{x}^{3}\tilde{u}_{\eps}$.

Because of estimate~\eqref{eq:energy-entropy} of Proposition~\ref{prop:weak-eps}, $\tilde{J}_{\eps}=F_{\eps}(\tilde{u}_{\eps})\,\D_{x}^{3}\tilde{u}_{\eps}$, and \eqref{eq:def-f-eps}, it holds for fixed $r>0$, 
\begin{eqnarray}
\lefteqn{\tilde{\E}\int_{0}^{T}\int_{\bT}(\D_{x}^{3}\tilde{u}_{\eps})^{2}\,\ind_{\left\{ \norm{\tilde{u}_{\eps}-\tilde{u}}_{L^{\infty}(Q_{T})}<\frac{r}{2}\right\} \cap\left\{ \tilde{u}>r\right\} }\,\diff x\,\diff t}\nonumber \\
 & \le & \frac{2^n}{r^n}\tilde{\E}\int_{0}^{T}\int_{\left\{ \tilde{u}_{\eps}(t)>\frac{r}{2}\right\} }F_{\eps}^2(\tilde{u}_{\eps})\,(\D_{x}^{3}\tilde{u}_{\eps})^{2}\,\diff x\,\diff t\le C(r,u_{0}),\label{bound_dr3_u_eps}
\end{eqnarray}
where $C(r,u_{0})<\infty$ is independent of $\eps$. Hence, by taking a subsequence again denoted by $\tilde{u}_{\eps}$, it holds that for some $\eta^r \in L^{2}( \tilde\Omega,\tilde\cF,\tilde{\P};L^{2}(Q_{T}))$,
\begin{equation}
\D_{x}^{3}\tilde{u}_{\eps}\,\ind_{\left\{ \norm{\tilde{u}_{\eps}-\tilde{u}}_{L^{\infty}(Q_{T})}<\frac{r}{2}\right\} \cap\left\{ \tilde{u}>r\right\} }\rightharpoonup\tilde{\eta}^r \,\ind_{\left\{ \tilde{u}>r\right\} }\quad\mbox{as}\quad\eps\searrow0\quad\mbox{in}\quad L^{2}(\tilde\Omega,\tilde\cF,\tilde{\P};L^{2}(Q_{T})).\label{weak_d3xu}
\end{equation}

We next show that $\tilde{\eta}^r=\D_{x}^{3}\tilde{u}$ on $\left\{ \tilde{u}> r\right\}$ for any $r > 0$, that is, for almost every $(\omega,t) \in \Omega \times [0,T]$ and all $\tilde\varphi \in C_\mathrm{c}^\infty\left(\{\tilde u(\omega,t) > r\}\right)$ we have
\[
\int_\bT \tilde\eta^r \, \tilde\varphi \, \diff x = - \int_\bT \tilde u \, \D_x^3\tilde\varphi \, \diff x.
\]
It is enough to show that for almost every $\omega \in \Omega$ and all $\tilde\varphi \in C_\mathrm{c}^\infty\left(\{\tilde u(\omega) > r\}\right)$ we have
\[
\int_0^T \int_\bT \tilde\eta^r \, \tilde\varphi \, \diff x \, \diff t = - \int_0^T \int_\bT \tilde u \, \D_x^3\tilde\varphi \, \diff x \, \diff t.
\]
For every $\omega \in \Omega$ and $N \in \N$ let $\tilde\chi_N \in C^\infty(\R^2)$ such that $\tilde\chi_N(t,x) = 0$ for $(t,x) \in \{\tilde u(\omega) > r\}^\mathrm{c}$ and
\[
\tilde\chi_N(t,x) = \begin{cases} 1 & \mbox{ if } \mathrm{dist}\left((t,x),\partial\{\tilde u(\omega) > r\}\right) \ge \frac 1N \\
0  & \mbox{ if } \mathrm{dist}\left((t,x),\partial\{\tilde u(\omega) > r\}\right) \le \frac{1}{N+1} \end{cases}
\]
for all $(t,x) \in \{\tilde u(\omega) > r\}$. Then it is enough to show that for almost every $\omega \in \Omega$, all $\tilde\varphi \in C_\mathrm{c}^\infty\left(\R^2\right)$, and all $N \in \N$ it holds
\[
\int_0^T \int_\bT \tilde\eta^r  \, \tilde\varphi \, \tilde\chi_N \, \diff x \, \diff t = - \int_0^T \int_\bT \tilde u \, \D_x^3 (\tilde\varphi \, \tilde\chi_N) \, \diff x \, \diff t.
\]
Observe that for every $\varphi\in C^\infty_c(\{\tilde u(\omega)>r\})$, we have $\mathrm{dist}(\mathrm{supp}\ \varphi, \partial\{\tilde u(\omega)>r\})>0.$ Hence, for $N$ sufficiently large, $\tilde\chi_N(t,x)\cdot\tilde\varphi=\tilde\varphi $ for $\tilde \varphi\in C^\infty_c(\{\tilde u(\omega)>r\}).$
The equality above in particular holds if for all $\tilde\varphi \in C_\mathrm{c}^\infty\left(\R^2\right)$, $N \in \N$, and $\tilde\theta \in L^\infty (\tilde \Omega,\tilde\F, \tilde\P)$ we have
\[
\tilde\E \int_0^T \int_\bT \tilde\eta^r \, \tilde\varphi \, \tilde\chi_N \, \tilde\theta \, \diff x \, \diff t = - \tilde\E \int_0^T \int_\bT \tilde u \, \D_x^3(\tilde\varphi \, \tilde\chi_N \, \tilde\theta) \, \diff x \, \diff t.
\]
Since $\tilde\varphi \, \tilde\chi_N \, \tilde\theta \in L^\infty(\tilde \Omega,\tilde\F, \tilde\P;C^3(Q_T))$, it suffices to prove that for $\tilde{\zeta}\in L^{\infty}(\tilde\Omega,\tilde\cF,\tilde{\P};C^{3}(Q_{T}))$ such that $\supp_{(t,x)\in Q_{T}}\tilde{\zeta}\Subset\left\{ \tilde{u}>r\right\}$, $\diff\tilde{\P}$-almost surely, we have
\begin{equs}            \label{eq:testing}
\tilde\E \int_0^T \int_\bT \tilde\eta^r \, \tilde\zeta \, \diff x \, \diff t = - \tilde\E \int_0^T \int_\bT \tilde u \, \D_x^3\tilde\zeta \, \diff x \, \diff t.
\end{equs}
Therefore, take $\tilde{\zeta} \in L^{\infty}(\tilde\Omega,\tilde\cF,\tilde{\P};C^{3}(Q_{T}))$ such that $\supp_{(t,x)\in Q_{T}}\tilde{\zeta}\Subset\left\{ \tilde{u}>r\right\}$, $\diff\tilde{\P}$-almost surely. Then, integration by parts gives 
\begin{equs}
\lefteqn{\tilde{\E}\int_{0}^{T}\int_{\bT}(\D_{x}^{3}\tilde{u}_{\eps})\,\ind_{\left\{ \norm{\tilde{u}_{\eps}-\tilde{u}}_{L^{\infty}(Q_{T})}<\frac{r}{2}\right\} \cap\left\{ \tilde{u}>r\right\} }\,\tilde{\zeta}\,\diff x\,\diff t}\\
 & = & \tilde{\E}\left[\ind_{\left\{ \norm{\tilde{u}_{\eps}-\tilde{u}}_{L^{\infty}(Q_{T})}<\frac{r}{2}\right\} }\int_{0}^{T}\int_{\bT}(\D_{x}^{3}\tilde{u}_{\eps}) \, \tilde{\zeta} \,\diff x\,\diff t\right]\\
 & = & -\tilde{\E}\int_{0}^{T}\int_{\bT}\tilde{u}_{\eps} \, \ind_{\left\{ \norm{\tilde{u}_{\eps}-\tilde{u}}_{L^{\infty}(Q_{T})}<\frac{r}{2}\right\}} \, \D_{x}^{3}\tilde{\zeta} \,\diff x\,\diff t\\
 & \to & -\tilde{\E}\int_{0}^{T}\int_{\bT}\tilde{u} \, \D_{x}^{3}\tilde{\zeta} \,\diff x\,\diff t\quad\mbox{as}\quad\eps\searrow0
 \label{eq:uD3zeta}
\end{equs}
for any $r>0$, where in the last line we have applied Vitali's convergence theorem. Indeed, by Proposition~\ref{prop:point-eps} it holds 
\[
\tilde{u}_{\eps} \, \ind_{\left\{ \norm{\tilde{u}_{\eps}-\tilde{u}}_{L^{\infty}(Q_{T})}<\frac{r}{2}\right\}} \,(\D_{x}^{3}\tilde{\zeta})\to\tilde{u}\, \D_{x}^{3}\tilde{\zeta} \quad \mbox{as}\quad\eps\searrow0,
\]
$\diff\tilde{\P}\otimes\diff t\otimes\diff x$-almost everywhere, and  for some $p'>1$ we have 
\begin{eqnarray*}
\lefteqn{\tilde{\E}\int_{0}^{T}\int_{\bT}\abs{\tilde{u}_{\eps} \, \ind_{\left\{ \norm{\tilde{u}_{\eps}-\tilde{u}}_{L^{\infty}(Q_{T})}<\frac{r}{2}\right\}} \, (\D_{x}^{3}\tilde{\zeta})}^{p'}\diff x\,\diff t}\\
 & \le & \|\D_{x}^{3}\tilde{\zeta}\|_{{L^{\infty}\left(\tilde\Omega,\tilde\cF,\tilde{\P};L^\infty(Q_{T})\right)}}^{p'} \left(\tilde{\E}\int_{0}^{T}\int_{\bT}(\tilde{u}_{\eps})^{p'} \, \ind_{\left\{ \norm{\tilde{u}_{\eps}-\tilde{u}}_{L^{\infty}(Q_{T})}<\frac{r}{2}\right\}} \, \diff x \, \diff t\right)\\
 & \le & C(u_0),
\end{eqnarray*}
for some constant $C(u_0)$ independent of $\varepsilon,r$, where we have used the $\eps$-uniform bound \eqref{eq:energy-entropy} of Proposition~\ref{prop:weak-eps}. Therefore, by \eqref{weak_d3xu} and \eqref{eq:uD3zeta}, we get \eqref{eq:testing}, which in turn implies that 
\begin{equs}    \label{eq:eta=D3u}
\tilde{\eta}^r=\D_{x}^{3}\tilde{u} 
\end{equs}
on $\left\{ \tilde{u}> r\right\}$ for every $r > 0$.

Now, take $\tilde{\phi}\in L^{\infty}(\tilde\Omega,\tilde\cF,\tilde{\P};L^{\infty}(Q_{T}))$, $r>0$, $\eps\in(0,1]$, and separate according to 
\begin{equation}
I(\eps):=\tilde{\E}\int_{0}^{T}\int_{\bT}\tilde{J}_{\eps}\,\tilde{\phi}\,\diff x\,\diff t=I_{1}(r,\eps)+I_{2}(r,\eps),\label{decomp_i1_i2}
\end{equation}
where because of $\tilde{u}\ge0$, $\diff\tilde{\P}$-almost surely (cf.~Proposition~\ref{prop:conv-weak-eps}), we may choose 
\begin{align*}
I_{1}(r,\eps) & :=\tilde{\E}\int_{0}^{T}\int_{\bT}\tilde{J}_{\eps}\,\ind_{\left\{ \norm{\tilde{u}_{\eps}-\tilde{u}}_{L^{\infty}(Q_{T})}<\frac{r}{2}\right\} \cap\left\{ \tilde{u}>r\right\} }\,\tilde{\phi}\,\diff x\,\diff t,\\
I_{2}(r,\eps) & :=\tilde{\E}\int_{0}^{T}\int_{\bT}\tilde{J}_{\eps}\,\ind_{\left\{ \norm{\tilde{u}_{\eps}-\tilde{u}}_{L^{\infty}(Q_{T})}\ge\frac{r}{2}\right\} \cup\left\{ 0<\tilde{u}\le r\right\} }\,\tilde{\phi}\,\diff x\,\diff t.
\end{align*}
Then, we separate according to 
\begin{equation}
I_{1}(r,\eps)-\tilde{\E}\int_{0}^{T}\int_{\bT}\tilde{u}^{\frac n 2}\,(\D_{x}^{3}\tilde{u})\,\ind_{\left\{ \tilde{u}>r\right\} }\,\diff x\,\diff t=I_{11}(r,\eps)+I_{12}(r,\eps),\label{separate_i1}
\end{equation}
where 
\begin{align*}
I_{11}(r,\eps) & :=\tilde{\E}\int_{0}^{T}\int_{\bT}\left(F_{\eps}\left(\tilde{u}_{\eps}\right)-\tilde{u}^{\frac n 2}\right)(\D_{x}^{3}\tilde{u}_{\eps})\,\ind_{\left\{ \norm{\tilde{u}_{\eps}-\tilde{u}}_{L^{\infty}(Q_{T})}<\frac{r}{2}\right\} \cap\left\{ \tilde{u}>r\right\} }\,\tilde{\phi}\,\diff x\,\diff t,\\
I_{12}(r,\eps) & :=\tilde{\E}\int_{0}^{T}\int_{\bT}\left(\D_{x}^{3}\tilde{u}_{\eps}-\D_{x}^{3}\tilde{u}\right)\ind_{\left\{ \norm{\tilde{u}_{\eps}-\tilde{u}}_{L^{\infty}(Q_{T})}<\frac{r}{2}\right\} \cap\left\{ \tilde{u}>r\right\} }\,\tilde{u}^{\frac n 2}\,\tilde{\phi}\,\diff x\,\diff t,
\end{align*}
where we have used $\tilde{J}_{\eps}=F_{\eps}\left(\tilde{u}_{\eps}\right)\D_{x}^{3}\tilde{u}_{\eps}$. For the first integral, we note that 
\begin{equation}\label{lim_i11}\begin{split}
I_{11}(r,\eps) & \le\left(\tilde{\E}\int_{0}^{T}\int_{\bT}\left(\frac{F_{\eps}\left(\tilde{u}_{\eps}\right)-\tilde{u}^{\frac n 2}}{F_{\eps}\left(\tilde{u}_{\eps}\right)}\right)^{2}\ind_{\left\{ \norm{\tilde{u}_{\eps}-\tilde{u}}_{L^{\infty}(Q_{T})}<\frac{r}{2}\right\} \cap\left\{ \tilde{u}>r\right\} }\,\diff x\,\diff t\right)^{\frac{1}{2}}\\
 & \times\left(\tilde{\E}\int_{0}^{T}\int_{\bT}F_{\eps}^{2}\left(\tilde{u}_{\eps}\right)(\D_{x}^{3}\tilde{u}_{\eps})^{2}\,\diff x\,\diff t\right)\norm{\tilde{\phi}}_{L^{\infty}\left(\tilde\Omega,\tilde\cF,\tilde{\P};L^{\infty}(Q_{T})\right)}\\
 &\stackrel{\eqref{eq:def-f-eps}, \eqref{eq:energy-entropy}}{\le} \frac{C}{F_{\eps}\left(\frac{r}{2}\right)}\left(\tilde{\E}\int_{0}^{T}\int_{\bT}\left(F_{\eps}\left(\tilde{u}_{\eps}\right)-\tilde{u}^{\frac n 2}\right)^{2}\,\diff x\,\diff t\right)^{\frac{1}{2}}\norm{\tilde{\phi}}_{L^{\infty}\left(\tilde\Omega,\tilde\cF,\tilde{\P};L^{\infty}(Q_{T})\right)},\\
 & \to0\quad\mbox{as}\quad\eps\searrow0,
\end{split}\end{equation}
where $C<\infty$ is independent of $r$, $\eps$, and $\tilde{\phi}$, and we used Vitali's convergence theorem and Proposition~\ref{prop:weak-eps}.

Because of 
\begin{eqnarray*}
\tilde{\E}\int_{0}^{T}\int_{\bT}\left(\tilde{u}^{\frac n 2}\,\tilde{\phi}\right)^{2}\diff x\,\diff t & \le & C \left(1 + \tilde{\E}\int_{0}^{T}\int_{\bT}\tilde{u}^{n+2} \, \diff x\,\diff t\right) \norm{\tilde{\phi}}_{L^{\infty}\left(\tilde\Omega,\tilde\cF,\tilde{\P};L^{\infty}(Q_{T})\right)}^{2} \\
&\stackrel{\eqref{eq:energy-entropy-lim}}{\le}& C(u^{(0)}) \norm{\tilde{\phi}}_{L^{\infty}\left(\tilde\Omega,\tilde\cF,\tilde{\P};L^{\infty}(Q_{T})\right)}^{2} < \infty,
\end{eqnarray*}
where the Sobolev embedding, mass conservation (Remark~\ref{rem:conservation-of-mass}), and Proposition~\ref{prop:conv-weak-eps} have been applied, we have $\tilde{u}^{\frac n 2} \, \tilde{\phi}\in L^{2}\left(\tilde\Omega,\tilde\cF,\tilde{\P};L^{2}(Q_{T})\right)$ and by the weak convergence stated in \eqref{weak_d3xu}, combined with \eqref{eq:eta=D3u},  it follows that  $I_{12}(r,\eps)\to0$ as $\eps\searrow0$, which in conjunction with \eqref{separate_i1} and \eqref{lim_i11} gives 
\begin{equation}
I_{1}(r,\eps)\to\tilde{\E}\int_{0}^{T}\int_{\bT}\tilde{u}^{\frac n 2} \, (\D_{x}^{3}\tilde{u})\,\ind_{\left\{ \tilde{u}>r\right\} }\,\tilde{\phi}\,\diff x\,\diff t\quad\mbox{as}\quad\eps\searrow0.\label{lim_i1}
\end{equation}

The integral $I_{2}(r,\eps)$ in \eqref{decomp_i1_i2} can be estimated as 
\begin{eqnarray*}
\abs{I_{2}(r,\eps)} & \le & C\left(\tilde{\E}\int_{0}^{T}\int_{\bT}F_{\eps}^{2}(\tilde{u}_{\eps})\left(\D_{x}^{3}\tilde{u}_{\eps}\right)^{2}\diff x\,\diff t\right)\norm{\tilde{\phi}}_{L^{\infty}\left(\tilde\Omega,\tilde\cF,\tilde{\P};L^{\infty}(Q_{T})\right)}\\
 &  & \times\left(\tilde{\E}\int_{0}^{T}\int_{\bT}
 \ind_{\left\{ \norm{\tilde{u}_{\eps}-\tilde{u}}_{L^{\infty}(Q_{T})}\ge\frac{r}{2}\right\} \cup\left\{ 0<\tilde{u}\le r\right\} }\,\diff x\,\diff t\right)^{\frac{1}{2}}\\
 & \stackrel{\eqref{eq:energy-entropy}}{\le} & C(u^{(0)}) \, \norm{\tilde{\phi}}_{L^{\infty}\left(\tilde\Omega,\tilde\cF,\tilde{\P};L^{\infty}(Q_{T})\right)}\\
 &  & \times\left(\tilde{\E}\int_{0}^{T}\int_{\bT}
 \ind_{\left\{ \norm{\tilde{u}_{\eps}-\tilde{u}}_{L^{\infty}(Q_{T})}\ge\frac{r}{2}\right\} \cup\left\{ 0<\tilde{u}\le r\right\} } \,\diff x\,\diff t\right)^{\frac{1}{2}}
\end{eqnarray*}
where $C, C(u^{(0)}) < \infty$ are independent of $r$, $\eps$, and $\tilde{\phi}$, and Proposition~\ref{prop:weak-eps}  has been used. Then, we note that 
\[
\ind_{\left\{ \norm{\tilde{u}_{\eps}-\tilde{u}}_{L^{\infty}(Q_{T})}\ge\frac{r}{2}\right\} \cup\left\{ 0<\tilde{u}\le r\right\} } \to
\ind_{\left\{ 0<\tilde{u}\le r\right\} }\quad\mbox{as}\quad\eps\searrow0,
\]
$\diff\tilde{\P}\otimes\diff t\otimes\diff x$-almost everywhere due to Proposition~\ref{prop:point-eps}. Therefore, by bounded convergence it follows
\begin{align*}
\limsup_{\eps\searrow0}I_{2}(r,\eps) & \le C\,\norm{\tilde{\phi}}_{L^{\infty}\left(\tilde\Omega,\tilde\cF,\tilde{\P};L^{\infty}(Q_{T})\right)}\times\left(\tilde{\E} \abs{\left\{0 < \tilde u \le r\right\}} \right)^{\frac{1}{2}},
\end{align*}
which in combination with \eqref{decomp_i1_i2} and \eqref{lim_i1} leads to 
\begin{align*}
& \limsup_{\eps\searrow0}\left|I(\eps)-\tilde{\E}\int_{0}^{T}\int_{\bT}\tilde{u}^{n}\,(\D_{x}^{3}\tilde{u})\,\ind_{\left\{ \tilde{u}>r\right\} }\,\tilde{\phi}\,\diff x\,\diff t\right| \nonumber \\
& \quad \le C \left(\tilde{\E} \abs{\left\{0 < \tilde u \le r\right\}}\right)^{\frac 1 2} \norm{\tilde{\phi}}_{L^{\infty}\left(\tilde\Omega,\tilde\cF,\tilde{\P};L^{\infty}(Q_{T})\right)},
\end{align*}
where $C < \infty$ is independent of $r$. In the limit $r \searrow 0$, we infer by monotone convergence $\limsup_{r \searrow 0} \tilde{\E} \abs{\left\{0 < \tilde u \le r\right\}} = 0$, which finishes the proof. 
\end{proof}
%

\subsection{Recovering the SPDE}\label{sec:proof_main_result}
In this section, we give the proof of the main result Theorem \ref{th:existence}:
\begin{proof}[Proof of Theorem~2.2]
We first note that the results of \S\ref{sec:compactness} can be applied due to the assumptions of Theorem~\ref{th:existence}, with $p$ in \S\ref{sec:compactness} replaced by $\frac{p}{n+2}$ from the statement of Theorem~\ref{th:existence}.

We will show that $\{ (\tilde \Omega,\tilde \cF,\tilde\F,\tilde \P), \ (\tilde\beta_k)_{k \in \Z},\ \tilde u(0),  \tilde u \}$ is a solution of \eqref{eq:exact-equation}. The fact that $\tilde{u}(0)$ has the same distribution as $u^{(0)}$ follows from Proposition~\ref{prop:point-eps} and \eqref{eq:same_distribution} therein. By Proposition~\ref{prop:conv-weak-eps}, $\tilde{u}$ is an $\tilde\F$-adapted continuous $H^1_\mathrm{w}(\bT)$-valued process, so that in particular $\tilde u(0)$ is $\tilde{\cF}_0$-measurable. The fact that the $\tilde{\beta}^k$ are independent real-valued standard $\tilde{\F}$-Wiener processes is the content of Lemma~\ref{lem:bm}. The H\"older regularity stated in \eqref{hoelder-main} is a consequence of \eqref{hoelder-point} of Proposition~\ref{prop:point-eps}. Moreover, \eqref{it:exact-solution-1} and \eqref{it:exact-solution-2} from Definition \ref{def:exact-solution} follow from \eqref{eq:energy-entropy-lim} and \eqref{eq:flux-dissipation-lim}. Hence, we only have to show \eqref{it:exact-solution-3}.

Denote by $\tilde u_\eps$ the sequence of Proposition \ref{prop:point-eps} and notice that by \eqref{eq:same_distribution} we have that $\tilde{u}_\eps$ satisfies \eqref{eq:approx-equation-2},  i.e., for $\varphi \in C^\infty(\bT)$ we have, $\diff\tilde\P$-almost surely
\begin{eqnarray}
\inner{\tilde u_\eps(t), \varphi}_{L^2(\bT)} &=& \inner{\tilde{u}_\eps(0), \varphi}_{L^2(\bT)} + \int_0^t \inner{F^2_\eps\left(\tilde u_\eps(t')\right) \D^3_x \tilde u_\eps(t'), \D_x\varphi}_{L^2(\bT)} \diff t' \nonumber  \\
&& - \frac{1}{2} \sumkz\int_0^t \inner{\sigma_k F_\eps'\left(\tilde u_\eps(t')\right) \D_x \left( \sigma_k F_\eps \left(\tilde u_\eps(t')\right) \right), \D_x \varphi}_{L^2(\bT)} \diff t' \nonumber \\
&& - \sumkz\int_0^t \inner{\sigma_k F_\eps \left(\tilde u_\eps(t')\right), \D_x\varphi}_{L^2(\bT)} \diff \tilde\beta^k_\eps(t'),
\label{eq:approx-equation-2-test}
\end{eqnarray}
for all $t \in [0,T]$.
We claim that for all  $t \in [0,T]$,

\begin{subequations}\label{lim-eps-0}
\begin{align}
& \inner{\tilde u_\eps(t), \varphi}_{L^2(\bT)} \to \inner{\tilde u(t), \varphi}_{L^2(\bT)},  \label{lim-eps-0-u} 
\end{align}
\begin{align}
& \int_0^t \intort F^2_\eps\left(\tilde u_\eps(t')\right) \left(\D^3_x \tilde u_\eps(t')\right) \D_x\varphi \, \diff x \, \diff t' \nonumber \\
& \quad \to \int_0^t \int_{\left\{\tilde u(t')>0\right\}} \left(\tilde u(t')\right)^n \left(\D^3_x \tilde u(t')\right) \D_x\varphi \, \diff x \, \diff t', \label{lim-eps-0-j}
\end{align}
and
\begin{align}
& \frac{1}{2} \sum_{k \in \Z} \int_0^t \inner{\sigma_k F_\eps'\left(\tilde u_\eps(t')\right) \D_x \left( \sigma_k F_\eps \left(\tilde u_\eps(t')\right) \right), \D_x \varphi}_{L^2(\bT)} \diff t' \nonumber \\
& \quad \to \frac{1}{2} \sum_{k \in \Z} \int_0^t \inner{\sigma_k F_0'\left(\tilde u(t')\right) \D_x \left( \sigma_k F_0\left(\tilde u(t')\right)\right), \D_x \varphi}_{L^2(\bT)} \diff t', \label{lim-eps-0-drift} 
\end{align}
\end{subequations}
as $\eps \to 0$, $\diff\tilde \P$-almost surely.

\paragraph*{Argument for \eqref{lim-eps-0-u}}
Since by Proposition~\ref{prop:point-eps} it holds $\norm{\tilde u_\eps - \tilde u}_{C(Q_T)} \to 0$ as $\eps \searrow 0$, $\diff\tilde\P$-almost surely, it follows
\[
\sup_{t \leq T} \abs{\inner{\tilde u_\eps(t), \varphi}_{L^2(\bT)} - \inner{\tilde u(t), \varphi}_{L^2(\bT)}} \le \norm{\tilde u_\eps - \tilde u}_{C(Q_T)} \norm{\varphi}_{L^1(\bT)} \to 0 \quad \mbox{as} \quad \eps \searrow 0,
\]
$\diff\tilde\P$-almost surely.

\paragraph*{Argument for \eqref{lim-eps-0-j}}
By Proposition~\ref{prop:identify-flux}, we can identify
\[
\int_0^t \intort F^2_\eps\left(\tilde u_\eps(t')\right) \left(\D^3_x \tilde u_\eps(t')\right) \D_x\varphi \, \diff x \, \diff t' = \int_0^t \int_{\bT} F_\eps\left(\tilde u_\eps(t')\right) \tilde J_\eps(t') \, \D_x\varphi \, \diff x \, \diff t'
\]
and
\[
\int_0^t \int_{\left\{\tilde u(t')>0\right\}} \left(\tilde u(t')\right)^n \left(\D^3_x \tilde u(t')\right) \D_x\varphi \, \diff x \, \diff t' = \int_0^t \int_{\bT} \left(\tilde u(t')\right)^{\frac n 2} \tilde J(t') \, \D_x\varphi \, \diff x \, \diff t',
\]
so that the limit in \eqref{lim-eps-0-j} follows from \eqref{eq:def-f-eps}, $\tilde u_\eps \to \tilde u$ in $C(Q_T)$, and $\tilde J_\eps \rightharpoonup \tilde J$ in $L^2(Q_T)$, $\diff\tilde\P$-almost surely.

\paragraph*{Argument for \eqref{lim-eps-0-drift}}
Applying the chain rule and integration by parts yields
\begin{equation*}\begin{split}
&\sum_{k \in \Z} \int_0^t \inner{\sigma_k F_\eps'\left(\tilde u_\eps(t')\right) \D_x \left( \sigma_k F_\eps \left(\tilde u_\eps(t')\right) \right), \D_x \varphi}_{L^2(\bT)} \diff t' \nonumber \\
&=-\sum_{k \in \Z} \int_0^t \inner{\int_1^{\tilde u_\eps(t')}(F'_\eps)^2(r) \, \diff r, \D_x(\sigma_k^2 \D_x \varphi)}_{L^2(\bT)} \diff t' \nonumber \\
& +\lefteqn{\frac{1}{2}\sum_{k \in \Z} \int_0^t \inner{  F_\eps'\left(\tilde u_\eps(t')\right)   F_\eps \left(\tilde u_\eps(t')\right) ,  (\D_x\sigma_k^2)\D_x \varphi}_{L^2(\bT)} \diff t'.} 
\end{split}\end{equation*}
Since by Proposition~\ref{prop:point-eps} we have $\norm{\tilde u_\eps - \tilde u}_{C(Q_T)} \to 0$ as $\eps \searrow 0$, $\diff\tilde\P$-almost surely, and by reversing the application of the chain rule and integration by parts, we obtain
\begin{equation*}\begin{split}
&\sum_{k \in \Z} \int_0^t \inner{\sigma_k F_\eps'\left(\tilde u_\eps(t')\right) \D_x \left( \sigma_k F_\eps \left(\tilde u_\eps(t')\right) \right), \D_x \varphi}_{L^2(\bT)} \diff t' \nonumber \\
&\to {-\sum_{k \in \Z} \int_0^t \inner{\int_1^{\tilde u(t')}(F'_0)^2(r) \, \diff r, \D_x(\sigma_k^2 \D_x \varphi)}_{L^2(\bT)} \diff t'} \nonumber \\
& +{\frac{1}{2}\sum_{k \in \Z} \int_0^t \inner{  F_0'\left(\tilde u(t')\right)   F_0 \left(\tilde u(t')\right) ,  (\D_x\sigma_k^2)\D_x \varphi}_{L^2(\bT)} \diff t'} \nonumber \\
&={\sum_{k \in \Z} \int_0^t \inner{\sigma_k F_0'\left(\tilde u(t')\right) \D_x \left( \sigma_k F_0 \left(\tilde u(t')\right) \right), \D_x \varphi}_{L^2(\bT)} \diff t'} \nonumber,
\end{split}\end{equation*}
uniformly in $t \in [0,T]$, $\diff\tilde\P$-almost surely. We note that the application of the chain rule is justified due to the $\diff\tilde\P$-almost surely boundedness of $u_\eps$ and $\tilde u$ and the local Lipschitz continuity of the occurring nonlinearities. 

Hence, we have showed  \eqref{lim-eps-0}. It  follows that for all  $t \in [0,T]$, 
\begin{equs}       \label{eq:def_Meps}
M_{\eps,\varphi}(t):= - \sum_{k \in \Z} \int_0^t \inner{\sigma_k F_\eps \left(\tilde u_\eps(t')\right), \D_x\varphi}_{L^2(\bT)} \diff \beta^k_\eps(t')  \to M_
\varphi(t) 
\end{equs}
as $\eps \searrow 0 $, $\diff\tilde{\P}$-almost surely,  for some  process $M_\varphi(t)$. Since the limits in \eqref{lim-eps-0} are continuous processes, so is $M_\varphi$. By virtue of Proposition~\ref{prop:weak-eps}, we can choose $\kappa \in (1,2)$ such that  for any $ \eps \in (0,1)$
 \begin{eqnarray*}
\lefteqn{\tilde{\E} \sup_{t \in [0,T]} | M_{\eps, \varphi}(t)|^{2 \kappa}} \\
&\le& C \  \lefteqn{\tilde\E \left(\sum_{k \in \Z} \int_0^T \inner{\sigma_k F_\eps \left(\tilde u_\eps(t)\right), \D_x\varphi}_{L^2(\bT)}^2\diff t \right )^\kappa } \\
&\le& C \, T^\kappa \left(\sum_{k \in \Z} \norm{\sigma_k}_{L^2(\bT)}^2\right)^\kappa \tilde\E \sup_{t \in [0,T]} \norm{F_\eps \left(\tilde u(t)\right)}_{L^\infty(\bT)}^{2\kappa} \norm{\D_x\varphi}_{L^2(\bT)}^{2\kappa} \\
&\stackrel{\eqref{eq:reg-sigma-k}, \eqref{eq:def-f-eps}, \eqref{eq:energy-entropy}}{\le}& C \left(1 + \tilde\E \abs{\Atunull}^{\kappa n} + \tilde\E \sup_{t \in [0,T]} \norm{\D_x \tilde u(t)}_{L^2(\bT)}^{\kappa n}\right) \stackrel{\eqref{eq:energy-entropy}}{\le} C\left( u^{(0)}\right),
\end{eqnarray*}
where $C\left( u^{(0)}\right) < \infty$ is independent of $\eps$ and where Remark~\ref{rem:conservation-of-mass} (mass conservation), the Sobolev embedding theorem, and Poincar\'e's inequality have been applied. In particular, we have 
\begin{equation}      \label{eq:uniform-integrability}
\sup_{\eps \in (0,1)} \tilde{\E} \sup_{t \in [0,T]} | M_{\eps, \varphi}(t)|^{2 \kappa} + \sup_{\eps \in (0,1)} \tilde\E \left(\sum_{k \in \Z} \int_0^T \inner{\sigma_k F_\eps \left(\tilde u_\eps(t)\right), \D_x\varphi}_{L^2(\bT)}^2\diff t \right )^\kappa < \infty,
\end{equation}
which implies by Fatou's lemma that 
\begin{equs}
\tilde{\E}|M_\varphi(t)|^{2\kappa} < \infty.
\end{equs} 
In order to complete the proof, we only have to show that  
\begin{equation}\label{mart-m-phi}
M_\varphi(t) = - \sum_{k \in \Z} \int_0^t \inner{\sigma_k F_0 \left(\tilde u(t')\right), \D_x\varphi}_{L^2(\bT)} \diff \beta^k(t').
\end{equation}
For this,  it suffices by virtue of  \cite[Proposition~A.1]{HOF2} to verify that for  $0 \le t' \le t \le T$ and $k \in \mathbb{Z}$,  we have 
\begin{align}
\tilde\E\left[ M_\varphi(t) -  M_\varphi(t') \middle| \tilde\cF_{t'}\right] &= 0, \label{rep-mart-eps-0-1} \\
\tilde\E\left[\left( M_\varphi(t)\right)^2 - \left( M_\varphi(t')\right)^2 - \sum_{k \in \Z} \int_{t'}^t \inner{\sigma_k F_0 \left(\tilde u(t'',\cdot)\right), \D_x\varphi}_{L^2(\bT)}^2 \diff t'' \middle| \tilde\cF_{t'}\right] &= 0, \label{rep-mart-eps-0-2} \\
\tilde\E\left[\tilde\beta^k(t) \,  M_\varphi(t) - \tilde\beta^k(t') \, M_\varphi(t') +  \int_{t'}^t \inner{\sigma_k F_0 \left(\tilde u(t'',\cdot)\right), \D_x\varphi}_{L^2(\bT)} \diff t'' \middle| \tilde\cF_{t'}\right] &= 0. \label{rep-mart-eps-0-3}
\end{align}
Notice  that  $M_{\eps,\varphi}$ as defined in \eqref{eq:def_Meps} is a square-integrable $\tilde\cF_{\eps,t}'$-martingale, where $\tilde\cF_{\eps,t}' := \sigma\left(\tilde u_\eps(t'), \tilde W_\eps(t') \colon 0 \le t' \le t\right)$. Hence, we infer that for $0 \le t' \le t \le T$ and any
\[
\Phi \in C\left(C(Q_{t'}) \times C\left([0,t'];H^2(\bT)\right);[0,1]\right)
\]
it holds
\begin{subequations}\label{rep-mart-eps}
\begin{align*}
\tilde\E\left[\left(\tilde M_{\eps,\varphi}(t) - \tilde M_{\eps,\varphi}(t')\right) \tilde\Phi_\eps(t')\right] &= 0, \\
\tilde\E\left[\left(\left(\tilde M_{\eps,\varphi}(t)\right)^2 - \left(\tilde M_{\eps,\varphi}(t')\right)^2 - \sum_{k \in \Z} \int_{t'}^t \inner{\sigma_k F_\eps \left(\tilde u_\eps(t'',\cdot)\right), \D_x\varphi}_{L^2(\bT)}^2 \diff t''\right) \tilde\Phi_\eps(t')\right] &= 0,  \\
\tilde\E\left[\left(\tilde\beta^k_\eps(t) \, \tilde M_{\eps,\varphi}(t) - \tilde\beta^k_\eps(t') \, \tilde M_{\eps,\varphi}(t') +  \int_{t'}^t \inner{\sigma_k F_\eps \left(\tilde u_\eps(t'',\cdot)\right), \D_x\varphi}_{L^2(\bT)} \diff t''\right) \tilde\Phi_\eps(t')\right] &= 0, 
\end{align*}
\end{subequations}
where
\[
\tilde\Phi_\eps(t') := \Phi\left( \tilde u_\eps |_{[0,t'] },  \tilde W_\eps |_{[0,t'] }\right).
\]
By the convergence stated in \eqref{conv-u-eps-0} and \eqref{conv-w-eps-0} of Proposition~\ref{prop:point-eps} and \eqref{eq:def_Meps},  combined with the uniform integrability of all the terms  appearing in the expectations above,  which in turn follows from \eqref{eq:uniform-integrability}, we conclude that 

\begin{subequations}
\begin{align*}
\tilde\E\left[\left(M_{\varphi}(t) - M_{\varphi}(t')\right) \tilde\Phi(t')\right] &= 0,  \\
\tilde\E\left[\left(\left( M_{\varphi}(t)\right)^2 - \left( M_{\varphi}(t')\right)^2 - \sum_{k \in \Z} \int_{t'}^t \inner{\sigma_k F_0 \left( u(t'',\cdot)\right), \D_x\varphi}_{L^2(\bT)}^2 \diff t''\right) \tilde\Phi(t')\right] &= 0,  \\
\tilde\E\left[\left(\tilde \beta^k(t) \,  M_{\varphi}(t) - \tilde\beta^k(t') \,  M_{\varphi}(t') +  \int_{t'}^t \inner{\sigma_k F_0 \left(\tilde u(t'',\cdot)\right), \D_x\varphi}_{L^2(\bT)} \diff t''\right) \tilde\Phi(t')\right] &= 0, 
\end{align*}
\end{subequations}
where
\[
\tilde\Phi(t') := \Phi\left( \tilde u|_{[0,t'] },  \tilde W |_{[0,t'] }\right).
\]
Since $\Phi$ was arbitrary, we conclude  that \eqref{rep-mart-eps-0-1}, \eqref{rep-mart-eps-0-2}, and \eqref{rep-mart-eps-0-3} are valid with $\tilde \cF_{t'}$ replaced by $\tilde \cF_{t'}'$.   The passage from $\tilde \cF_{t'}'$ to $\tilde \cF_{t'}$ follows by a standard continuity argument employing Vitali's convergence theorem. This finishes the proof.
\end{proof}

\section*{Acknowledgements}
\renewcommand{\leftmark}{\textsc{Acknowledgements}}
\renewcommand{\rightmark}{\textsc{Acknowledgements}}
\addcontentsline{toc}{section}{Acknowledgements}
 
KD is grateful towards  Bielefeld University and TU Delft for their hospitality during the preparation of this manuscript. BG acknowledges support by the Max Planck Society through the Max Planck Research Group \emph{Stochastic partial differential equations} and by the Deutsche Forschungsgemeinschaft (DFG, German Research Foundation) through project B8 of the CRC 1283 \emph{Taming uncertainty and profiting from randomness and low regularity in analysis, stochastics and their applications}. MVG wishes to thank Bielefeld University and the Max Planck Institute for Mathematics in the Sciences in Leipzig for their kind hospitality. GG acknowledges support by Deutsche Forschungsgemeinschaft (DFG, German Research Foundation) through project \#397495103 entitled \emph{Free boundary propagation and noise: analysis and numerics of stochastic degenerate parabolic equations}.

\renewcommand{\leftmark}{\textsc{References}}
\renewcommand{\rightmark}{\textsc{References}}
\addcontentsline{toc}{section}{References}
\bibliography{ThinFilms}
\bibliographystyle{Martin}

\end{document}